\documentclass[11pt, reqno]{amsart}

\usepackage{times}

\usepackage{color}
\usepackage[all,arc]{xy}
\usepackage{enumerate}
\usepackage{mathrsfs}
\usepackage{upref}

\usepackage{amsmath}
\usepackage{amsthm}
\usepackage{amsfonts}
\usepackage{amssymb}
\usepackage{amsbsy}
\usepackage{graphicx}
\usepackage{enumitem}
\usepackage{physics}
\usepackage[breaklinks]{hyperref}
\usepackage{bbm}
\usepackage[mathscr]{euscript}

\usepackage{tikz}
\tikzstyle{startstop} = [rectangle, rounded corners, minimum width=2cm, minimum height=1cm,text centered, draw=black]
\tikzstyle{arrow} = [thick,->,>=stealth]

\newcommand{\ms}[1]{\mathscr{#1}}
\newcommand{\mc}[1]{\mathcal{#1}}
\newcommand{\ovl}[1]{\overline{#1}}

\newcommand{\varep}{ \varepsilon }

\newcommand{\sse} {\subseteq}

\newcommand{\Z}{\mathbb{Z}}
\newcommand{\R}{\mathbb{R}}
\newcommand{\Q}{\mathbb{Q}}
\newcommand{\C}{\mathbb{C}}
\newcommand{\E}{\mathbb{E}}

\newcommand{\icomplex}{\textup{i}}

\newcommand{\ra}{\rightarrow}
\newcommand{\toinf}{\ra \infty}

\newcommand{\beq}{\begin{equation}}
\newcommand{\eeq}{\end{equation}}
\newcommand{\mbf}[1]{\mathbf{#1}}
\newtheorem{theorem}{Theorem}
\newtheorem{prop}[theorem]{Proposition}
\newtheorem{lemma}[theorem]{Lemma}
\newtheorem{cor}[theorem]{Corollary}
\newtheorem{conjecture}[theorem]{Conjecture}
\theoremstyle{definition}
\newtheorem{definition}[theorem]{Definition}
\newtheorem{remark}[theorem]{Remark}

\newcommand{\p}{\mathbb{P}}

\numberwithin{equation}{section}
\numberwithin{theorem}{section}
\newcommand{\ind}{\mathbbm{1}}
\newcommand{\ptl}{\partial}
\DeclareMathOperator*{\esssup}{ess\,sup}

\newcommand{\torus}{\mathbb{T}}
\newcommand{\rade}{R\aa de}
\newcommand{\groupid}{\textup{id}}
\newcommand{\cutoff}{\Lambda}
\newcommand{\rconn}{\mathbf{A}}
\newcommand{\rgclass}{\mathbf{Q}}
\newcommand{\weakarrow}{\rightharpoonup}
\newcommand{\liegroup}{G}
\newcommand{\lalg}{\mathfrak{g}}
\newcommand{\elemlalg}{X}
\newcommand{\const}{C}
\newcommand{\timefn}{\tau}
\newcommand{\normbdfn}{\rho}

\newcommand{\connspace}{\mc{A}}
\newcommand{\gaugetransf}{\mc{G}}
\newcommand{\threetorus}{\torus^3}
\newcommand{\metricthreetorus}{d_{\threetorus}}
\newcommand{\gaction}[2]{#1^{#2}}
\newcommand{\sym}{S_{\textup{YM}}}
\newcommand{\orbitspace}{\connspace / \gaugetransf}
\newcommand{\projection}{\pi}
\newcommand{\cutoffspace}[1]{\Omega_{#1}}
\usepackage{xparse}
\DeclareDocumentCommand{\toporbitspace}{o}{\ms{T}_{q\IfValueT{#1}{, #1}}}
\DeclareDocumentCommand{\ctoporbitspace}{o}{\ms{T}_{c\IfValueT{#1}{, #1}}}
\DeclareDocumentCommand{\borelorbitspace}{o}{\mc{B}_{q\IfValueT{#1}{, #1}}}
\DeclareDocumentCommand{\cborelorbitspace}{o}{\mc{B}_{c\IfValueT{#1}{, #1}}}
\newcommand{\wloop}{\ell}
\newcommand{\character}{\chi}
\newcommand{\loopset}{\mathbb{L}}

\newcommand{\lalgdim}{d_\lalg}

\newcommand{\fouriercutoff}{M}

\newcommand{\regwl}[2]{W_{#1, #2}}
\newcommand{\regwlorbit}[3]{W_{#1, #2, #3}}

\newcommand{\vectorspace}{V}
\newcommand{\oneform}{$1$-form}
\newcommand{\oneforms}{$1$-forms}
\newcommand{\twoform}{$2$-form}
\newcommand{\zeroform}{$0$-form}
\newcommand{\zeroforms}{$0$-forms}

\newcommand{\chset}{\mbf{Ch}}
\newcommand{\indexset}{\mbf{I}}

\newcommand{\ymsemigroup}{\Phi}
\newcommand{\nonlineardistspace}{\mc{X}}
\newcommand{\topnonlineardistspace}{\ms{T}_{\nonlineardistspace^{1,2}}}
\newcommand{\borelnonlineardistspace}{\mc{B}_{\nonlineardistspace^{1,2}}}
\newcommand{\evalmap}{\mathbf{e}}
\newcommand{\rvnldistspace}{\mbf{X}}
\newcommand{\inclnldistpace}{\iota_{\nonlineardistspace_0^{1,2}}}
\DeclareDocumentCommand{\nucutoff}{o o o}{\nu_{\IfValueT{#2}{#2} \IfValueT{#3}{, #3}; #1}}
\DeclareDocumentCommand{\mucutoff}{o o o}{\mu_{\IfValueT{#2}{#2} \IfValueT{#3}{, #3}; #1}}
\newcommand{\obsnldistspace}[3]{[W_{\nonlineardistspace^{1,2}}]_{#1, #3, #2}}

\newcommand{\unitary}{\textup{U}}
\newcommand{\specialunitary}{\textup{SU}}
\newcommand{\orthogonal}{\textup{O}}
\newcommand{\specialorthogonal}{\textup{SO}}

\newcommand{\gt}{\sigma}
\newcommand{\metricspace}{d_{\torus^3}}
\newcommand{\dloop}{d_{\loopset}}
\newcommand{\curv}[1]{F_{#1}}

\newcommand{\honeconnspace}{\connspace^{1,2}}
\newcommand{\htwogaugetransf}{\gaugetransf^{2,2}}
\newcommand{\honeorbitspace}{\honeconnspace / \htwogaugetransf}
\newcommand{\gorbithone}[1]{[#1]^{1,2}}

\newcommand{\newtop}{\ms{T}}
\newcommand{\newborel}{\mc{B}}

\newcommand{\consttb}{{\const_1}}
\newcommand{\exptb}{{\beta}}
\newcommand{\constfourp}{{\const_2}}
\newcommand{\constgf}{{\const_3}}

\renewcommand{\hat}{\widehat}
\renewcommand{\tilde}{\widetilde}

\begin{document}

\title[A state space for 3D Yang--Mills]{A state space for 3D Euclidean Yang--Mills theories}
\author{Sky Cao}
\address{Department of Mathematics, Massachusetts Institute of Technology, Building 2, Cambridge, MA 02138}
\email{skycao@mit.edu}

\author{Sourav Chatterjee}
\address{Department of Statistics, Stanford University, Sequoia Hall, 390 Jane Stanford Way, Stanford, CA 94305}
\email{souravc@stanford.edu}
\thanks{S.~Cao was partially supported by NSF grant DMS RTG 1501767}
\thanks{S.~Chatterjee was partially supported by NSF grants DMS 1855484 and DMS 2113242}
\keywords{Euclidean Yang--Mills theory, Yang--Mills heat equation, Wilson loop} 
\subjclass[2020]{81T13, 82B28, 60B05}

\begin{abstract}
It is believed that Euclidean Yang--Mills theories behave like the massless Gaussian free field (GFF) at short distances. This makes it impossible to define the main observables for these theories --- the Wilson loop observables --- in dimensions greater than two, because line integrals of the GFF do not exist in such dimensions. Taking forward a proposal of Charalambous and Gross, this article shows that it is possible to define Euclidean Yang--Mills theories on the 3D unit torus as `random distributional gauge orbits', provided that they indeed behave like the GFF in a certain sense. One of the main technical tools is the existence of the Yang--Mills heat flow on the 3D torus starting from GFF-like initial data, which is established in a companion paper. A key consequence of this construction is that under the GFF assumption, one can define a notion of `regularized Wilson loop observables' for Euclidean Yang--Mills theories on the 3D unit torus.
\end{abstract}

\maketitle

\tableofcontents

\section{Introduction}\label{section:introduction}
Construction of Euclidean Yang--Mills theories is an important step towards the construction of quantum Yang--Mills theories, which are the building blocks of the Standard Model of quantum mechanics~\cite{GJ1987, jaffewitten}. In this article, we only deal with Euclidean Yang--Mills theories on the 3D unit torus $\torus^3 = \R^3/\Z^3$, and as such, we start with some definitions required to introduce these theories. 

\begin{remark}
In this paper, we use the term ``state space" in the sense of probability theory, i.e. a space on which a probability measure is defined, or in which random variables take values. This is in contrast to the usual notion of a ``state space" in quantum mechanics. 
\end{remark}

\subsection{Euclidean Yang--Mills theories}\label{section:euclidean-ym-theories}
Let $\liegroup$ be a compact Lie group and let $\lalg$ denote the Lie algebra of $\liegroup$. For a technical reason (see Remark \ref{remark:G-restriction}), we assume that $\liegroup$ is a finite product of groups from the following list: $\{\torus^n\}_{n \geq 1}$, $\{\unitary(n)\}_{n \geq 1}$, $\{\specialunitary(n)\}_{n \geq 1}$, $\{\orthogonal(n)\}_{n \geq 1}$, $\{\specialorthogonal(2n+1)\}_{n \geq 1}$. 

A {\it connection} on the trivial principal $G$-bundle $\torus^3 \times G$ is a function $A: \torus^3 \ra \lalg^3$, that is, a $3$-tuple of functions $A = (A_1, A_2, A_3)$, with $A_i : \torus^3 \ra \lalg$ for $i=1,2,3$. Note that $A$ can also be viewed as a $\lalg$-valued $1$-form on $\torus^3$. Let $\connspace$ denote the set of smooth connections. 

Let $\gaugetransf$ be the set of smooth maps from $\torus^3$ into $G$, considering $G$ as a submanifold of Euclidean space. Note that $\gaugetransf$ is a group under pointwise multiplication. We will refer to the elements of $\gaugetransf$ as smooth {\it gauge transformations}. Given $A\in \connspace$ and $\gt \in \gaugetransf$, define $A^\gt = (A_1^\gt, A_2^\gt, A_3^\gt)$ as
\beq\label{eq:gaugetransform}
A_i^\gt := \gt^{-1}A_i \gt + \gt^{-1} \partial_i\gt, 
\eeq
where $\partial_i \gt$ is the derivative of $\gt$ in coordinate direction $i$, considering $\gt$ as a map from $\torus^3$ into an ambient Euclidean space containing $G$. It is not hard to verify that $A^\gt$ is a smooth connection, and that the map $A\mapsto A^\gt$ defines an action of the group $\gaugetransf$ on the set $\connspace$. 

The {\it curvature} $\curv{A} = ((\curv{A})_{ij}, 1 \leq i, j \leq 3)$ of a smooth connection $A$ is defined~as
\begin{align}\label{eq:curvdef}
(\curv{A})_{ij} := \ptl_i A_j - \ptl_j A_i +[A_i, A_j], 
\end{align}
where $[A_i, A_j] = A_iA_j - A_j A_i$.  The {\it Yang--Mills action} of a smooth connection $A$ is defined~as
\beq\label{eq:sym-def} \sym(A) := -\sum_{1 \leq i, j \leq 3} \int_{\torus^3} \Tr((\curv{A})_{ij}(x)^2) dx. \eeq
Pure Euclidean Yang--Mills theory on $\torus^3$ with gauge group $G$ and coupling strength $g>0$ is the hypothetical probability measure 
\begin{align}\label{eq:ym-measure} 
d\mu(A) = Z^{-1} \exp\biggl(-\frac{1}{g^2}\sym(A)\biggr)d\lambda( A),
\end{align}
where $Z$ is the normalizing constant and $\lambda$ is Lebesgue (that is, translation invariant) measure on $\connspace$. 

A function $f$ on $\connspace$ is called {\it gauge invariant} if $f(A)=f(A^\gt)$ for any $A\in \connspace$ and $\gt\in \gaugetransf$. For example, it is not difficult to check that the Yang--Mills action $\sym$ is  gauge invariant. Physicists tell us that any physical observable must be gauge invariant (see e.g. \cite[Section 2.1.2]{TongGaugeTheory}). This suggests that the Yang--Mills measure \eqref{eq:ym-measure} should be defined on the quotient space $\orbitspace$, rather than on $\connspace$. The most important gauge invariant observables are the {\it Wilson loop observables}, defined below.

\subsection{Wilson loop observables}\label{section:wilson-loop-observables}
We define a {\it path} in $\torus^3$ to be any continuous function $\wloop : [0, 1] \ra \torus^3$. If moreover $\wloop(0) = \wloop(1)$, then we say that $\wloop$ is a loop. We say that a path $\wloop$ is $C^1$ if $\wloop'$ exists on $(0, 1)$ and is uniformly continuous. 
We say that a path $\wloop : [0,1] \ra \torus^3$ is piecewise $C^1$ if there exist points $0 = a_0 < a_1 < \cdots < a_k = 1$ such that for each $1 \leq i \leq k$, the restriction of $\wloop$ to $[a_{i-1}, a_i]$ is $C^1$. Let $\loopset$ denote the set of all piecewise $C^1$ loops in $\torus^3$. Note that two different elements of $\loopset$ may be reparametrizations of the same loop, but we will treat them as different loops because it will reduce some unnecessary complications.

In this paper, whenever we say that $\character$ is a character of $\liegroup$, we mean that there is a finite-dimensional representation $\rho$ of $\liegroup$ such that $\character = \Tr \rho$. Since $\liegroup$ is a compact Lie group, we can always assume that $\rho$ is a unitary representation, that is, $\rho : \liegroup \ra \unitary(m)$, where $m = \character(\groupid)$ (see, e.g., \cite[Section 8.1]{Huang1999}). 

Let $A$ be a smooth connection. Given a vector $v \in \R^3$ and $x \in \torus^3$, we denote $A(x) \cdot v := \sum_{i=1}^3 A_i(x) v_i \in \lalg$. Given a piecewise $C^1$ loop $\wloop\in \loopset$, consider a solution $h : [0,1] \ra \liegroup$ of the ODE
\beq\label{eq:wilson-loop-ode} 
h'(t) = h(t) A(\wloop(t)) \cdot \wloop'(t), ~~ h(0) = \groupid,
\eeq
where $\groupid$ is the identity element of $G$. 
From standard ODE theory, it follows that $h$ exists and is unique~\cite[Theorem 3.9]{Tes2012}. Moreover, it is not hard to see that each $h(t)$ is an element of $G$. The element $h(1)$ is called the {\it holonomy} of $A$ around the loop $\wloop$ (see \cite[Definition 3.7]{CG2015} or \cite[Section 1]{Che2019}). 

When $\liegroup=\unitary(1)$, each $A_j$ is a function from $\mathbb{T}^3$ into the imaginary axis. Hence, the parallel transport is simply the exponential of the line integral $\int_\ell \sum A_j dx_j$. When $\liegroup$ is non-Abelian, the $A_j$'s are non-Abelian matrix-valued functions. In this case, $h(1)$ is the {\it path ordered} exponential of the line integral $\int_\ell \sum A_j dx_j$, defined as the solution of the above ODE.

Now, given a character $\character$ of $\liegroup$, we define the function $\regwl{\wloop}{\character}:\connspace \to \C$ as
\begin{align*}
\regwl{\wloop}{\character}(A) := \character(h(1)). 
\end{align*}
This function $\regwl{\wloop}{\character}$ is called a Wilson loop observable. Just like the Yang--Mills action, Wilson loop observables are also gauge invariant, meaning that $\regwl{\wloop}{\character}(A) = \regwl{\wloop}{\character}(A^\gt)$ for any $\wloop$, $\character$, $A$, and $\gt$. This can be seen by noting that the function $h_\gt : [a, b] \ra \liegroup$ defined by $h_\gt(t) := \gt(\wloop(a))^{-1} h(t) \gt(\wloop(t))$ satisfies the ODE \eqref{eq:wilson-loop-ode} with $A$ replaced by $\gaction{A}{\gt}$, and then using that $\wloop(1) = \wloop(0)$ since $\wloop$ is a loop, along with the fact that characters are conjugacy invariant.

Thus, each $\regwl{\wloop}{\character}$ can be treated as a map from $\orbitspace$ into $\C$, and we will treat them as such. It is a fact that $\regwl{\wloop}{\character}$ is invariant under reparametrizations of the loop $\wloop$, but this will not be important for us.

We will need the following result by Sengupta~\cite{S1994}, which says that the collection of all Wilson loop observables separates  gauge orbits. Recall that $\liegroup$ is assumed to be a finite product of groups from the following list: $\{\torus^n\}_{n \geq 1}$, $\{\unitary(n)\}_{n \geq 1}$, $\{\specialunitary(n)\}_{n \geq 1}$, $\{\orthogonal(n)\}_{n \geq 1}$, $\{\specialorthogonal(2n+1)\}_{n \geq 1}$. This assumption is needed for this result. 

\begin{lemma}[Corollary of Theorems 1 and 2 of Sengupta~\cite{S1994}]\label{lmm:wilson-loops-determine-gauge-eq-class}
Suppose that for some $A_1,A_2\in \connspace$, we have that $\regwl{\wloop}{\character}(A_1) = \regwl{\wloop}{\character}(A_2)$ for all piecewise $C^1$ loops $\wloop$ and all characters $\character$ of $G$. Then there is a smooth gauge transformation $\gt$ such that $A_2 = \gaction{A_1}{\gt}$.
\end{lemma}

\begin{remark}\label{remark:G-restriction}
Lemma \ref{lmm:wilson-loops-determine-gauge-eq-class} is the only reason for the assumption on $\liegroup$. By considering a more general class of observables (i.e. those satisfying the assumptions of \cite[Proposition 2.1.2(b)]{S1992}), it is possible to remove these assumptions. Ultimately, we choose not to do so, since Wilson loop observables are the most commonly used observables in Yang--Mills, and the restriction on $\liegroup$ is satisfied by practically any $\liegroup$ that is relevant to Yang--Mills (in particular, $\specialunitary(2), \specialunitary(3)$).
\end{remark}

\subsection{The difficulty of defining Wilson loop observables}
For $d\ge 2$, the {\it massless Gaussian free field} (GFF) $\phi$ on the unit torus $\torus^d$ is a random distribution with the property that for any smooth $f$ and $g$, $\phi(f)$ and $\phi(g)$ are jointly Gaussian random variables, with mean zero and covariance 
\[
\int f(x) G(x,y) g(y)dx dy,
\] 
where $G$ is the Green's function on $\torus^d$. Taking $f = \delta_x$ and $g =\delta_y$, we get the formal expression $\E(\phi(x)\phi(y)) = G(x,y)$, even though $\phi$ is not defined pointwise. Note that the covariance blows up to infinity as $y\to x$. This is consistent with the fact that $\phi$ is infinity at any given point. See Section \ref{section:gauge-invariant-free-field} for more on the GFF.

It is believed that the random connections drawn from a Yang--Mills measure are not really random functions, but rather, random distributions. Furthermore, it is believed that these random distributions have similar behavior as the GFF. This has been verified rigorously in 2D (see e.g. \cite{Che2019}), but not in higher dimensions. 

So, assuming that Euclidean Yang--Mills theories behave like the GFF, let us consider the question of computing Wilson loop observables. Recall that Wilson loop observables are defined using line integrals. Let $\ell$ be a loop and suppose that we want to integrate a $d$-dimensional massless free field $\phi$ along $\ell$. The Green's function $G(x,y)$ blows up like $\log(1/|x-y|)$ as $y\to x$ when $d=2$, and like $|x-y|^{2-d}$ when $d\ge 3$. From this, it is not hard to see that the line integral of $\phi$ along $\ell$ is well-defined in 2D, but blows up in higher dimensions. Thus, we do not expect to be able to define Wilson loop observables directly in $d\ge 3$, and some indirect approach has to be taken. See \cite[Section 3.1]{Che2019} for a similar discussion.

\subsection{Yang--Mills heat flow}\label{ymheatsec}
Since Wilson loop observables are unlikely to be definable for Euclidean Yang--Mills theories in $d\ge 3$, one idea is that we should first make the theories  smooth by some kind of regularization (also called {\it renormalization}). The problem is that we need the regularization to be gauge covariant, meaning that if $R(A)$ is the regularized version of a distributional connection $A$, then $R(A^\gt)$ must equal $R(A)^\gt$. We need this so that the regularization is in fact acting on the {\it physical} space $\orbitspace$, rather than the unphysical space $\connspace$. A simple regularization, such as taking convolution with a smooth kernel, is not going to work. 

The works on Yang--Mills in the 80s focused on lattice approximations of Yang--Mills theories and smoothing using a gauge-covariant procedure known as phase cell renormalization (see Section \ref{reviewsec} for references). This program was not completed, and so the Yang--Mills existence problem is still open (see e.g. \cite[Section 6.5]{jaffewitten}). An idea that has emerged in the last fifteen years (see e.g. \cite{Luscher2010a, Luscher2010b, LW2011, NaNeu2006}) is that Yang--Mills theories can be regularized by the {\it Yang--Mills heat flow}, which is the  gradient flow of the Yang--Mills action. The Yang--Mills heat flow has played a central role in various areas of mathematics, starting with the paper by Atiyah and Bott \cite{AB1983}.  See \cite[Section 1]{Feehan2016} for a historical overview of this equation and its many applications in mathematics and physics, as well as for an encyclopedic account of existing results. See also \cite{CG2013, CG2015, CG2017, G2016, G2017, OT2017, Waldron2019} for some newer results.
In our setting, the Yang--Mills heat flow is the following PDE on time-dependent smooth connections on $\torus^3$:
\begin{align}\label{eq:ymdef}
\ptl_t A_i(t) = - \sum_{j=1}^3 \ptl_j (\curv{A(t)})_{ij} -  \sum_{j=1}^3 [A_j(t), (\curv{A(t)})_{ij}], 
\end{align}
holding for $1 \leq i \leq 3$ and $t\ge 0$. By a result of \rade~\cite{R1992}, this equation has a unique smooth solution for all $t > 0$ if the initial data $A(0)$ is smooth. Using this uniqueness, it is not hard to prove that the Yang--Mills heat flow is gauge covariant, meaning that if $\{B(t)\}_{t\ge 0}$ and $\{A(t)\}_{t\ge 0}$ are Yang--Mills heat flows with $B(0)=A(0)^\gt$ for some $\gt\in \gaugetransf$, then $B(t) = A(t)^\gt$ for every $t$. Thus, the Yang--Mills heat flow lifts to a flow on the quotient space $\orbitspace$. Let us denote this flow by $\{\Phi_t\}_{t\ge 0}$. That is, if $A\in \connspace$ and $[A]$ is the image of $A$ in $\orbitspace$, then $\Phi_t([A]) = [A(t)]$, where $\{A(t)\}_{t\ge 0}$ is the Yang--Mills heat flow with $A(0)=A$. It is not hard to see that $\{\Phi_t\}_{t\ge 0}$ is a semigroup of operators, that is, $\Phi_t \Phi_s = \Phi_{t+s}$ for all $t,s\ge 0$. We will refer to it as the Yang--Mills heat semigroup.

We will see later that the Yang--Mills heat flow indeed regularizes random distributional connections with GFF-like properties, so that regularized Wilson loop observables can be defined. This is the one of the main results of this paper (Theorem~\ref{thm:free-field-behavior-implies-tightness}). But before that, we need to define the notion of a distributional gauge orbit --- that is, an $\orbitspace$-valued distribution on $\torus^3$ --- because this is the space on which the Yang--Mills measure \eqref{eq:ym-measure} will be defined, under the assumption of GFF-like behavior. This is done in the next subsection.

\subsection{Distributional gauge orbits}
Recall that a distribution is a linear functional acting on some space of test functions. The action of a distribution on a  test function yields a real or complex number, which is interpreted as the integral of the product of the distribution and the test function. Thus, although a distribution cannot be defined pointwise, its average over a region is a well-defined real or complex number. This familiar concept of a distribution is difficult to generalize to the space $\orbitspace$ of gauge orbits. The reason is that there is no natural interpretation of $\orbitspace$ as a space of linear functionals on any space of test functions, which can be abstracted to define the space of distributional gauge orbits as a larger class of linear functionals. Note that this problem does not arise if we work with $\connspace$ instead of $\orbitspace$, but we do not want to do that because $\orbitspace$ is the physical space while $\connspace$ is unphysical, as mentioned earlier.

Therefore, to define distributional gauge orbits, we need to look at distributions from a different point of view. For a large class of distributions on, say, $\R^d$, one can define the ordinary heat flow whose initial data is a distribution rather than a function. More specifically, we often have that the solution of the heat equation $\partial_t \phi = \Delta \phi$ exists and is a smooth function for all $t>0$, even though $\phi(0,\cdot)$ is a distribution. In this case, we may identify the distribution $\phi(0, \cdot)$ with the flow $\{\phi(t,\cdot)\}_{t>0}$ of smooth functions. For example, this procedure identifies the Dirac delta at the origin in $\R^d$ with the collection of smooth functions $\{\phi(t, \cdot)\}_{t>0}$ given by $\phi(t,x) = (4\pi t)^{-d/2} e^{-|x|^2/4t}$. 

Conversely, we may take any smooth $\phi$ that solves the heat equation in $(0,\infty)\times \R^d$, and declare that it represents a distribution on $\R^d$. Given a smooth test function $f$ with compact support, the integral $\int \phi(t,x)f(x)dx$ would represent the action of this distribution on the function $g(t,\cdot)$, where $g$ is the solution of the heat equation with initial data $f$. It is not hard to see that this prescription is consistent with the previous paragraph.

The notion of distributions as solutions of the heat equation in $(0,\infty)\times \R^d$ can be generalized to the space of gauge orbits, as follows.  Let $\nonlineardistspace$ be the space of all solutions of the Yang--Mills heat equation \eqref{eq:ymdef} in $(0,\infty)\times \torus^3$, modulo gauge transforms. That is, an element $X\in \nonlineardistspace$ is a collection $\{X(t)\}_{t>0}$ of elements of $\orbitspace$, such that for each $t>s>0$, $X(t) = \Phi_{t-s}(X(s))$, where $\{\Phi_t\}_{t\ge 0}$ is the Yang--Mills heat semigroup defined in Subsection \ref{ymheatsec}. Note that, as above, we are leaving out $t=0$ to account for the possibility of roughness tending to infinity as $t\downarrow 0$. The space $\nonlineardistspace$ is our space of distributional gauge orbits. Note that $\orbitspace$ is naturally embedded in $\nonlineardistspace$ by the map that sends each element of $\orbitspace$ to the Yang--Mills heat flow starting from the element.

Having defined $\nonlineardistspace$, we need to define a topology on this space. This is tricky, because $\nonlineardistspace$ has no natural embedding in a vector space or a familiar metric space. We will work with the following topology, which provides a set of continuous functions that should suffice for most purposes. For each $\wloop\in \loopset$, each character $\character$ of $G$, and each $t>0$, define the map $\regwlorbit{\wloop}{\character}{t}: \nonlineardistspace\to \C$ as
\[
\regwlorbit{\wloop}{\character}{t}(X) := \regwl{\wloop}{\character}(X(t)). 
\]
We will refer to $\regwlorbit{\wloop}{\character}{t}$ as a {\it regularized Wilson loop observable}. We endow $\nonlineardistspace$ with the topology generated by the functionals $\regwlorbit{\wloop}{\character}{t}$, which we denote by $\newtop$, and the Borel $\sigma$-algebra generated by this topology, which we denote by $\newborel(\nonlineardistspace)$. One convenient feature of working with this $\sigma$-algebra is that we have the following result, which allows us to think of the regularized Wilson loops as the finite dimensional distributions of a $(\nonlineardistspace, \newborel(\nonlineardistspace))$-valued random variable. The proof is in Section \ref{section:completing-the-proof}.

\begin{lemma}\label{fdlemma}
Let $\mu$ and $\nu$ be two probability measures on $(\nonlineardistspace, \newborel(\nonlineardistspace))$ such that for any finite collection $\wloop_i$, $\character_i$, and $t_i$, $1 \leq i \leq k$, the pushforwards of $\mu$ and $\nu$ to $\C^k$ by the map $X \mapsto (W_{\wloop_i, \character_i, t_i}(X), 1 \leq i \leq k)$ are equal. Then $\mu=\nu$. In the language of random variables, this means that if $\mbf{X}$ and $\mbf{Y}$ are two $(\nonlineardistspace, \newborel(\nonlineardistspace))$-valued random variables such that for any finite collection $\wloop_i$, $\character_i$, and $t_i$, $1 \leq i \leq k$, we have that $(\regwlorbit{\wloop_i}{\character_i}{t_i}(\mbf{X}), 1 \leq i \leq k)$ has the same law as $(\regwlorbit{\wloop_i}{\character_i}{t_i}(\mbf{Y}), 1 \leq i \leq k)$, then $\mbf{X}$ and $\mbf{Y}$ have the same law.
\end{lemma}

By definition, $\regwlorbit{\wloop}{\character}{t}$ is a measurable map from $\nonlineardistspace$ into $\C$. But these are not the only measurable maps. The next lemma gives another class of measurable maps that we need in our main result. The proof is in Section \ref{section:completing-the-proof}. 
\begin{lemma}\label{symlmm}
For each $t$, the map $X\mapsto \sym(X(t))$ is measurable from $\nonlineardistspace$ into $\R$.
\end{lemma}

We are now ready to state the first main result of this paper in the next subsection.

\subsection{The tightness result}
The following theorem gives a criterion for tightness of any given sequence of probability measures $\{\mu_n\}_{n\ge 1}$ on the space $\nonlineardistspace$. Once tightness is established, Lemma~\ref{fdlemma} reduces the question of weak convergence to showing that the pushforwards of $\mu_n$ by $\regwlorbit{\wloop}{\character}{t}$ converge weakly for each $\wloop$, $\character$, $t$.

\begin{theorem}\label{mainthm}
Let $\{\rvnldistspace_n\}_{n \geq 1}$ be a sequence of $\nonlineardistspace$-valued random variables. Suppose that for all $t > 0$, $\{\sym(\rvnldistspace_n(t))\}_{n \geq 1}$ is a tight sequence of real-valued random variables. Then there exists a subsequence $\{\rvnldistspace_{n_k}\}_{k \geq 1}$, and a probability measure $\mu$ on $(\nonlineardistspace, \newborel(\nonlineardistspace))$, such that the laws of $\rvnldistspace_{n_k}$ converge weakly to $\mu$, in the sense that for all bounded continuous functions $f : (\nonlineardistspace, \newtop) \ra \R$, we have that $\E f(\rvnldistspace_{n_k}) \ra  \int_{\nonlineardistspace} f d\mu$ as $k\to \infty$.
\end{theorem}

\begin{remark}
For some intuition as to why $\sym(\rvnldistspace_n(t))$ may be well-behaved, recall that the Yang--Mills heat flow is the gradient flow of the Yang--Mills action $\sym$, so that the Yang--Mills heat flow tries to decrease $\sym$ as fast as possible.
\end{remark}


Theorem \ref{mainthm} is proved in Section \ref{mainproof}.~The proof uses a combination of Uhlenbeck's compactness theorem, various properties of the Yang--Mills heat flow, and Sobolev embedding. We close off this subsection with the following conjecture.

\begin{conjecture}\label{conjecture:tightness}
Let $\liegroup$ be a Lie group which satisfies the condition given at the beginning of Section \ref{section:euclidean-ym-theories}. We conjecture that for some suitable sequence of $\nonlineardistspace$-valued random variables $\{\rvnldistspace_n\}_{n \geq 1}$, whose laws approximate the Yang--Mills measure, the tightness condition of Theorem \ref{mainthm} holds.
\end{conjecture}

\subsection{Tightness for fields with free field-like short distance behavior}\label{section:free-field-behavior-implies-tightness}

We now give some evidence for Conjecture \ref{conjecture:tightness}. The starting observation is that Yang--Mills theories should behave like the free field at short distances. For instance, the phenomenon of {\it asymptotic freedom} (see the original papers \cite{GW1973, Pol1973}, as well as the overview \cite{Bethke2007}) says that the particles described by these theories behave like free particles at very short distances. Also, it is mentioned in \cite[Section 3.1]{Che2019} that the 3D Yang--Mills theory is expected to have the same regularity as the 3D GFF. This leads to the second main result of this paper (see Theorem \ref{thm:free-field-behavior-implies-tightness} below), which says that if the aforementioned behavior holds in a certain sense, then the tightness condition of Theorem \ref{mainthm} is satisfied.

First, we set some notation. For integers $n \geq 1$, let $[n] := \{1, \ldots, n\}$. For vectors $v \in \R^n$ for some $n$, we write $|v|$ for the Euclidean norm of $v$, and we write $|v|_\infty := \max_{1 \leq i \leq n} |v_i|$ for the $\ell^\infty$ norm of~$v$. Let $\{e_n\}_{n \in \Z^3}$ be the Fourier basis on $\torus^3$. Explicitly, if we identify functions on $\threetorus$ with 1-periodic functions on $\R^3$, then $e_n(x) = e^{\icomplex 2\pi n \cdot x}$, where $\icomplex = \sqrt{-1}$. 

The metric on $\threetorus$ (that is, the metric induced by the standard Euclidean metric of $\R^3$) will be denoted by $\metricspace$. Explicitly, if $\Pi : \R^3 \ra \threetorus$ is the canonical projection map, then $\metricspace(x, y) := \inf \{|x_0 - y_0| : \Pi(x_0) = x, \Pi(y_0) = y\}$. Here $|x_0 - y_0|$ is the Euclidean distance between $x_0, y_0 \in \R^3$.

Next, since we assumed that $\liegroup \sse \unitary(N)$, the Lie algebra $\lalg$ is a real finite-dimensional Hilbert space, with inner product given by $\langle X_1, X_2 \rangle = \Tr(X_1^* X_2)$. Let $\lalgdim$ be the dimension of $\lalg$, and fix an orthonormal basis $(X^a, a \in [\lalgdim])$ of $\lalg$. With this notation, we may view a ($\lalg$-valued) {\oneform} $A : \threetorus \ra \lalg^3$ as a collection $(A^a_j, a \in [\lalgdim], j \in [3])$ of functions $A^a_j : \threetorus \ra \R$, satisfying the relation
\beq\label{eq:one-form-R-valued-functions-relation} A_j = \sum_{a \in [\lalgdim]} A^a_j X^a, ~~ j \in [3].\eeq
Now let $\rconn = (\rconn(x), x \in \threetorus)$ be a $\lalg^3$-valued stochastic process with smooth sample paths. By \eqref{eq:one-form-R-valued-functions-relation}, we may express
\[ \rconn_j(x) = \sum_{a \in [\lalgdim]} \rconn^{a}_j(x) X^a, j \in [3],\]
and thus we obtain the $\R$-valued stochastic process $(\rconn^{a}_j(x), a \in [\lalgdim], j \in [3], x \in \threetorus)$, which also has smooth sample paths. We will view this $\R$-valued process as equivalent to the original $\lalg^3$-valued process.  Given a test function $\phi \in C^\infty(\threetorus, \R)$, let
\[ (\rconn^{a}_j, \phi) := \int_{\threetorus} \rconn^{a}_j(x) \phi(x) dx, ~~ a \in [\lalgdim], j \in [3].\]
We also similarly define $(\rconn^{a}_j, \phi)$ for complex test functions $\phi \in C^\infty(\threetorus, \C)$. For $n \in \Z^3$, we define the Fourier coefficient $\hat{\rconn}^a_j(n) := (\rconn^a_j, e_{-n})$ (recall that $e_{-n}(x) = e^{-\icomplex 2\pi n \cdot x}$). For $N \geq 0$, we define the Fourier truncation 
\[
\rconn_N = (\rconn_{N}(x), x \in \threetorus) = (\rconn^a_{N, j}(x), a \in [\lalgdim], j \in [3], x \in \threetorus),
\]
given by
\beq\label{eq:fourier-truncation} \rconn^a_{N, j}(x) := \sum_{\substack{n \in \Z^3 \\ |n|_\infty \leq N}} \hat{\rconn}^a_j(n) e_n(x), ~~ a \in [\lalgdim], j \in [3], x \in \threetorus.\eeq
Note that $\rconn_N$ is also a $\lalg^3$-valued stochastic process with smooth sample paths. Let $K = (K^{a_1 a_2}_{j_1 j_2}, a_1, a_2 \in [\lalgdim], j_1, j_2 \in [3])$ be a collection of smooth functions $K^{a_1 a_2}_{j_1 j_2} \in C^\infty(\threetorus \times \threetorus, \R)$. We define the quadratic form
\[ (\rconn, K \rconn) := \sum_{\substack{a_1, a_2 \in [\lalgdim] \\ j_1, j_2 \in [3]}} \int_{\threetorus} \int_{\threetorus} \rconn^{a_1}_{j_1}(x) K^{a_1 a_2}_{j_1 j_2}(x, y) \rconn^{a_2}_{j_2}(y) dx dy. \]
Note that using \eqref{eq:fourier-truncation}, we have that for any $N \geq 0$,
\beq\label{eq:quadratic-form-fourier-truncation} (\rconn_N, K \rconn_N) = \sum_{\substack{a_1, a_2 \in [\lalgdim] \\ j_1, j_2 \in [3]}} \sum_{\substack{n^1, n^2 \in \Z^3 \\ |n^1|_\infty, |n^2|_\infty \leq N}} \hat{\rconn}^{a_1}_{j_1}(n^1) \hat{\rconn}^{a_2}_{j_2}(n^2) (e_{n^1}, K^{a_1 a_2}_{j_1 j_2} e_{n^2}), \eeq
where
\[ (e_{n^1}, K^{a_1 a_2}_{j_1 j_2} e_{n^2}) := \int_{\threetorus} \int_{\threetorus} e_{n^1}(x) K^{a_1 a_2}_{j_1 j_2}(x, y) e_{n^2}(y) dx dy.\]
We can now state the assumptions needed for Theorem \ref{thm:free-field-behavior-implies-tightness}. Let $\{\rconn^n\}_{n \geq 1}$ be a sequence of $\lalg^3$-valued stochastic processes $\rconn^n = (\rconn^n(x), x \in \threetorus)$ with smooth sample paths. Think of this as a sequence of random smooth {\oneforms} whose laws approximate the Yang--Mills measure \eqref{eq:ym-measure}. Suppose that the following assumptions are satisfied for all $n \geq 1$.
\begin{itemize}
     \item {\it Translation invariance.} For any $a \in [\lalgdim], j \in [3]$, $x \in \threetorus$, we have that $\E[(\rconn^{n, a}_j(x))^2] < \infty$. Moreover, for any $a_1, a_2 \in [\lalgdim]$, $j_1, j_2 \in [3]$, $x, y \in \threetorus$, we have that
    \[\begin{split}
    \E[\rconn^{n, a_1}_{j_1}(x) \rconn^{n, a_2}_{j_2}(y)] &= \E[\rconn^{n, a_1}_{j_1}(x-y) \rconn^{n, a_2}_{j_2}(0)] \\
    &= \E[\rconn^{n, a_1}_{j_1}(0) \rconn^{n, a_2}_{j_2}(y-x)]. 
    \end{split}\]
    \item {\it $L^2$ regularity.} For all $\phi \in C^\infty(\threetorus, \R)$, $a \in [\lalgdim]$, $j \in [3]$, we have that $\E[(\rconn^{n, a}_j, \phi)^2] < \infty$. Moreover, we have that
    \[ \lim_{N \toinf} \E[|(\rconn^{n, a}_{N, j}, \phi) - (\rconn^{n, a}_j, \phi)|^2] = 0. \]
    Finally, we also assume that for any $a_1, a_2 \in [\lalgdim]$, $j_1, j_2 \in [3]$, $\phi_1, \phi_2 \in C^\infty(\threetorus, \R)$, 
    \[\begin{split} 
    \E[ (\rconn^{n, a_1}_{j_1}, \phi_1)&(\rconn^{n, a_2}_{j_2}, \phi_2)] = \\
    &\int_{\threetorus} \int_{\threetorus} \E[\rconn^{n, a_1}_{j_1}(x) \rconn^{n, a_2}_{j_2}(y)] \phi_1(x) \phi_2(y) dx dy. 
    \end{split}\]
    \item {\it Tail bounds.} There exist constants $\consttb > 0, \exptb > 0$ (which are independent of $n$), such that the following hold. For any $\phi \in C^\infty(\threetorus, \R)$, $a \in [\lalgdim], j \in [3]$, we have that for all $u \geq 0$,
    \[ \p(|(\rconn^{n, a}_j, \phi)| > u) \leq \consttb \exp\bigg(- \frac{1}{\consttb} \bigg(\frac{u}{(\E[(\rconn^{n, a}_j, \phi)^2])^{1/2}}\bigg)^{\exptb}\bigg). \]
    Additionally, let $K = (K^{a_1 a_2}_{j_1 j_2}, a_1, a_2 \in [\lalgdim], j_1, j_2 \in [3])$ be a collection of smooth functions $K^{a_1 a_2}_{j_1 j_2} \in C^\infty(\threetorus, \R)$, such that for all $a \in [\lalgdim]$, $j_1, j_2 \in [3]$, $x, y \in \threetorus$, we have that $K^{aa}_{j_1 j_2}(x, y) = 0$. For $N \geq 0$, let $\sigma_{n, N, K}^2 := \E [(\rconn^n_N, K \rconn^n_N)^2]$ (this is finite by recalling \eqref{eq:quadratic-form-fourier-truncation}, and using the previous finite second moment and tail bound assumptions to obtain that the Fourier coefficients have finite fourth moments: $\E[|\hat{\rconn}^{n, a}_j(n)|^4] < \infty$ for all $n, a, j$). Then we have that for all $u \geq 0$,
    \[ \p\big(|(\rconn^n_N, K \rconn^n_N)| > u\big) \leq \consttb \exp\bigg(-\frac{1}{\consttb} \bigg(\frac{u}{ \sigma_{n, N, K}}\bigg)^{\exptb}\bigg). \]
    \item {\it Approximate Wick's theorem.} There is some constant $\constfourp \geq 0$ (which is independent of $n$) such that the following holds. Let $a_1, a_2 \in [\lalgdim]$, $j_1, j_2 \in [3]$, $\phi_1, \phi_2, \phi_3, \phi_4 \in C^\infty(\threetorus, \C)$. Assume that $a_1 \neq a_2$. Let $Z_1 = (\rconn^{n, a_1}_{j_1}, \phi_1)$, $Z_2 = (\rconn^{n, a_2}_{j_2}, \phi_2)$, $Z_3 = (\rconn^{n, a_1}_{j_1}, \phi_3)$, $Z_4 = (\rconn^{n, a_2}_{j_2}, \phi_4)$. Then
    \[ |\E[Z_1 Z_2 \ovl{Z_3 Z_4}]| \leq \constfourp \big( |\E[Z_1 \ovl{Z_3}] \E[Z_2 \ovl{Z_4}]| + |\E[Z_1 \ovl{Z_4}] \E [Z_2 \ovl{Z_3}]| \big) .\]
(This condition is inspired by the fact that the four-point functions of a Gaussian process may be expressed in terms of its two-point functions.)
    \item {\it Local behavior like the GFF.} There exist constants $\constgf \geq 0$ and $\alpha \in (5/3, 3)$ (which are independent of $n$) such that for all $a_1, a_2 \in [\lalgdim]$, $j_1, j_2 \in [3]$, $x, y \in \threetorus$, we have that
    \[ |\E[\rconn^{n, a_1}_{j_1}(x) \rconn^{n, a_2}_{j_2}(y)]| \leq \constgf \metricspace(x, y)^{-(3-\alpha)}.\]
\end{itemize}

\begin{remark}\label{remark:assumptions-free-field-implies-tightness}
The key conceptual assumption is the last one, which assumes that the sequence $\{\rconn^n\}_{n \geq 1}$ has local behavior like the GFF, uniformly in $n$. To see how this relates to the GFF, take $\alpha = 2$. Then the assumption says that the covariance functions of $\{\rconn^n\}_{n \geq 1}$ can only be as singular as $\metricspace(x, y)^{-1}$. Note that this is exactly the singularity for the Green's function in 3D, and the Green's function in 3D is covariance function of the 3D GFF (see, e.g., \cite[Section 3]{WP2020}). Therefore this assumption says that the covariance functions of $\{\rconn^n\}_{n \geq 1}$ can be as singular as the covariance function of the GFF. Because we are thinking of $\{\rconn^n\}_{n \geq 1}$ as a sequence of approximations to the Yang--Mills measure \eqref{eq:ym-measure}, this assumption is thus a mathematical formulation of the statement that Yang--Mills theories should have the same short distance behavior as the free field. Actually, since we allow $\alpha \in (5/3, 3)$ (and in particular, $\alpha \in (5/3, 2]$), we allow for short distance behavior which is even more singular than the free field, which may be easier to show.
\end{remark}

\begin{remark}
The first assumption (of translation invariance) can always be ensured, because we can replace $(\rconn^n(x), x \in \threetorus)$ with $(\rconn^n(U + x), x \in \threetorus)$, where $U$ is a uniform random element of $\threetorus$ which is independent of $\rconn^n$. Note that the Yang--Mills theory \eqref{eq:ym-measure} is formally translation invariant, since the Yang--Mills action $\sym$ is translation invariant. The second to fourth assumptions are technical assumptions which are needed in the proof of Theorem \ref{thm:free-field-behavior-implies-tightness}; they essentially say that on a qualitative level, $\rconn^n$ is like a Gaussian field, in the sense that we have tail bounds and an approximate Wick's theorem (where in contrast to the usual Wick's theorem, we only assume a bound on the four-point correlations in terms of the two-point correlations, as opposed to exact equality). 
\end{remark}

Before we state Theorem \ref{thm:free-field-behavior-implies-tightness}, we need the following lemma. The proof is in Appendix \ref{section:measurability}.

\begin{lemma}\label{lemma:measurability-from-stochastic-process}
Let $\rconn = (\rconn(x), x \in \threetorus)$ be a $\lalg^3$-valued stochastic process with smooth sample paths. Define $\rvnldistspace \in \nonlineardistspace$ by $\rvnldistspace(t) := \Phi_t([\rconn])$ for $t > 0$. Then $\rvnldistspace$ is an $\nonlineardistspace$-valued random variable (i.e., it is measurable).
\end{lemma}

We can now state Theorem \ref{thm:free-field-behavior-implies-tightness}. The proof is in Section \ref{section:proofs-second-third}. The main technical details in the proof are in the companion paper \cite{CaoCh2021}. 

\begin{theorem}\label{thm:free-field-behavior-implies-tightness}
Let $\{\rconn^n\}_{n \geq 1}$ be a sequence of $\lalg^3$-valued stochastic processes $\rconn^n = (\rconn^n(x), x \in \threetorus)$ with smooth sample paths. Suppose that the previously listed assumptions are satisfied. For $n \geq 1$, define an $\nonlineardistspace$-valued random variable $\rvnldistspace_n$ by $\rvnldistspace_n(t) := \Phi_t([\rconn^n])$ for $t > 0$ (recall Lemma \ref{lemma:measurability-from-stochastic-process}). Then for all $t > 0$,
\beq\label{eq:action-bounded-expectation} \sup_{n \geq 1} \E[\sym(\rvnldistspace_n(t))] < \infty. \eeq
Consequently, by Theorem \ref{mainthm}, there is a subsequence $\{\rvnldistspace_{n_k}\}_{k \geq 1}$ and a probability measure $\mu$ on $(\nonlineardistspace, \mc{B}(\nonlineardistspace))$ such that the laws of $\rvnldistspace_{n_k}$ converge weakly to $\mu$.
\end{theorem}


We close off this subsection  with the following conjecture.

\begin{conjecture}
Let $\liegroup$ be a Lie group which satisfies the condition given at the beginning of Section \ref{section:euclidean-ym-theories}. We conjecture that for some suitable sequence $\{\rconn^n\}_{n \geq 1}$ of $\lalg^3$-valued stochastic processes with smooth sample paths, whose laws approximate the Yang--Mills measure, the previously listed assumptions are satisfied.
\end{conjecture}

\begin{remark}
As previously mentioned, the sequence $\{\rconn^n\}_{n \geq 1}$ may be obtained by, for example, Fourier truncation or lattice gauge theories. One small thing is that if we work with lattice gauge theories, then the obtained $\rconn^n$ may not necessarily be smooth. But we can always apply a smoothing procedure to obtain smoothness. See Remark \ref{remark:smoothing-procedure} for more.
\end{remark}

\begin{remark}[Comment on 4D]
One might naturally wonder whether we could also try to implement this approach in 4D. Unfortunately, there is a major barrier, which is that the basic ``probabilistic smoothing" phenomenon which is the key to the proof of Theorem \ref{thm:free-field-behavior-implies-tightness} (whose proof relies on \cite[Proposition 3.18]{CaoCh2021}) no longer holds. Stated in different terms, in 4D, the proof of local existence of the Yang--Mills heat flow with GFF initial data (proven in 3D in \cite{CaoCh2021, CCHS2022}) breaks down. This difficulty is ultimately due to the fact that 4D is critical for Yang-Mills. The SPDE approach of \cite{CCHS2020, CCHS2022} faces similar obstacles.
\end{remark}

\subsection{The gauge invariant free field}\label{section:gauge-invariant-free-field}

Since Theorem \ref{thm:free-field-behavior-implies-tightness} applies to fields which have GFF-like short distance behavior, we may as well apply it to the actual GFF. This gives a construction of an object which we call the ``gauge invariant free field", since it can be interpreted as an element $\rvnldistspace \in \nonlineardistspace$ with GFF initial data. This shows that $\nonlineardistspace$ is nontrivial, in that it indeed contains ``distributional" elements. We proceed to describe the construction of $\rvnldistspace$.

We first set notation for Fourier coefficients. Recall the Fourier basis $\{e_n\}_{n \in \Z^3}$. Given $f \in L^1(\threetorus, \R)$, define the Fourier coefficient
\[ \hat{f}(n) := \int_{\threetorus} f(x) \ovl{e_n(x)} dx \in \C, ~~ n \in \Z^3.  \]
Note (since $f$ is $\R$-valued) that $\hat{f}(-n) = \ovl{\hat{f}(n)}$ for all $n \in \Z^3$. Also, if $f \in C^\infty(\threetorus, \R)$, then for any $k \geq 1$,
\beq\label{eq:fourier-coefficients-rapid-decay-general-dim} \sup_{n \in \Z^3} (1 + |n|^2)^{k/2} |\hat{f}(n)| < \infty. \eeq
(See, e.g., \cite[Chapter 3, equation (1.7)]{T2011a}.)

Now, as in \cite[Section 1.1]{CaoCh2021}, we construct the 3D $\lalg^3$-valued GFF $\rconn^0$ as a $\lalg^3$-valued stochastic process $\rconn^0 = ((\rconn^0, \phi), \phi \in C^\infty(\threetorus, \R))$, as follows. Fix a subset $I_\infty \sse \Z^3$ such that $0 \notin I_\infty$, and such that for each $n \in \Z^3 \setminus \{0\}$, exactly one of $n, -n$ is in $I_\infty$. Let $(Z^a_j(n), a \in [\lalgdim], j \in [3], n \in I_\infty)$ be an i.i.d.~collection of standard complex Gaussian random variables. For $n \in \Z^3 \setminus \{0\}$, $n \notin I_\infty$, define $Z^a_j(n) := \ovl{Z^a_j(-n)}$. Then, for $\phi \in C^\infty(\threetorus, \R)$, define
\[\begin{split}
(\rconn^{0, a}_j, \phi) &:= \sum_{\substack{n \in \Z^3 \\ n \neq 0}} \frac{Z^a_j(n)}{|n|} \hat{\phi}(n) \in \R,  ~~ a \in [\lalgdim], j \in [3], \\
(\rconn^{0}_j, \phi) &:= \sum_{a \in [\lalgdim]} (\rconn^{0, a}_j, \phi) X^a \in \lalg, ~~ j \in [3], \\
(\rconn^0, \phi) &:= ((\rconn^0_j, \phi), j \in [3]) \in \lalg^3. 
\end{split}\]
(Using standard concentration arguments, we have $|Z^a_j(n)| = O(\sqrt{\log |n|})$ a.s. Combining this with \eqref{eq:fourier-coefficients-rapid-decay-general-dim}, we have that a.s., for all $\phi \in C^\infty(\threetorus, \R)$, the sum defining $(\rconn^{0, a}_j, \phi)$ is absolutely summable.)
For $N \geq 0$, we define the Fourier truncations $\rconn^0_N$ of $\rconn^0$ as follows. First, for $a \in [\lalgdim], j \in [3]$, define $\rconn^{0, a}_{N, j} = (\rconn^{0, a}_{N, j}(x), x \in \threetorus)$ by
\[ \rconn^{0, a}_{N, j}(x) := \sum_{\substack{n \in \Z^3 \\ |n|_\infty \leq N \\ n \neq 0}} \frac{Z^a_j(n)}{|n|} e_n(x), ~~ x \in \threetorus.\]
Then define $\rconn_N^0 = (\rconn_N^0(x), x \in \threetorus)$ by $\rconn_N^0(x) := (\rconn^0_{N, j}(x), j \in [3]) \in \lalg^3$, where $\rconn^0_{N, j}(x) := \sum_{a \in [\lalgdim]} \rconn^{0, a}_{N, j}(x) X^a$. Note that $\rconn_N^0$ is a $\lalg^3$-valued stochastic process with smooth sample paths. We have that $\rconn^0_N$ converges to $\rconn^0$ in various natural senses; for instance, a.s., for all $\phi \in C^\infty(\threetorus, \R)$, we have that $(\rconn^0_N, \phi) \ra (\rconn^0, \phi)$, where $(\rconn^0_N, \phi) := \int_{\threetorus} \rconn^0_N(x) \phi(x) dx$.

For $N \geq 0$, define $\rvnldistspace_N \in \nonlineardistspace$ by $\rvnldistspace_N(t) := \Phi_t([\rconn^0_N])$ for $t > 0$. By Lemma \ref{lemma:measurability-from-stochastic-process}, $\rvnldistspace_N$ is an $\nonlineardistspace$-valued random variable. The proof of the next theorem is in Section~\ref{section:proofs-second-third}. The main technical details in the proof are in the companion paper \cite{CaoCh2021}. 

\begin{theorem}\label{thm:gauge-invariant-free-field}
There exists an $\nonlineardistspace$-valued random variable $\rvnldistspace$ such that the laws of $\rvnldistspace_N$ converge weakly to the law of $\rvnldistspace$ as $N\to \infty$.
\end{theorem}

We call the random variable $\rvnldistspace$ the ``gauge invariant free field''. As previously noted, the initial data of $\rvnldistspace_N$ converges (in a certain sense) to the 3D $\lalg^3$-valued GFF $\rconn^0$, and thus we may think of $\rvnldistspace$ as a free field in the quotient space of distributional connections modulo gauge transforms; hence the name ``gauge invariant free field". 

\subsection{Related literature}\label{reviewsec}
Construction of Euclidean Yang--Mills theories in dimension two is now fairly well-understood, due to the contributions of many authors over the last forty years~~\cite{BFS1979, BFS1980, BFS1981, CCHS2020, Che2019, F1990, F1991, GKS1989, KK1987, L2003, L2010, S1992, S1993, S1997}. See \cite[Section 6]{Ch2018} for more references and discussion. 

Once we move to dimensions higher than two, there are fewer results. Gross~\cite{G1983} and Driver~\cite{Dri1987} proved the convergence of $\unitary(1)$ lattice gauge theory to the free electromagnetic field~\cite{G1975, Guerra} in dimensions three and four, respectively. Dimock \cite{Dimock2018, Dimock2020, Dimock2020b} proved ultraviolet stability for QED (which is $\unitary(1)$ theory with a fermion field) in 3D. It is mentioned in the recent paper of Chandra et al.~\cite{CCHS2020} on 2D YM that they have a forthcoming work which extends some of the results of their paper to 3D. Beyond this, the state of the art for three and four dimensions and a general (that is, non-Abelian) Lie group $\liegroup$ seems to be the phase cell renormalization results of Ba{\l}aban starting with \cite{Bal1983}, the phase cell renormalization results of Federbush starting with \cite{Fed1986}, and an alternative approach of Magnen, Rivasseau, and S\'{e}n\'{e}or \cite{MRS1993}. For more on phase cell renormalization as well as a more complete list of references, see \cite[Section 6]{Ch2018}. See also \cite[Chapter II.8]{Seiler1982} for an approach which views construction of Euclidean Yang--Mills theories in arbitrary dimensions as a problem of constructing the expectations of certain observables. 

However, as far as we can tell, none of the above papers actually construct non-Abelian Euclidean Yang--Mills theories in dimensions three and higher, in the sense that they do not construct probability measures on spaces of connections, or gauge orbits, or physically relevant observables, as was done by the various previously cited works in 2D. The absence of a suitable state space on which to define the Yang--Mills measure was the primary motivation for the works of Charalambous and Gross~\cite{CG2013, CG2015, CG2017} and Gross~\cite{G2016, G2017}, who proposed the construction of such state spaces in 3D using the Yang--Mills heat flow. This idea of using the Yang--Mills heat flow as a regularization procedure also appeared earlier in the physics literature, in work of Narayanan and Neuberger \cite{NaNeu2006} as well as work of L\"{u}scher \cite{Luscher2010a, Luscher2010b} and L\"{u}scher and Weisz \cite{LW2011}. A major obstruction in Charalambous and Gross's program, as it turned out, was that it was not known whether the Yang--Mills heat flow exists and is well-behaved if the initial data is a distributional connection. In this paper (combined with the companion paper \cite{CaoCh2021}), we show that this is indeed the case, if the initial distributional connection is the GFF, or ``GFF-like" (in terms of short-distance behavior).

Finally, we point out work of Chandra et al. \cite{CCHS2022}, which takes an SPDE approach towards construction of the 3D Yang--Mills measure. The state space that the authors construct in the paper is close related to our state space, in the sense that smoothing by the Yang--Mills heat flow is crucially used in both constructions.



\subsection{Organization of the paper}
The remainder of this paper is organized as follows. In Section \ref{section:3d-u(1)}, we show (for the purposes of illustration) how the 3D $\unitary(1)$ theory fits into our framework, by using Theorem~\ref{mainthm} and Lemma~\ref{fdlemma} to construct 3D $\unitary(1)$ Yang--Mills theory on $\nonlineardistspace$. In Section \ref{mainproof} (the bulk of this paper), we prove Theorem \ref{mainthm}, and along the way develop the technical details needed in the proof. In Section \ref{section:proofs-second-third}, we prove Theorems~\ref{thm:free-field-behavior-implies-tightness} and \ref{thm:gauge-invariant-free-field}; the main technical details behind the proofs of these theorems are in the companion paper \cite{CaoCh2021}.

\subsection{Acknowledgements}
We thank Nelia Charalambous, Persi Diaconis, Len Gross, and the anonymous referees for various helpful comments.

\section{Construction of 3D \texorpdfstring{$\unitary(1)$}{U(1)} theory}\label{section:3d-u(1)}
In this section, we show (for the purposes of illustration) how the 3D $\unitary(1)$ theory fits into our framework, by using Theorem~\ref{mainthm} and Lemma~\ref{fdlemma} to construct 3D $\unitary(1)$ Yang--Mills theory on $\nonlineardistspace$. Our construction is consistent with the previous construction of Gross (see Remark \ref{remark:u(1)-comparison-with-Gross}). Our framework additionally allows us to define regularized Wilson loop observables for the $\unitary(1)$ theory.

Yang--Mills theory with $G=\unitary(1)$ is the theory of electromagnetic fields. It turns out that in this case, the gauge orbit of a connection $A$ is determined by the curvature form $F_A$. Moreover, since the commutator in the definition \eqref{eq:curvdef} of $F_A$  vanishes for $\unitary(1)$ theory, the definition \eqref{eq:ym-measure} of the Yang--Mills measure indicates that $F_A$ is a Gaussian field. Indeed, this is how $\unitary(1)$ Euclidean Yang--Mills theory has been traditionally understood --- as a Gaussian field representing the curvature form, known as the {\it free electromagnetic field}~\cite{G1975, Guerra}. For non-Abelian theories, the curvature form does not determine the gauge orbit of a connection, nor is the curvature form expected to be a Gaussian field. So the traditional definition of $\unitary(1)$ theory does not generalize to non-Abelian theories. 

We begin with some preliminaries. Throughout this section, we assume that $\liegroup = \unitary(1)$. Note then that the Lie algebra is $\lalg = \icomplex \R$. We will need to use the Fourier transform for functions on $\threetorus$. Recall that $\{e_n\}_{n \in \Z^3}$ is the Fourier basis --- explicitly, if we identify functions on $\threetorus$ with 1-periodic functions on $\R^3$, then $e_n(x) = e^{\icomplex 2\pi n \cdot x}$, where $\icomplex = \sqrt{-1}$. Given a smooth {\oneform} $A$, define the Fourier coefficient
\[ \hat{A}(n) := \int_{\threetorus} A(x) \ovl{e_n(x)} dx, ~~ n \in \Z^3.\]
Note since $\lalg = \icomplex \R$, the Fourier coefficient $\hat{A}(n)$ is an element of $\C^3$. Thus let $|\hat{A}(n)|$ be the usual Euclidean norm. Since $A$ is smooth, we have the Fourier inversion formula
\beq\label{eq:fourier-inversion} A(x) = \sum_{n \in \Z^3} \hat{A}(n) e_n(x), ~~ x \in \threetorus, \eeq
where the sum on the right hand side above converges absolutely (see, e.g., \cite[Chapter 3, Section 1]{T2011a}). Note that since $\bar{z} = -z$ for $z \in \icomplex \R$, we have that $\ovl{A(x)} = - A(x)$ for all $x \in \torus^3$, and so upon equating Fourier coefficients, we have that $\hat{A}(-n) = -\ovl{\hat{A}(n)}$ for all $n \in \Z^3$.

Next, let $t \geq 0$. Given a smooth {\oneform} $A$, define the smooth {\oneform} $e^{t \Delta} A$ by
\beq\label{eq:heat-kernel-def} e^{t \Delta} A := \sum_{n \in \Z^3} e^{-4\pi^2 |n|^2 t} \hat{A}(n) e_n. \eeq
This definition ensures that $t \mapsto e^{t \Delta} A$ satisfies the heat equation with initial data $A$, i.e., $\ptl_t e^{t \Delta} A = \Delta e^{t \Delta} A = \sum_{i=1}^3 \ptl_{ii} e^{t \Delta} A$.

The proofs of the following two lemmas are in Appendix \ref{section:3d-miscellaneous-proofs}.

\begin{lemma}\label{lemma:da-l2-norm-formula}
Let $A$ be a smooth {\oneform}. Then
\[ \sym(A) = 8 \pi^2 \sum_{n \in \Z^3} (|n|^2 |\hat{A}(n)|^2 - |n \cdot \hat{A}(n)|^2). \]
Note that each term in the above sum is nonnegative by the Cauchy--Schwarz inequality.
\end{lemma}

\begin{lemma}\label{lemma:wilson-loop-heat-kernel-regularized}
Let $A$ be a smooth {\oneform}. Let $\wloop : [0, 1] \ra \threetorus$ be a piecewise $C^1$ loop, $\character$ be a character of $\unitary(1)$. For any $t \geq 0$, we have
\beq\label{eq:regwl-in-terms-of-fourier-coef} W_{\wloop, \character}(e^{t \Delta} A) = \character\Bigg( \exp\bigg(\sum_{n \in \Z^3} e^{-4\pi^2|n|^2 t} \hat{A}(n) \cdot \int_0^1 e_n(\wloop(s)) \wloop'(s) ds\bigg)\Bigg).\eeq
\end{lemma}

The last preliminary result we need is the following. Let $\connspace_0 := \{A \in \connspace : \hat{A}(0) = 0 \text{ and } n \cdot \hat{A}(n) = 0 \text{ for all $n \in \Z^3$}\}$. The following lemma shows that $\connspace_0$ is identified with the orbit space $\orbitspace$. (This is what's known as gauge-fixing; this particular gauge is known as the Coulomb gauge -- see e.g. \cite{Che2019}.) As we will see, this will allow us to sample the Fourier coefficients as independent complex Gaussians. (There is an issue whether formally, the ``Lebesgue measure" on $\connspace$ is mapped to ``Lebesgue measure" on $\connspace_0$. This is true here, since $\connspace_0$ is a linear subspace of $\connspace$. For non-Abelian Lie groups, this will not be the case, which necessitates the introduction of Faddeev--Popov ghosts --- see, e.g., \cite{Fadd2009} for more on this. At any rate, our construction for the $\unitary(1)$ theory will be consistent with the existing one in the literature --- see Remark \ref{remark:u(1)-comparison-with-Gross}.) 

\begin{lemma}\label{lemma:gauge-fixing-u(1)}
Every gauge orbit $[A]$ has a unique element $A_1 \in [A]$ such that $\hat{A}_1(0) = 0$ and such that $n \cdot \hat{A}_1(n) = 0$ for all $n \in \Z^3$. Consequently, the orbit space $\orbitspace$ is identified with~$\connspace_0$.
\end{lemma}

We now begin to construct the random variables $\rvnldistspace_n$ which appear in Theorem~\ref{mainthm}. The laws of these random variables should approximate (at least in a formal sense) the Yang-Mills measure \eqref{eq:ym-measure}. We next see how this is done.

For each $n \in \Z^3 \setminus \{0\}$, fix orthogonal vectors $u_n^1, u_n^2 \in \R^3$ such that $n \cdot u_n^k = 0$ and $|u_n^k| = 1$ for $k = 1, 2$. That is, we fix an orthonormal basis of the subspace of $\R^3$ which is orthogonal to $n$. We impose that $u_{-n}^k = u_n^k$ for all $n \in \Z^3 \setminus \{0\}$, $k = 1, 2$. 

Let $A \in \connspace_0$. Then $n \cdot \hat{A}(n) = 0$ for all $n \in \Z^3$. Thus for each $n \in \Z^3 \setminus \{0\}$, there exists unique $z_n^1, z_n^2 \in \C$ such that
\[ \hat{A}(n) = z_n^1 u_n^1 + z_n^2 u_n^2.\]
Also, recall that we have $\hat{A}(-n) = - \ovl{\hat{A}(n)}$ for all $n \in \Z^3$, which implies that $z_{-n}^k = - \ovl{z_n^k}$ for all $n \in \Z^3 - \{0\}$, $k = 1, 2$. We have by Lemma \ref{lemma:da-l2-norm-formula} (combined with the fact that $n \cdot \hat{A}(n) = 0$ for all $n \in \Z^3$) that
\[ \sym(A) = 8\pi^2 \sum_{n \in \Z^3}  |n|^2( |z_n^1|^2 + |z_n^2|^2).\]
This allows us to place a Gaussian density on the $z_n^k$, which we proceed to do next.

First, note that there is some redundancy, namely, $z_{-n}^k = - \ovl{z_n^k}$. To capture this redundancy, let $I_1 \sse I_2 \sse \cdots$ be a sequence of nested subsets of $\Z^3$ with the following property. For each $M \geq 1$, the set $I_M \sse \{n \in \Z^3 : |n|_\infty \leq M\}$ is such that $0 \in I_M$, and additionally for each $n \in \Z^3 \setminus \{0\}$ with $|n|_\infty \leq M$, exactly one of $n, -n$ is in $I_M$. Let $I_\infty := \bigcup_{M \geq 1} I_M$. We can then write
\[ \sym(A) = 16 \pi^2 \sum_{n \in I_\infty} |n|^2( |z_n^1|^2 + |z_n^2|^2), \]
and so
\[ \exp(- \frac{1}{g^2}S_{\textup{YM}}(A)) = \prod_{n \in I_\infty} \exp(- \frac{1}{g^2}16\pi^2 |n|^2 (|z_n^1|^2 + |z_n^2|^2)).\]
This suggests that we should sample the coefficients $z_n^k$, $n \in I_\infty$, $k = 1, 2$ independently from a normal distribution, with mean zero and the right variance. We do this as follows. Take a collection of independent random variables $(Z_n^k, n \in I_\infty, n \neq 0, k = 1, 2)$ such that $Z_n^k = X_n^k + i Y_n^k$, where for each $n \in I_\infty$, we have $X_n^1, Y_n^1, X_n^{2}, Y_n^{2} \stackrel{i.i.d.}{\sim} N(0, g^2 / (32 \pi^2 |n|^2))$. For $n \in \Z^3$, $n \notin I_\infty$, $k = 1, 2$, let $Z_{n}^k := - \ovl{Z_{-n}^k}$. For notational simplicity, let $Z_n := Z_n^1 u_n^1 + Z_n^2 u_n^2$ (note that this is consistent with the condition $Z_{-n} = - \ovl{Z_n}$). Note that by construction, $n \cdot Z_n = 0$ and $\E [|Z_n|^2] = g^2 O(|n|^{-2})$.

We next define the random variables $\rvnldistspace_n$ which appear in Theorem \ref{mainthm}. First, we state the following lemma which will be needed. The proof is in Appendix \ref{section:3d-miscellaneous-proofs}.

\begin{lemma}\label{lemma:in-nldist-space}
Let $A_0 \in \connspace_0$. Define the function $X : (0, \infty) \ra \orbitspace$ by $X(t) := [e^{t \Delta} A_0]$ for $t > 0$. Then $X \in \nonlineardistspace$.
\end{lemma}

\begin{definition}\label{def:u(1)-fourier-cutoff}
Let $\fouriercutoff \geq 1$. First, define the following $\connspace$-valued random variable:
\[ \rconn^0_M := \sum_{\substack{|n|_\infty \leq M \\ n \neq 0}} Z_n e_n. \]
Note that $\rconn^0_M \in \connspace_0$, since $n \cdot Z_n = 0$ for all $n \in \Z^3 \setminus \{0\}$.
Then, define the $\nonlineardistspace$-valued random variable $\rvnldistspace_M$ by setting $\rvnldistspace_M(t) := [e^{t \Delta} \rconn^0_M]$ for $t > 0$. By Lemma~\ref{lemma:in-nldist-space}, we have that indeed $\rvnldistspace_M \in \nonlineardistspace$.
\end{definition}

We proceed to show that the sequence $\{\rvnldistspace_M\}_{M \geq 1}$ converges in distribution. We will then define the limiting measure to be the 3D $\unitary(1)$ theory. We will follow the standard path of proving tightness and the convergence of the finite dimensional distributions. First, we show the tightness condition of Theorem \ref{mainthm}.

\begin{lemma}\label{lemma:u(1)-tightness}
For all $t > 0$, $\{S_{\textup{YM}}(\rvnldistspace_\fouriercutoff(t))\}_{\fouriercutoff \geq 1}$ is a tight family of $\R$-valued random variables.
\end{lemma}
\begin{proof}
Let $t > 0$. Observe that 
\[ S_{\textup{YM}}(\rvnldistspace_\fouriercutoff(t)) = S_{\textup{YM}}(e^{t\Delta} \rconn^0_M) .\]
Now by \eqref{eq:heat-kernel-def} and Lemma \ref{lemma:da-l2-norm-formula}, we have that 
\begin{align*}
\E \big[\sym(e^{t\Delta} \rconn^0_M)\big] &\leq 8\pi^2 \sum_{n \in \Z^3} e^{-8 \pi^2 |n|^2 t} |n|^2 \E\big[|Z_n|^2\big].
\end{align*}
To finish, note by construction that $\E [|Z_n|^2] = g^2 O(|n|^{-2})$, and thus the right hand side above is finite (and also independent of $M$). The desired result follows.
\end{proof}

Next, we show the analogue of convergence of the finite dimensional distributions.~First, for any piecewise $C^1$ loop $\wloop : [0, 1] \ra \threetorus$ and $t > 0$, define the ($\icomplex\R$)-valued random variable
\[ H_{\wloop, t} := \sum_{n \in \Z^3 \setminus \{0\}} e^{-4\pi^2 |n|^2 t} Z_n \cdot \int_0^1 e_n(\wloop(s)) \wloop'(s) ds. \]
Note that a.s., the sum above converges absolutely, since $\E[|Z_n|^2] = g^2O(|n|^{-2})$ and $\int_0^1 |\wloop'(s)| ds < \infty$, so that $\sum_{n \in \Z^3 \setminus \{0\}} e^{-4\pi^2 |n|^2 t} \E[|Z_n|] \int_0^1 |\wloop'(s)| ds < \infty$ (here we also use that if a sum of non-negative random variables has finite expectation, then the sum converges a.s.).

\begin{lemma}\label{lemma:u(1)-fdds-converge}
For any finite collection $\wloop_i, \character_i, t_i$, $1 \leq i \leq k$, where $\wloop_i : [0, 1] \ra \threetorus$ is a piecewise $C^1$ loop, $\character_i$ is a character of $\liegroup$, and $t_i > 0$, we have that as $\fouriercutoff \toinf$, $(W_{\wloop_i, \character_i, t_i}(\rvnldistspace_\fouriercutoff), 1 \leq i \leq k)$ converges in distribution to $(\character_i(\exp(H_{\wloop_i, t_i})), 1 \leq i \leq k)$.
\end{lemma}
\begin{proof}
We in fact have a.s. convergence. To see this, note that for any $\wloop, \character, t$, we have
\begin{align*}
W_{\wloop, \character, t}(\rvnldistspace_\fouriercutoff) & = W_{\wloop, \character}(e^{t\Delta} \rconn^0_M)\\
&= \character\Bigg(\exp \bigg(\sum_{\substack{|n|_\infty \leq M \\ n \neq 0}} e^{-4\pi^2 |n|^2 t} Z_n \cdot \int_0^1 e_n(\wloop(s)) \wloop'(s) ds \bigg)\Bigg),
\end{align*}
where in the second equality, we have used Lemma \ref{lemma:wilson-loop-heat-kernel-regularized}. By the remark right before the statement of this lemma, we have that
\[ \lim_{M \toinf} \sum_{\substack{|n|_\infty \leq M \\ n \neq 0}} e^{-4\pi^2 |n|^2 t} Z_n \cdot \int_0^1 e_n(\wloop(s)) \wloop'(s) ds  \stackrel{a.s.}{=} H_{\ell, t}.\]
From this, the a.s.~convergence of $W_{\wloop_i, \character_i, t_i}(\rvnldistspace_\fouriercutoff)$ to $\chi_i(\exp(H_{\wloop_i, t_i}))$ for each $1 \leq i \leq k$ follows.
\end{proof}

By combining Lemmas \ref{lemma:u(1)-tightness} and \ref{lemma:u(1)-fdds-converge} and applying Theorem \ref{mainthm}, we can now show that the $\rvnldistspace_\fouriercutoff$ converge in distribution.

\begin{prop}\label{prop:u(1)-finite-dim-approximations-convergence}
The random variables $\rvnldistspace_\fouriercutoff$ converge in distribution (as $\fouriercutoff \toinf$) to a $\nonlineardistspace$-valued random variable $\rvnldistspace_\infty$, whose law is characterized as follows. For any finite collection $\wloop_i, \character_i, t_i$, $1 \leq i \leq k$, where $\wloop_i : [0, 1] \ra \threetorus$ is a piecewise $C^1$ loop, $\character_i$ is a character of $\liegroup$, and $t_i > 0$, we have that
\[ (W_{\wloop_i, \character_i, t_i}(\rvnldistspace_\infty), 1 \leq i \leq k) \text{ is equal in law to } (\character_i(\exp(H_{\wloop_i, t_i})), 1 \leq i \leq k).\]
(It follows by Lemma \ref{fdlemma} that the law of $\rvnldistspace_\infty$ is uniquely determined by the above.)
\end{prop}
\begin{proof}
By Lemma \ref{lemma:u(1)-tightness}, the tightness condition of Theorem \ref{mainthm} is satisfied, and thus for any subsequence $\{\rvnldistspace_{\fouriercutoff_k}\}_{k \geq 1}$, there exists a further subsequence $\{\rvnldistspace_{\fouriercutoff_{k_j}}\}_{j \geq 1}$ which converges in distribution, say to some $\nonlineardistspace$-valued random variable $\tilde{\rvnldistspace}_\infty$. To finish, it suffices to show that subsequential limits are unique, i.e., the law of $\tilde{\rvnldistspace}_\infty$ is unique. To see this, by Lemma \ref{lemma:u(1)-fdds-converge} and the fact that regularized Wilson loop observables are bounded continuous functions on $\nonlineardistspace$ (by the definition of the topology on $\nonlineardistspace$), we have that
\[ (W_{\wloop_i, \character_i, t_i}(\tilde{\rvnldistspace}_\infty), 1 \leq i \leq k) \text{ is equal in law to }  (\character_i(\exp(H_{\wloop_i, t_i})), 1 \leq i \leq k)\]
for any finite collection $\wloop_i, \character_i$, $t_i$, $1 \leq i \leq k$. By Lemma \ref{fdlemma}, this gives the desired uniqueness of the subsequential limits. 
\end{proof}

In light of Proposition \ref{prop:u(1)-finite-dim-approximations-convergence}, we can define the 3D $\unitary(1)$ theory as follows.

\begin{definition}
The 3D $\unitary(1)$ theory is defined to be the law of the $\nonlineardistspace$-valued random variable $\rvnldistspace_\infty$ from Proposition \ref{prop:u(1)-finite-dim-approximations-convergence}. Thus it is a probability measure on $(\nonlineardistspace, \mc{B}_\nonlineardistspace)$.
\end{definition}

\begin{remark}\label{remark:explicit-formula-rvnldistpsace-infty}
Naturally, the random variable $\rvnldistspace_\infty$ can be defined using the following explicit formula:
\beq\label{eq:x-infty-explicit-formula} \rvnldistspace_\infty(t) := \bigg[\sum_{n \in \Z^3 \setminus \{0\}} e^{-4\pi^2 |n|^2 t} Z_n e_n \bigg], ~~ t > 0.  \eeq
Note that since $\E[|Z_n|^2] = g^2 O(|n|^{-2})$, we have that $\sum |n|^{-2} |Z_n|^2 < \infty$ a.s. Using this, we may obtain that a.s., for any $t > 0$, the sum on the right hand side defines a smooth {\oneform}, and thus the right hand side above is well-defined for any $t > 0$. (Off this a.s.~event, we set $\rvnldistspace_\infty(t)$ to be the gauge orbit of the zero {\oneform} (say) for all $t > 0$.) With this definition of $\rvnldistspace_\infty$, one can show that $\rvnldistspace_M \ra \rvnldistspace_\infty$~a.s.
\end{remark}

\begin{remark}\label{remark:u(1)-comparison-with-Gross}
Let us see how our construction of the 3D $\unitary(1)$ theory compares to Gross's construction \cite{G1983}, which was defined as the law of a random distributional curvature form. First, observe that the sequence of random variables $\{\rconn^0_M\}_{M \geq 1}$ from Definition~\ref{def:u(1)-fourier-cutoff} a.s.~converges in a distributional space (for instance, in the negative Sobolev space $H^{-r}(\threetorus, \lalg^3)$ for large enough $r \geq 0$) to the following distributional {\oneform}:
\[ \rconn^0_\infty := \sum_{\substack{n \in \Z^3 \\ n \neq 0}} Z_n e_n. \]
Note that the explicit formula for the random variable $\rvnldistspace_\infty$ in Remark \ref{remark:explicit-formula-rvnldistpsace-infty} is exactly given by $\rvnldistspace_\infty(t) = [e^{t\Delta} \rconn^0_\infty]$ for $t > 0$. On the other hand, we may take the curvature $\curv{\rconn^0_\infty}$, which is a random distributional curvature form.
It can then be checked that $\curv{\rconn^0_\infty}$ has law which coincides with Gross's construction of the 3D $\unitary(1)$ theory as a random distributional curvature form (see in particular \cite[Equation (1.2)]{G1983}). 
\end{remark}


\section{Proof of the tightness result}\label{mainproof}

We begin towards the proof of Theorem \ref{mainthm}. The arguments we develop along the way will also allow us to prove Lemmas \ref{fdlemma} and \ref{symlmm}. Instead of working directly with smooth {\oneforms} $\connspace$ and smooth gauge transformations $\gaugetransf$, we will find it convenient to work mostly with $H^1$ {\oneforms} $\honeconnspace$ and $H^2$ gauge transformations $\htwogaugetransf$ (these will be defined in Section \ref{section:3dalt-notation}). This is mostly due to the property of Weak Uhlenbeck Compactness (see Theorem \ref{thm:weak-uhlenbeck-compactness}), which is satisfied by $\honeconnspace$, $\htwogaugetransf$. Thus our plan will be to first prove the analogue of Theorem \ref{mainthm} for the nonlinear distribution space obtained by using $\honeconnspace, \htwogaugetransf$ in place of $\connspace, \gaugetransf$, and then to deduce Theorem \ref{mainthm} from this analogue.

To summarize Section \ref{mainproof}, the aforementioned analogue of Theorem \ref{mainthm} is proven in Section \ref{section:nonlinear-dist-space}. In Sections \ref{section:3dalt-notation}--\ref{section:orbit-space-measure-theory}, we build up the needed results which are needed in the proof of the analogue. Then in Section \ref{section:completing-the-proof}, we use the preceding material to prove Theorem \ref{mainthm}, as well as Lemmas \ref{fdlemma} and \ref{symlmm}.

\subsection{Preliminaries}\label{section:3dalt-notation}

We now introduce some more notation, beyond what has already been introduced, that will be used throughout the rest of the manuscript. Throughout this section, $G$ will remain fixed, and $\const$ will denote a generic constant that may depend only on $\liegroup$. It may change from line to line, and even within a line. To express constants which depend on some term, say $\character$, we will write $\const_\character$. Again, $\const_\character$ may change from line to line, and even within a line. 

For general matrices $M$ (not necessarily belonging to $\liegroup$ or $\lalg$), let 
\[
\|M\| := \sqrt{\Tr(M^* M)}
\]
denote the Frobenius norm of $M$. We will often use the Cauchy--Schwarz inequality  $|\Tr(M_1^* M_2)| \leq \|M_1\| \|M_2\|$. Note also that the Frobenius norm is submultiplicative, that is, $\|M_1 M_2\| \leq \|M_1\| \|M_2\|$ (this, too, follows by the Cauchy--Schwarz inequality). 

We next begin to define a new space of {\oneforms} $\honeconnspace$ and gauge transformations $\htwogaugetransf$. We will see that $\honeconnspace$ is the space of $H^1$ {\oneforms}, and $\htwogaugetransf$ is the space of $H^2$ gauge transformations. (The superscripts attached to $\honeconnspace, \htwogaugetransf$ follow the convention of \cite{W2004}, and they are inspired by the usual notation for Sobolev spaces.)

Let $\lalg$ be endowed with the inner product $\langle \elemlalg_1, \elemlalg_2 \rangle_\lalg := \Tr(\elemlalg_1^* \elemlalg_2)$. This inner product turns $\lalg$ into a finite-dimensional real Hilbert space, and moreover, the norm associated with this inner product is the Frobenius norm. More generally, given any finite nonempty index set $I$, we endow $\lalg^I$ (whose elements will be denoted by tuples $\elemlalg = (\elemlalg_i, i \in I)$) with the inner product 
\[
\langle \elemlalg^1, \elemlalg^2\rangle_{\lalg^I} := \sum_{i \in I} \langle \elemlalg^1_i, \elemlalg^2_i \rangle_\lalg.
\]
Again, this inner product turns $\lalg^I$ into a finite-dimensional real Hilbert space. In a slight abuse of notation, we will write $\|\cdot\|$ also for the norm on $\lalg^I$, that is, for $\elemlalg \in \lalg^I$, we will write 
\[
\|\elemlalg\| := \langle \elemlalg, \elemlalg \rangle_{\lalg^I}^{1/2}.
\]
Let $(V, \langle \cdot, \cdot \rangle_V)$ be a finite-dimensional real Hilbert space. We will denote the norm on $V$ by $\|\cdot\|$ instead of $\|\cdot\|_V$. Let $1 \leq p \leq \infty$. Given a (measurable) function $f : \threetorus \ra \vectorspace$, let
\[
\|f\|_{L^p(\threetorus, \vectorspace)} := \begin{cases}  \big(\int_{\threetorus} \|f(x)\|^p dx\big)^{1/p} & 1 \leq p < \infty, \\ 
\esssup_{x \in \threetorus} \|f(x)\| & p = \infty .
\end{cases}
\]
For $1 \leq p \leq \infty$, let $L^p(\threetorus, V)$ be the real Banach space of (a.e.~equivalence classes of) measurable functions $\threetorus \ra V$ with finite $L^p$ norm. We will write $\|\cdot\|_p$ instead of $\|\cdot\|_{L^p(\threetorus, V)}$ for brevity. Note that when $p = 2$, $L^2(\threetorus, V)$ is a Hilbert space with inner product given by 
\beq\label{eq:l2-inner-product} (f_1, f_2) := \int \langle f_1(x), f_2(x) \rangle_V dx .\eeq 
Next, we define the relevant Sobolev spaces. First, let us talk about weak derivatives. 
Take any $1 \leq i \leq 3$. We say that $f$ is weakly differentiable in the $i$th coordinate if $f \in L^1(\threetorus, V)$, and there exists $h \in L^1(\threetorus, \vectorspace)$ such that for all smooth $\phi \in C^\infty(\threetorus, \vectorspace)$, we have the integration by parts identity
\[ \int_{\threetorus} \langle h(x), \phi(x)\rangle_{\vectorspace} dx = -\int_{\threetorus} \langle f(x), (\ptl_i \phi)(x)\rangle_{\vectorspace} dx, \]
where $\ptl_i \phi$ is the derivative of $\phi$ in the $i$th coordinate. Note that if $h$ exists, then it is unique in $L^1(\threetorus, V)$ (that is, as a function, it is unique up to a.e.~equality). We write $h = \ptl_i f$. 

The Sobolev space $H^1(\threetorus, \vectorspace)$ (also often denoted by $W^{1,2}(\threetorus, \vectorspace)$) is the real Hilbert space of (a.e.~equivalence classes of) measurable functions $f : \threetorus \ra \vectorspace$ such that $\|f\|_2 < \infty$, and for any $1 \leq i \leq 3$, the weak derivative $\ptl_i f$ exists and satisfies $\|\ptl_i f\|_2 < \infty$. The $H^1$ norm is defined as
\[ \|f\|_{H^1(\threetorus, \vectorspace)} := \bigg(\|f\|_2^2 + \sum_{i=1}^3 \|\ptl_i f\|_2^2\bigg)^{1/2}. \]
We will usually write $\|f\|_{H^1}$ instead of $\|f\|_{H^1(\threetorus, \vectorspace)}$. Note that $H^1(\threetorus, V)$ is a Hilbert space with inner product
\[ (f_1, f_1)_{H^1} := (f_1, f_2) + \sum_{i=1}^3 (\ptl_i f_1, \ptl_i f_2) .\]
The Sobolev space $H^2(\threetorus, \vectorspace)$ (also denoted by $W^{2,2}(\threetorus, \vectorspace)$) is the subspace of $H^1(\threetorus, \vectorspace)$ for which all the second order weak derivatives (that is, weak derivatives of the weak derivatives) exist and are in $L^2(\threetorus, \vectorspace)$.

\begin{definition}
Define the space of {\oneforms} $\honeconnspace := H^1(\threetorus, \lalg^3)$.
\end{definition}

Note that since $\honeconnspace$ is a Hilbert space, there are at least two natural topologies on $\honeconnspace$ --- the norm topology and the weak topology. We will write $A_n \ra A$ to mean convergence in the norm topology and we will write $A_n \weakarrow A$ to mean convergence in the weak topology. We will frequently use the following well-known fact about weakly convergent sequences (for a proof, see~\cite[Lemma 21.11]{Jost2005}).

\begin{lemma}\label{lemma:weak-convergence-and-h1-norm}
Let $\{A_n\}_{n \leq \infty} \sse \honeconnspace$ be such that $A_n \weakarrow A_\infty$. Then 
\[ \sup_{n \leq \infty} \|A_n\|_{H^1} < \infty. \]
\end{lemma}

We will also extensively use the following well-known result, sometimes called the Rellich--Kondrachov theorem (for a proof, see \cite[Theorem B.2]{W2004}).
\begin{theorem}\label{sobolevthm}
For each $1 \leq p \leq 6$, there exists a constant $\const_p$ such that for any $A \in \honeconnspace$, we have
\[ \|A\|_p \leq \const_p \|A\|_{H^1}.\]
Moreover, for $1 \leq p < 6$, the embedding $\honeconnspace \hookrightarrow L^p$ is compact (that is, the unit ball of $\honeconnspace$ is a relatively compact subset of $L^p$).
\end{theorem}

\begin{remark}\label{soboloverem}
In this paper, we will apply this theorem many times, and so instead of formally citing it each time, we will just say something like ``by the Sobolev embedding $H^1 \hookrightarrow L^p$" whenever we want to apply this theorem. A consequence of the compact embedding $H^1 \hookrightarrow L^p$ for $p < 6$ is that if $A_n \weakarrow A$, then $A_n \ra A$ in $L^p$ for all $p < 6$. To see this, note that if $A_n \weakarrow A$, then $\sup_{n \geq 1} \|A_n\|_{H^1} < \infty$ by Lemma \ref{lemma:weak-convergence-and-h1-norm}. By the compact embedding, we obtain subsequences convergent in $L^p$ for $p < 6$. Using that $A_n \weakarrow A$, it follows that any subsequential limit must in fact be $A$, and thus it follows that $A_n \ra A$ in $L^p$ for $p < 6$. We will use the same phrase ``by the Sobolev embedding $H^1 \hookrightarrow L^p$" whenever we want to apply this result.
\end{remark}

Next, we define a new space of gauge transformations $\htwogaugetransf$. We want to impose smoothness requirements on gauge transformations in a way that interacts nicely with the fact that our {\oneforms} are in $H^1$. First, we need to give meaning to $d\gt$ for $\gt$ which is not necessarily smooth. Given $\gt : \threetorus \ra \liegroup$ and $1 \leq i \leq 3$, the weak derivative $\ptl_i \gt$ is defined as follows. Recall that $G \sse \unitary(N)$, and so $G$ consists of $N \times N$ complex matrices. We may thus embed $G \hookrightarrow (\R^2)^{N \times N}$, and regard $\gt$ as a function from $\threetorus$ into $(\R^2)^{N \times N}$. The weak derivative $\ptl_i \gt$ (when it exists) is then just as we defined earlier for functions in $L^1(\threetorus, V)$, where now $V = (\R^2)^{N \times N}$. If $\ptl_i \gt$ exists for all $1 \leq i \leq 3$, we define $d\gt := (\ptl_i \gt, 1 \leq i \leq 3)$. Now, it's not hard to see that if $\ptl_i \gt$ exists, then $\ptl_i \gt(x) \in T_{\gt(x)}\liegroup$ for a.e. $x \in \threetorus$, from which it follows that $\gt(x)^{-1} \ptl_i \gt(x) \in \lalg$ for a.e.~$x \in \threetorus$. Thus if $d\gt$ exists, then $\gt^{-1} d\gt$ can be identified with an element of $L^1(\threetorus, \lalg^3)$, i.e., an $L^1$ {\oneform} (recall that weak derivatives by definition are in $L^1$). The following definition of $\htwogaugetransf$ requires the {\oneform} $\gt^{-1} d\gt$ to further be in $H^1$.

\begin{definition}
Following \cite{W2004}, we define 
\begin{align}
\label{eq:gauge-transformation-set-def} \htwogaugetransf &:= \{ s \exp(\xi) : s \in C^\infty(\torus^3, G), \xi \in H^2(\torus^3, \lalg)\} \\
&= \{\gt \in C(\torus^3, G) : \text{$d\gt$ exists and } \gt^{-1} d\gt \in \honeconnspace \}. \nonumber
\end{align}
(The set equality in the above display follows by the discussion just before \cite[Corollary B.9]{W2004}; see also \cite[Lemma B.5 and Remark B.6(ii)]{W2004}.) Note that $\gaugetransf \sse \htwogaugetransf$, that is, every smooth gauge transformation is an element of $\htwogaugetransf$.
\end{definition}

\begin{remark}
At several points in this paper, we will cite results from \cite{W2004} involving $\honeconnspace$ and $\htwogaugetransf$. The space $\honeconnspace$ is denoted by $\honeconnspace(P)$ and $\honeconnspace(P |_{\threetorus})$ in \cite{W2004} (where for us, $P = \threetorus \times G$ is a trivial principal $G$-bundle) --- see, e.g., \cite[Appendix B]{W2004}. Similarly, the set $\htwogaugetransf$ is denoted by $\htwogaugetransf(P)$ and $\htwogaugetransf(P|_{\threetorus})$ in \cite{W2004} --- see, for example, the discussion just after \cite[Equation (A.11)]{W2004}, as well as the discussion just before \cite[Corollary B.9]{W2004}.
\end{remark}

By using the second characterization of $\htwogaugetransf$ in \eqref{eq:gauge-transformation-set-def}, we can show that $\htwogaugetransf$ is a group under pointwise multiplication --- since, if $\gt_1, \gt_2 \in \htwogaugetransf$, then $\gt_1 \gt_2 \in C(\threetorus, \liegroup)$, and $(\gt_1 \gt_2)^{-1} d(\gt_1 \gt_2) = \gt_2^{-1} (\gt_1^{-1} d\gt_1) \gt_2 + \gt_2^{-1} d\gt_2$. By assumption, $\gt_2^{-1} d\gt_2 \in \honeconnspace$ and $\gt_1^{-1} d\gt_1 \in \honeconnspace$, and so $\gt_2^{-1} (\gt_1^{-1} d\gt_1) \gt_2 \in \honeconnspace$ as well.

The action of $\htwogaugetransf$ on $\honeconnspace$ is defined exactly as in~\eqref{eq:gaugetransform}. That is, for $A \in \honeconnspace$, $\gt \in \htwogaugetransf$, we let $\gaction{A}{\gt}$ to be as in \eqref{eq:gaugetransform}.
As with the action of $\gaugetransf$ on $\connspace$, we have that the action of $\htwogaugetransf$ on $\honeconnspace$ is a group action. Thus, as in the smooth case, for $A \in \honeconnspace$, we define the gauge orbit
\[ \gorbithone{A} := \{\gaction{A}{\gt} : \gt \in \htwogaugetransf\},\]
and also define the orbit space
\[ \honeorbitspace := \{\gorbithone{A} : A \in \honeconnspace\}.\]
For the following definitions, we will assume that all derivatives are defined, at least weakly. Given a {\oneform} $A$, the exterior derivative $dA$ is a $2$-form --- that is, an antisymmetric array of functions $((dA)_{ij}, 1 \leq i, j \leq 3)$ --- where for each $1 \leq i, j \leq 3$, $(dA)_{ij} : \threetorus \ra \lalg$ is given by
\beq\label{eq:d-a-def} (dA)_{ij} := \ptl_i A_j - \ptl_j A_i.\eeq
Given a function $f : \threetorus \ra \lalg$ (i.e., a {\zeroform}), the exterior derivative $df$ is a {\oneform} given by 
\[ (df)_i := \ptl_i f, ~~1 \leq i \leq 3.\]
For a {\oneform} $A$, the {\zeroform} $d^*A : \threetorus \ra \lalg$ is given by
\[ d^*A := - \sum_{i=1}^3 \ptl_i A_i . \]
For a {\twoform} $F : \threetorus \ra \lalg^{3 \times 3}$, the {\oneform} $d^*F$ is given by
\beq\label{eq:d-star-2-form-def} (d^*F)_i := \sum_{j=1}^3 \ptl_j F_{ij}, ~~ 1 \leq i \leq 3.\eeq
Finally, given {\oneforms} $A, B : \threetorus \ra \lalg^3$, the {\twoform} $[A \wedge B]$ is given by 
\beq\label{eq:A-wedge-B-one-form} [A \wedge B]_{i, j} := [A_i, B_j] - [A_j, B_i], ~~ 1 \leq i, j \leq 3. \eeq
Note  that $[A \wedge B] = [B \wedge A]$ and $[A \wedge A]_{ij} = 2[A_i, A_j]$. 

Given $A \in \honeconnspace$, the curvature $\curv{A}$ is defined as in \eqref{eq:curvdef}. Note this may also be written $\curv{A} = dA + (1/2)[A \wedge A]$. The YM action (for elements of $\honeconnspace$) is defined to be the function $\sym^{1,2} : \honeconnspace \ra [0, \infty)$, $A \mapsto \|\curv{A}\|_2^2$. Note that this definition coincides with \eqref{eq:sym-def}. Also, as in the smooth case, $\sym^{1,2}$ is gauge invariant, and so we will often write $\sym^{1,2}(\gorbithone{A}) = \sym^{1,2}(A)$.

The following lemma shows that $\curv{A}$ is in $L^2(\threetorus, \lalg^{3 \times 3})$ for all $A \in \honeconnspace$, and thus $\sym^{1,2}$ is indeed finite on $\honeconnspace$. Moreover, it bounds $\sym^{1,2}(A)$ in terms of~$\|A\|_{H^1}$.

\begin{lemma}\label{lemma:curvature-bounded-by-h1-norm}
Let $A \in \honeconnspace$. We have
\[ \sym^{1,2}(A)^{1/2} = \|\curv{A}\|_2 \leq \const (\|A\|_{H^1} + \|A\|_{H^1}^2). \]
\end{lemma}
\begin{proof}
We have
\[ \|\curv{A}\|_2 \leq \|dA\|_2 + \|[A \wedge A]\|_2. \]
Note that $\|dA\|_2 \leq \const \|A\|_{H^1}$, and $\|[A \wedge A]\|_2 \leq \const \|A\|_4^2$. By the Sobolev embedding $H^1 \hookrightarrow L^4$ (Theorem \ref{sobolevthm}), $\|A\|_4\le \const \|A\|_{H^1}$. This completes the proof.
\end{proof}

The following well-known theorem due to Uhlenbeck \cite{Uh1982}, which is often called Weak Uhlenbeck Compactness, is a key tool. We will see how it is used in Section~\ref{section:topology-and-measure-theory}. This result only holds in 3D (though there is a version which holds in general dimensions --- see \cite[Theorem A]{W2004}).

\begin{theorem}[Theorem A of \cite{W2004}]\label{thm:weak-uhlenbeck-compactness}
Let $\{A_n\}_{n \geq 1} \sse \honeconnspace$ be a sequence such that $\sup_{n \geq 1} \sym^{1,2}(A_n) < \infty$. Then there exists a subsequence $\{A_{n_k}\}_{k \geq 1}$ and a sequence $\{\gt_{n_k}\}_{k \geq 1} \sse \htwogaugetransf$ such that $\gaction{A_{n_k}}{\gt_{n_k}}$ converges weakly in $\honeconnspace$.
\end{theorem}

The next two lemmas will be needed in the coming sections. The proofs are in Appendix \ref{section:3d-miscellaneous-proofs}.

\begin{lemma}\label{lemma:limiting-gauge-transformation}
Let $\{A_n\}_{n \leq \infty}, \{B_n\}_{n \leq \infty} \sse \honeconnspace$ be sequences such that $A_n \weakarrow A_\infty$, $B_n \weakarrow B_\infty$, and such that for all $n \geq 1$, there exists $\gt_n \in \htwogaugetransf$ such that $B_n = \gaction{A_n}{\gt_n}$. Then there exists $\gt \in \htwogaugetransf$ such that $B_\infty = \gaction{A_\infty}{\gt}$. 
\end{lemma}

\begin{lemma}\label{lemma:wilson-loop-gauge-invariant-not-smooth-gauge-transform}
Let $A, \tilde{A}$ be smooth {\oneforms}. Suppose $\gt \in \htwogaugetransf$ is such that $\tilde{A}(x) = \gaction{A}{\gt}(x)$ for a.e.~$x \in \threetorus$. Then $\gt$ is smooth.
\end{lemma}

We close this section with the following lemma, which will be needed later. This result follows immediately by \cite[Example 58]{Pon1966}, which implies that the set of equivalence classes of irreducible characters of a second countable compact group is countable. (We also use the fact that every character can be written as a finite sum of irreducible characters --- see, e.g., \cite[Section 8.1]{Huang1999}.)

\begin{lemma}\label{lemma:lie-group-countable-characters}
The set of characters of $\liegroup$ is countable.
\end{lemma}

\subsection{Properties of regularizing flows}\label{section:3d-alt-ym-heat-equation}

In this section, we state the results about solutions to the Yang--Mills heat flow with initial data in $\honeconnspace$ which will be needed later. All proofs are in Appendix~\ref{section:ym-proof-h1}.

Given a {\oneform} $A$ and a {\twoform} $F$, define the {\oneform} (recalling the definition of $d^*F$ in \eqref{eq:d-star-2-form-def})
\beq\label{eq:d-star-a-def} d^*_{A} F := d^*F +  [A \lrcorner F],\eeq
where the {\oneform} $[A \lrcorner F]$ is given by
\beq\label{eq:interior-product-def} [A \lrcorner F]_i := \sum_{j=1}^3 [A_j, F_{ij}], ~~ 1 \leq i \leq 3. \eeq
With this notation, the Yang--Mills heat flow defined in \eqref{eq:ymdef} can be written more compactly as the following PDE on {\oneforms}:
\beq\label{eq:YM}\tag{$\textup{YM}$} 
\ptl_t A(t) = - d^*_{A(t)} \curv{A(t)}. \eeq
We need to define what it means to be a solution to the Yang--Mills heat flow when the initial data is in $\honeconnspace$. We will use the notion of solution given by Charalambous and Gross in \cite[Definition 2.2]{CG2013}. They use the term ``strong solution", however we will just use the term ``solution". We first define the following notion, which will be used in Definition \ref{def:YM-solution-def}.

\begin{definition}[Strong derivative]
Given $T > 0$ and a function $f : (0, T) \ra L^2(\threetorus, \lalg^3)$, the strong derivative of $f$ at some point $t \in (0, T)$ is defined as
\[ f'(t) := \lim_{h \ra 0} \frac{f(t+h) - f(t)}{h}, \]
where limit is in the norm topology of $L^2(\threetorus, \lalg^3)$, assuming it exists. 
\end{definition}

\begin{definition}\label{def:YM-solution-def}
Let $0 < T \leq \infty$. We say that $A : [0, T) \ra \honeconnspace$ is a solution to \eqref{eq:YM} on $[0, T)$ if $A$ is a continuous function (when $\honeconnspace$ is equipped with its $H^1$ norm topology) such that the following conditions all hold:
\begin{enumerate}[label=(\alph*)]
    \item The curvature $\curv{A(t)} \in H^1(\threetorus, \lalg^{3 \times 3})$ for all $t \in (0, T)$,
    \item when regarding $A$ as a function from $(0, T)$ into $L^2(\threetorus, \lalg^3)$, the strong derivative $A'(t)$ exists for all $t \in (0, T)$, 
    \item $A'(t) = -d^*_{A(t)} \curv{A(t)}$ for all $t \in (0, T)$,
    \item $\|\curv{A(t)}\|_\infty$ (recall the definition of $L^p$ norms in Section \ref{section:3dalt-notation}) is bounded on any bounded interval $[a, b] \sse (0, T)$,
    \item $t^{3/4} \|\curv{A(t)}\|_\infty$ is bounded on some interval $(0, b) \sse (0, T)$.
\end{enumerate}
Let $A_0 \in \honeconnspace$. If additionally $A(0) = A_0$, then we say that $A$ is a solution to \eqref{eq:YM} on $[0, T)$ with initial data $A(0) = A_0$.
\end{definition}



Hereafter, whenever we say that $A$ is a solution to \eqref{eq:YM}, by default we mean that $A$ is a solution in the sense of Definition \ref{def:YM-solution-def}. We will need the following theorem regarding the global existence and uniqueness of solutions to the Yang--Mills heat flow. 

\begin{theorem}[See Theorem 2.5 of \cite{CG2013} and Theorem 1 of \cite{R1992}]\label{thm:YM-global-existence}
Let $A_0 \in \honeconnspace$. There exists a solution $A$ to \eqref{eq:YM} on $[0, \infty)$ with initial data $A(0) = A_0$. If $A_0$ is smooth, then $A \in C^\infty([0, \infty) \times \threetorus, \lalg^3)$. Given two solutions $A, \tilde{A}$ to \eqref{eq:YM} on some interval $[0, T)$ with the same initial data $A(0) = \tilde{A}(0) \in \honeconnspace$, we have $A = \tilde{A}$ on~$[0, T)$. \end{theorem}

We will also need the following result later on. This is part of \cite[Theorem 1]{R1992}, and thus the proof is omitted.

\begin{lemma}\label{lemma:ym-continuity-in-initial-data-smooth-case}
Let $\{A_{0, n}\}_{n \leq \infty} \sse \honeconnspace$ be a sequence of {\oneforms} such that $A_{0, n}$ is smooth for all $n \leq \infty$, and such that $A_{0, n} \ra A_{0, \infty}$ in $H^1$ norm. For each $n \leq \infty$, let $A_n$ be the solution to \eqref{eq:YM} on $[0, \infty)$ with initial data $A_n(0) = A_{0, n}$. Then for all $0 \leq T < \infty$, we have
\beq\label{eq:ym-continuity-in-initial-data-smooth-case} \lim_{n\to\infty} \sup_{0 \leq t \leq T} \|A_n(t) - A_\infty(t)\|_{H^1} = 0. \eeq
\end{lemma}

Note that if $A$ is a solution to \eqref{eq:YM} on some interval $[0, T)$, and $\gt \in \htwogaugetransf$, then $\gaction{A}{\gt}$ (i.e.,  the function $t \mapsto \gaction{A(t)}{\gt}$) is also a solution to \eqref{eq:YM}. Combining this with the uniqueness part of Theorem \ref{thm:YM-global-existence}, we have the following lemma, which says that the Yang--Mills heat flow is gauge covariant.

\begin{lemma}\label{lemma:ym-gauge-covariant-h1-initial-data}
Let $T > 0$. Let $A_0 \in \honeconnspace$, $\gt \in \htwogaugetransf$. Let $A, B$ be the solutions to \eqref{eq:YM} on $[0, T)$ with initial data $A(0) = A_0$, $B(0) = \gaction{A_0}{\gt}$. Then $B = \gaction{A}{\gt}$, that is, $B(t) = \gaction{A(t)}{\gt}$ for all $t \in [0, T)$.
\end{lemma}

We will need the following proposition from \cite{CG2013} concerning growth of solutions to \eqref{eq:YM}.

\begin{prop}[Corollary 10.3 of \cite{CG2013}]\label{prop:YM-h1-norm-estimate}
There is a continuous non-decreasing map $\normbdfn : [0, \infty)^2 \ra [0, \infty)$, such that for any solution $A$ to \eqref{eq:YM} on $[0, T)$ for some $0 < T \leq \infty$, we have that for all $0 \leq t < T$,
\[ \sup_{0 \leq s \leq t} \|A(s)\|_{H^1} \leq \normbdfn(t, \|A(0)\|_{H^1}).\]
Here, non-decreasing means $\normbdfn(t_1, M_1) \leq \normbdfn(t_2, M_2)$ for all $(t_1, M_1), (t_2, M_2) \in [0, \infty)^2$ such that $t_1 \leq t_2$ and $M_1 \leq M_2$.
\end{prop}

Hereafter, $\normbdfn$ will always denote the function given by Proposition \ref{prop:YM-h1-norm-estimate}.

We will also need a variant of the Yang--Mills heat flow, due to Zwanziger \cite{Zwan1981}, DeTurck \cite{DeT1983}, Donaldson \cite{Don1985}, Sadun \cite{Sad1987}. We will refer to this equation as the ZDDS equation. The ZDDS flow is equivalent to the Yang--Mills heat flow on the gauge orbit space (see Lemma \ref{lemma:ym-zdds-solution-gauge-transformation}), but has better smoothing properties on the space of connections. For a discussion of how the ZDDS  equation naturally arises from \eqref{eq:YM}, as well as some of its uses, see \cite[Section 1]{CCHS2020}. We will now state the ZDDS equation and list some of its properties. The proofs are in Appendix \ref{section:ym-proof-h1}.

Given a {\oneform} $A$ and a {\zeroform} $f$, define the {\oneform}
\[ d_{A} f := df + [A \wedge f], \]
where $[A \wedge f]$ is the {\oneform} given by
\[ [A \wedge f]_i := [A_i, f], ~~ 1 \leq i \leq 3.\]
The ZDDS equation is the following:
\beq\label{eq:ZDDS}\tag{$\textup{ZDDS}$} 
\ptl_t A(t) = - (d_{A(t)}^* \curv{A(t)} + d_{A(t)} d^* A(t)). \eeq
The difference between \eqref{eq:ZDDS} and \eqref{eq:YM} is the presence of $d_{A(t)} d^* A(t)$. We use the following notion of solution to \eqref{eq:ZDDS}, given in \cite[Theorem 2.14]{CG2013}.

\begin{definition}[Solution to \eqref{eq:ZDDS}]\label{def:ZDDS-solution-def}
Let $0 < T \leq \infty$. We say that $A : [0, T) \ra \honeconnspace$ is a solution to \eqref{eq:ZDDS} on $[0, T)$ if $A$ is a continuous function (when $\honeconnspace$ is equipped with its $H^1$ norm topology) such that the following conditions all hold:
\begin{enumerate}[label=(\alph*)]
    \item $\curv{A(t)} \in H^1(\threetorus, \lalg^{3 \times 3})$ and $d^* A(t) \in H^1(\threetorus, \lalg)$ for all $t \in (0, T)$,
    \item when regarding $A$ as a function from $(0, T)$ into $L^2(\threetorus, \lalg^3)$, the strong derivative $A'(t)$ exists for all $t \in (0, T)$,
    \item $A'(t) = - (d_{A(t)}^* \curv{A(t)} + d_{A(t)} d^*A(t))$ for all $t \in (0, T)$, \item $t^{3/4} \|\curv{A(t)}\|_\infty$ is bounded on $(0, T)$.
\end{enumerate}
Let $A_0 \in \honeconnspace$. If additionally $A(0) = A_0$, then we say that $A$ is a solution to \eqref{eq:ZDDS} on $[0, T)$ with initial data $A(0) = A_0$.
\end{definition}

Hereafter, whenever we say that $A$ is a solution to \eqref{eq:ZDDS}, by default we mean that $A$ is a solution in the sense of Definition \ref{def:ZDDS-solution-def}. We will need the following proposition regarding local existence of solutions to \eqref{eq:ZDDS} (note, in particular, that the proposition shows that solutions to \eqref{eq:ZDDS} are smooth at positive times). This result was essentially proven by Charalambous and Gross \cite[Theorem 2.14]{CG2013}, except that their result did not include continuity of solutions in the initial data. This additional property may be easily proven by Charalambous and Gross's methods. 

\begin{prop}[Cf.~Theorem 2.14 of \cite{CG2013}]\label{prop:ZDDS-local-existence}
There is a continuous non-increasing function $\timefn : [0, \infty) \ra (0, \infty)$ such that the following holds. For any $A_0 \in \honeconnspace$, there is a unique solution $A$ to \eqref{eq:ZDDS} on $[0, \timefn( \|A_0\|_{H^1}))$ with initial data $A(0) = A_0$. Moreover, $A \in C^\infty((0, \timefn( \|A_0\|_{H^1})) \times \threetorus, \lalg^3)$. Additionally, let $\{A_{0, n}\}_{n \leq \infty} \sse \honeconnspace$, $A_{0, n} \ra A_{0, \infty}$. For each $n \leq \infty$, let $A_n$ denote the solution to \eqref{eq:ZDDS} with initial data $A_n(0) = A_{0, n}$. Then for any $0 < t_0 < \timefn(\|A_{0, \infty}\|_{H^1})$, we have
\[ \begin{split}
\sup_{0 < t \leq t_0} &\big\{\|A_n(t) - A_\infty(t)\|_{H^1} + t^{1/4} \|A_n(t) - A_\infty(t)\|_\infty ~+ \\
&t^{3/4}\big(\|dA_n(t) - dA_\infty(t)\|_\infty + \|d^*A_n(t) - d^*A_\infty(t)\|_\infty\big) \big\}\ra 0.
\end{split}\]
(Note that since $\timefn$ is continuous, we have $\timefn(\|A_{0, n}\|_{H^1}) \ra \timefn( \|A_{0, \infty}\|_{H^1})$. Thus,  for any $t_0 < \timefn(\|A_{0, \infty}\|_{H^1})$, we have that $[0, t_0] \sse [0, \timefn(\|A_{0, n}\|_{H^1}))$ for all large enough~$n$.)
\end{prop}

\begin{remark}\label{remark:ym-zdds-comparison}
For the reason why we need to work with both \eqref{eq:YM} and \eqref{eq:ZDDS}, note that each equation has its pros and cons. Solutions to the equation \eqref{eq:YM} exist globally, however they may not be smooth, because \eqref{eq:YM} is only weakly parabolic (see \cite[Remark 2.9]{CG2013}). On the other hand, solutions to \eqref{eq:ZDDS} may not exist globally (at least, this hasn't been proven, as far as we can tell), but they are indeed smooth at positive times. 
\end{remark}

The following lemma shows that solutions to \eqref{eq:YM} and \eqref{eq:ZDDS} are related by a time-dependent gauge transformation. This result almost directly follows by \cite[Corollary 9.4]{CG2013}, except there is one small additional thing to show.~We give the proof in Appendix \ref{section:ym-proof-h1}. 

\begin{lemma}\label{lemma:ym-zdds-solution-gauge-transformation}
Let $A_0 \in \honeconnspace$. Let $B$ be the solution to \eqref{eq:ZDDS} on some interval $[0, T)$ with initial data $B(0) = A_0$. Let $A$ be the solution to \eqref{eq:YM} on $[0, T)$ with initial $A(0) = A_0$. For any $0 < t < T$, there exists a gauge transformation $\gt \in \htwogaugetransf$, possibly depending on $t$, such that $A(t) = \gaction{B(t)}{\gt}$.
\end{lemma}

\begin{remark}\label{remark:smoothing-procedure}
By Proposition \ref{prop:ZDDS-local-existence}, we can use \eqref{eq:ZDDS} to smooth out rough initial data (even more, see \cite[Section 19.7]{Feehan2016} for a local existence result for $L^3$ initial data). Moreover, this smoothing procedure is gauge invariant, in the sense that if $B$ is a solution to \eqref{eq:ZDDS} on some interval $[0, T)$, then for all $0 < s < t < T$, we have that $[B(t)] = \Phi_{t-s}([B(s)])$ (by Lemmas \ref{lemma:ym-zdds-solution-gauge-transformation} and \ref{lemma:wilson-loop-gauge-invariant-not-smooth-gauge-transform}). Thus in the setting of Theorem \ref{thm:free-field-behavior-implies-tightness}, if we obtain $\rconn^n$ which is not smooth (by, e.g., interpolating a lattice gauge configuration), we can fix this by running \eqref{eq:ZDDS} for a small amount of time, to obtain a smoothed version of $\rconn^n$.
\end{remark}

We next state the following result which will be needed in Section \ref{section:def-regularized-wilson-loops}. The proof is in Appendix \ref{section:ym-proof-h1}. This result is new, in that it is not contained in \cite{CG2013}, as was essentially the case for the previous results. 


\begin{lemma}\label{lemma:zdds-solutions-weak-continuity}
There is a continuous non-increasing function $\timefn_0 : [0, \infty) \ra (0, \infty)$ such that $\timefn_0 < \timefn$ (where $\timefn$ is as in Proposition \ref{prop:ZDDS-local-existence}), and such that the following holds. Suppose that $\{A_{0, n}\}_{n \leq \infty} \sse \honeconnspace$ and $A_{0, n} \weakarrow A_{0, \infty}$, so that $M := \sup_{n \leq \infty} \|A_{0, n}\|_{H^1} < \infty$ by Lemma \ref{lemma:weak-convergence-and-h1-norm}. For each $n \leq \infty$, let $A_n$ be the solution to \eqref{eq:ZDDS} on $[0, \timefn(\|A_{0, n}\|_{H^1}))$ with initial data $A_n(0) = A_{0, n}$. Then
\[ \sup_{0 \leq t \leq \timefn_0( M)} t^{1/2} \|A_n(t) - A_\infty(t)\|_{H^1} \ra 0.\]
(Note that for all $n \leq \infty$, $A_n$ is defined on $[0, \timefn_0(M)]$, since $\timefn_0(M) < \timefn(M) \leq \timefn(\|A_{0, n}\|_{H^1})$ by the monotonicity of $\timefn$.)
\end{lemma}

\subsection{Properties of Wilson loop observables}

Later on, we will need some results about the Wilson loop observables that were introduced in Section \ref{section:wilson-loop-observables}. The following lemmas show that Wilson loop observables depend in a continuous manner on the {\oneform} $A$ and the loop $\wloop$. The proofs are in Appendix \ref{section:Wilson-proofs}.

\begin{lemma}\label{lemma:wilson-loop-continuous-dependence-on-conn-general-character}
Let $\character$ be a character of $\liegroup$. Let $\wloop : [0, 1] \ra \threetorus$ be a piecewise $C^1$ loop. Let $A, \tilde{A}$ be smooth {\oneforms}. Then
\[ |W_{\wloop, \character}(A) - W_{\wloop, \character}(\tilde{A})|^2 \leq \const_\character \int_0^1 \|A(\wloop(t)) - \tilde{A}(\wloop(t))\| \cdot |\wloop'(t)| dt,\]
where $\const_\character$ is a constant which depends only on $\character$.
\end{lemma}

\begin{lemma}\label{lemma:wilson-loop-continuous-dependence-on-loop-general-character}
Let $\character$ be a character of $\liegroup$. Let $\wloop_1, \wloop_2 : [0, 1] \ra \threetorus$ be piecewise $C^1$ loops. Let $A$ be a smooth {\oneform}. Then
\[ \begin{split}
|W_{\wloop_1, \character}(A) - W_{\wloop_2, \character}(A)|^2 \leq \const_\character \int_0^1 \big(\|A(\wloop_1(t)) &- A(\wloop_2(t))\| \cdot |\wloop_1'(t)| ~+ \\
&\|A(\wloop_2(t))\| \cdot |\wloop_1'(t) - \wloop_2'(t)| \big)dt,
\end{split}\]
where $\const_\character$ is a constant which depends only on $\character$.
\end{lemma}

Recall that $\loopset$ is the set of piecewise $C^1$ loops $\wloop : [0, 1] \ra \threetorus$.

\begin{definition}\label{def:loopset-metric}
Define the following metric on $\loopset$:
\beq\label{eq:loopset-metric} d_\loopset(\wloop_1, \wloop_2) := \sup_{t \in [0, 1]} \metricspace(\wloop_1(t), \wloop_2(t)) + \int_0^1 |\wloop_1'(t) - \wloop_2'(t)| dt. \eeq
\end{definition}

\begin{definition}\label{def:indexset}
Let $\chset$ be the set of characters of $\liegroup$. Define the index set $\indexset := \loopset \times \chset \times (0, \infty)$. Define the following metric on $\indexset$:
\beq\label{eq:indexset-metric} d_\indexset((\wloop_1, \character_1, t_1), (\wloop_2, \character_2, t_2)) := \dloop(\wloop_1, \wloop_2) + \ind(\character_1 = \character_2) + |t_1 - t_2|.\eeq
\end{definition}

The following lemma shows that $(\loopset, \dloop)$ and consequently $(\indexset, d_\indexset)$ are separable (recall that $\chset$ is countable by Lemma \ref{lemma:lie-group-countable-characters}). The proof is in Appendix \ref{section:Wilson-proofs}.

\begin{lemma}\label{lemma:indexset-separable}
The space $(\loopset, \dloop)$ is separable. Consequently, $(\indexset, d_\indexset)$ is also separable.
\end{lemma}

\subsection{Regularized Wilson loop observables}\label{section:def-regularized-wilson-loops}

In this section, we define regularized Wilson loop observables for $H^1$ initial data and show that they determine gauge orbits.~These observables will be needed in many of the proofs in the coming sections. First, we need the following definition and lemma.

\begin{definition}[Yang--Mills heat semigroup]\label{def:ym-semigroup}
Let $t \geq 0$. Define the function $\tilde{\ymsemigroup}_t^{1,2} : \honeconnspace \ra \honeorbitspace$ as follows. Given $A_0 \in \honeconnspace$, let $A : [0, \infty) \ra \honeconnspace$ be the solution to \eqref{eq:YM} with initial data $A(0) = A_0$. Then let $\tilde{\ymsemigroup}_t^{1,2}(A_0) := \gorbithone{A(t)}$. By the gauge covariance of solutions to \eqref{eq:YM} (Lemma \ref{lemma:ym-gauge-covariant-h1-initial-data}), we have that $\tilde{\ymsemigroup}_t^{1,2}$ is gauge invariant, that is, $\tilde{\ymsemigroup}_t^{1,2}(A_0) = \tilde{\ymsemigroup}_t^{1,2}(\gaction{A_0}{\gt})$ for all $A_0 \in \honeconnspace$, $\gt \in \htwogaugetransf$. Thus there is a unique function $\ymsemigroup_t^{1,2} : \honeorbitspace \ra \honeorbitspace$ such that $\tilde{\ymsemigroup}_t^{1,2} = \ymsemigroup_t^{1,2} \circ \pi^{1,2}$, where recall $\pi^{1,2} : \honeconnspace \ra \honeorbitspace$ is the canonical projection map $A_0 \mapsto \gorbithone{A_0}$. Note that we have $\ymsemigroup_{s+t}^{1,2} = \ymsemigroup_s^{1,2} \circ \ymsemigroup_t^{1,2}$ for $s, t \geq 0$. We will often refer to this relation as the semigroup property of $(\ymsemigroup_t^{1,2}, t \geq 0)$.
\end{definition}

\begin{lemma}\label{lemma:orbitspace-gauge-orbit-smooth-choice}
Let $A_0 \in \honeconnspace$. For all $t > 0$, there exists a smooth {\oneform} $A \in \ymsemigroup_t^{1,2}([A_0]^{1,2})$.
\end{lemma}
\begin{proof}
Let $t_0 > 0$, and take $B_0 \in \ymsemigroup_{t_0}^{1,2}([A_0]^{1,2})$ (so $B_0 \in \honeconnspace$). By Proposition \ref{prop:ZDDS-local-existence}, we may obtain the solution $B$ to \eqref{eq:ZDDS} on some non-empty interval $[0, T)$ with initial data $B(0) = B_0$, and moreover $B$ is smooth on $(0, T) \times \threetorus$. Then by Lemma \ref{lemma:ym-zdds-solution-gauge-transformation}, we have that 
\[ \gorbithone{B(t)} = \ymsemigroup_t^{1,2}([B_0]^{1,2}) = \ymsemigroup_t^{1,2}(\ymsemigroup_{t_0}^{1,2}([A_0]^{1,2})) = \ymsemigroup_{t + t_0}^{1,2}([A_0]^{1,2}) \]
for all $t \in (0, T)$, and thus the claim holds for all $t \in (t_0, t_0 + T)$. Now take some $t_1 \in (t_0, t_0 + T)$. By Theorem \ref{thm:YM-global-existence}, we may obtain the solution $A$ to \eqref{eq:YM} on $[0, \infty)$ with initial data $A(0) = B(t_1 - t_0)$, and moreover since $B(t_1 - t_0)$ is smooth, we have that $A \in C^\infty([0, \infty) \times \threetorus, \lalg^3)$. For all $t \in [0, \infty)$, we have that $\gorbithone{A(t)} = \ymsemigroup_{t_1 + t}^{1,2}([A_0]^{1,2})$, and thus the claim holds for all $t \in (t_1, \infty)$, which implies the claim holds for all $t \in (t_0, \infty)$. Since $t_0$ was arbitrary, the desired result now follows.
\end{proof}

\begin{definition}[Regularized Wilson loop observables]\label{def:regwl}
Let $\wloop : [0, 1] \ra \threetorus$ be a piecewise $C^1$ loop, $\character$ be a character of $\liegroup$, and $t > 0$. We define a function $W_{\wloop, \character, t}^{1,2} : \honeconnspace \ra \C$ as follows. Given $A_0 \in \honeconnspace$, (applying Lemma \ref{lemma:orbitspace-gauge-orbit-smooth-choice}) let $A \in \ymsemigroup_t^{1,2}([A_0]^{1,2})$ be a smooth {\oneform}. Then define $W_{\wloop, \character, t}^{1,2}(A_0) := W_{\wloop, \character}(A)$. The fact that $W_{\wloop, \character, t}^{1,2}$ is well-defined will follow by Lemma \ref{lemma:f-ell-t-well-defined} below. Additionally, note that $W_{\wloop, \character, t}^{1,2}$ is gauge invariant by construction. Thus we define $[W]_{\wloop, \character, t}^{1,2} : \honeorbitspace \ra \C$ to be the unique function such that $W_{\wloop, \character, t}^{1,2} = [W]_{\wloop, \character, t}^{1,2} \circ \pi^{1,2}$, where $\pi^{1,2} : \honeconnspace \ra \honeorbitspace$ is the quotient map $A \mapsto \gorbithone{A}$. Both $W_{\wloop, \character, t}^{1,2}$ and $[W]_{\wloop, \character, t}^{1,2}$ will be called ``regularized Wilson loop observables''. 
\end{definition}

\begin{remark}
Evidently, the name ``regularized Wilson loop observable" comes from the fact that we are first using the Yang--Mills heat flow to regularize the given $H^1$ initial data, and then applying the usual Wilson loop observable to the regularized version.
\end{remark}

\begin{lemma}\label{lemma:f-ell-t-well-defined}
For any piecewise $C^1$ loop $\wloop : [0, 1] \ra \threetorus$, character $\character$ of $\liegroup$, and $t > 0$, $W_{\wloop, \character, t}^{1,2}$ is well-defined.
\end{lemma}
\begin{proof}
Let $A_0 \in \honeconnspace$, and let $A, B \in \ymsemigroup_t^{1,2}([A_0]^{1,2})$ be smooth {\oneforms}. Then there exists $\gt \in \htwogaugetransf$ such that $A = \gaction{B}{\gt}$. By Lemma \ref{lemma:wilson-loop-gauge-invariant-not-smooth-gauge-transform}, we have that in fact $\gt$ is smooth. Thus by the gauge invariance of Wilson loop observables, we have that $W_{\wloop, \character}(A) = W_{\wloop, \character}(B)$. The well-definedness follows.
\end{proof}

The next result shows that the regularized Wilson loop observables are (sequentially) weakly continuous. In the following, recall the functions $\normbdfn, \timefn$ from Propositions \ref{prop:YM-h1-norm-estimate} and \ref{prop:ZDDS-local-existence}. 

\begin{lemma}\label{lemma:regularized-wilson-loop-weak-sequential-continuity}
Let $\wloop : [0, 1] \ra \threetorus$ be a piecewise $C^1$ loop, $\character$ a character of $\liegroup$, and $t > 0$. Let $\{A_{0, n}\}_{n \leq \infty} \sse \honeconnspace$, with $A_{0, n} \weakarrow A_{0, \infty}$. Then $W_{\wloop, \character, t}^{1,2}(A_{0, n}) \ra W_{\wloop, \character, t}^{1,2}(A_{0, \infty})$.
\end{lemma}
\begin{proof}
For $n \leq \infty$, let $B_n$ be the solution to \eqref{eq:ZDDS} on $[0, \timefn(\|A_{0, n}\|_{H^1}))$ with initial data $B_n(0) = A_{0, n}$. Let $T_0 := \timefn_0(\sup_{n \leq \infty}\|A_{0, n}\|_{H^1}) > 0$, where $\timefn_0$ is as in Lemma \ref{lemma:zdds-solutions-weak-continuity} (by the weak convergence of $\{A_{0, n}\}_{n \geq 1}$ and Lemma \ref{lemma:weak-convergence-and-h1-norm}, the sup is finite). By Lemma \ref{lemma:zdds-solutions-weak-continuity}, we have that
\begin{align}\label{bnconv}
 \sup_{0 \leq s \leq T_0} s^{1/2} \|B_n(s) - B_\infty(s)\|_{H^1} \ra 0. 
\end{align}
Fix some $t>0$ and take $t_0 \in (0, \min\{T_0, t\})$. From Definition \ref{def:regwl} and Lemma \ref{lemma:ym-zdds-solution-gauge-transformation}, we have that 
\beq\label{eq:regwl-intermediate-identity} W_{\wloop, \character, t}(A_{0, n}) = W_{\wloop, \character, t - t_0}(B_n(t_0)) \text{ for all $n \leq \infty$.}\eeq
Now for $n \leq \infty$, let $A_n$ be the solution to \eqref{eq:YM} on $[0, \infty)$ with initial data $A_n(0) = B_n(t_0)$. Note that each $B_n(t_0)$ is smooth, and that $\|B_n(t_0) - B_\infty(t_0)\|_{H^1} \ra 0$ by \eqref{bnconv}. Then by Lemma \ref{lemma:ym-continuity-in-initial-data-smooth-case}, we have that for all $s \geq 0$, $\|A_n(s) - A_\infty(s)\|_{H^1} \ra 0$. Let $M := \normbdfn(t - t_0, \sup_{n \leq \infty}\|A_n(0)\|_{H^1})$ (recall $\normbdfn$ is as in Proposition \ref{prop:YM-h1-norm-estimate}), and let $T := \timefn(M)$. Take $s \in (0, \infty) \cap (t - t_0 - T, t - t_0)$. For $n \leq \infty$, let $\tilde{B}_n$ be the solution to \eqref{eq:ZDDS} on $[0, \timefn(\|A_n(s)\|_{H^1}))$ with initial data $\tilde{B}_n(0) = A_n(s)$. By Proposition \ref{prop:YM-h1-norm-estimate}, we have that $\timefn(\|A_n(s)\|_{H^1}) \geq T$ for all $n \leq \infty$. Then by Proposition \ref{prop:ZDDS-local-existence}, we have that $\|\tilde{B}_n(t - t_0 - s) - \tilde{B}_\infty(t - t_0 - s)\|_{\infty} \ra 0$. To finish, note by \eqref{eq:regwl-intermediate-identity} and Definition \ref{def:regwl}, we have that for all $n \leq \infty$,
\[ W_{\wloop, \character, t}(A_{0, n}) = W_{\wloop, \character, t - t_0 - s}(A_n(s)) = W_{\wloop, \character}(\tilde{B}_n(t - t_0 - s)).\]
Now applying Lemma  \ref{lemma:wilson-loop-continuous-dependence-on-conn-general-character}, the desired result follows.
\end{proof}

The final technical result of this section, stated below, shows that the regularized Wilson loop observables determine (i.e., separate) gauge orbits. Even more, it shows that in fact a countable collection of regularized Wilson loop observables determines gauge orbits. 

First, recall the definitions of $(\loopset, \dloop)$ (Definition \ref{def:loopset-metric}) and $(\indexset, d_\indexset)$ (Definition \ref{def:indexset}). It follows by Lemma \ref{lemma:indexset-separable} that $(\loopset, \dloop)$ is separable. Fix a countable dense subset $\loopset_0$ of $(\loopset, \dloop)$, and define $\indexset_0 := \loopset_0 \times \chset \times (\Q \cap (0, \infty))$ (recall that $\chset$ is the set of characters of $\liegroup$). From the definition of $d_\indexset$, and the fact that $\chset$ is countable (by Lemma \ref{lemma:lie-group-countable-characters}), we have that $\indexset_0$ is a countable dense subset of $(\indexset, d_\indexset)$.

\begin{lemma}[Regularized Wilson loop observables determine gauge orbits]\label{lemma:countable-index-set-suffices}
Suppose that $\gorbithone{A_1}, \gorbithone{A_2} \in \honeorbitspace$ are such that for all $(\wloop, \character, t) \in \indexset_0$, we have that
\[ [W]_{\wloop, \character, t}^{1,2}(\gorbithone{A_1}) = [W]_{\wloop, \character, t}^{1, 2}(\gorbithone{A_2}). \]
Then $\gorbithone{A_1} = \gorbithone{A_2}$. 
\end{lemma}
\begin{proof}
We will show that there exists $\gt \in \htwogaugetransf$ such that $A_1 = \gaction{A_2}{\gt}$. First, note that for all $(\wloop, \character, t) \in \indexset_0$, we have $[W]_{\wloop, \character, t}^{1, 2}(\gorbithone{A_j}) = W_{\wloop, \character, t}^{1,2}(A_j)$ for $j = 1, 2$. By Lemma \ref{lemma:wilson-loop-continuous-dependence-on-loop-general-character} and the fact that $\loopset_0$ is dense in $(\loopset, \dloop)$, we can obtain that for all $\wloop \in \loopset$, $\character \in \chset$, and $q \in \Q \cap (0, \infty)$,
\[ W_{\wloop, \character, q}^{1, 2}(A_1) = W_{\wloop, \character, q}^{1, 2}(A_2).\]
Now let $M := \max\{\|A_1\|_{H^1}, \|A_2\|_{H^1}\}$, and let $\timefn$ be as in Proposition \ref{prop:ZDDS-local-existence}. Let $B_1, B_2$ be the solutions to \eqref{eq:ZDDS} on $[0, \timefn(M))$ with initial data $B_1(0) = A_1$, $B_2(0) = A_2$.  Take a sequence $\{q_n\}_{n \geq 1} \sse \Q \cap (0, \infty)$, $q_n \downarrow 0$. Let $N_0$ be such that $q_n < \timefn(M)$ for all $n \geq N_0$. Then for all $n \geq N_0$, we have (by Definition \ref{def:regwl} and Lemma \ref{lemma:ym-zdds-solution-gauge-transformation})
\[ W_{\wloop, \character, q_n}^{1, 2}(A_j) = W_{\wloop, \character}(B_j(q_n)), ~~ j = 1, 2, \wloop \in \loopset, \character \in \chset.\]
Combining the previous two displays and applying Lemma \ref{lmm:wilson-loops-determine-gauge-eq-class}, we obtain that for all $n \geq N_0$, there exists a smooth gauge transformation $\gt_n$ such that $B_1(q_n) = \gaction{B_2(q_n)}{\gt_n}$. Since $q_n \downarrow 0$ and $B_1, B_2$ are continuous functions into $\honeconnspace$ (recall Definition \ref{def:ZDDS-solution-def}), we have that $B_1(q_n) \ra A_1$, $B_2(q_n) \ra A_2$. Combining the previous two observations and applying Lemma \ref{lemma:limiting-gauge-transformation}, we obtain that there exists $\gt \in \htwogaugetransf$ such that $A_1 = \gaction{A_2}{\gt}$, as desired.
\end{proof}

\subsection{Topology on the orbit space}\label{section:topology-and-measure-theory}

In this section, we place a topology on the orbit space $\honeorbitspace$. This topology will be used to make $\honeorbitspace$ into a measurable space (by taking the Borel $\sigma$-algebra).

First, we review the concept of sequential convergence, which will be needed to define the topology on the orbit space. The following definitions and results on sequential convergence all come from \cite[Section 2]{D1964}. 

\begin{definition}[Sequential convergence, $L$-convergence, and $L^*$-convergence]\label{def:sequential-convergence}
Let $S$ be a set. A sequential convergence $C$ on $S$ is a relation between sequences $\{s_n\}_{n \geq 1} \sse S$ and elements $s \in S$, denoted by $s_n \ra_C s$, with the following properties:
\begin{enumerate}
    \item If $s_n = s$ for all $n$, then $s_n \ra_C s$.
    \item If $s_n \ra_C s$, then for any subsequence $\{s_{n_k}\}_{k \geq 1}$, also $s_{n_k} \ra_C s$.
\end{enumerate}
An $L$-convergence $C$ is a sequential convergence with the additional property:
\begin{enumerate}
\setcounter{enumi}{2}
\item If $s_n \ra_C s$ and $s_n \ra_C t$, then $s = t$.
\end{enumerate}
An $L^*$-convergence $C$ is an $L$-convergence with the additional property:
\begin{enumerate}
\setcounter{enumi}{3}
\item If $\{s_n\}_{n \geq 1}$ is such that for every subsequence $\{s_{n_k}\}_{k \geq 1}$, there exists a further subsequence $\{s_{n_{k_j}}\}_{j \geq 1}$ such that $s_{n_{k_j}} \ra_C s$, then $s_n \ra_C s$.
\end{enumerate}
\end{definition}
Given a topological space $(S, \ms{T}_S)$, the sequential convergence $\mc{C}(\ms{T}_S)$ on $S$ is defined by the criterion  that $s_n \ra_{\mc{C}(\ms{T}_S)} s$ if and only if $s_n \ra s$ in $(S, \ms{T}_S)$, that is, if and only if for all open sets $U \in \ms{T}_S$ with $U \ni s$, we have that $s_n \in U$ eventually (i.e., for $n$ sufficiently large). Note that properties 1, 2, and 4 in Definition~\ref{def:sequential-convergence} are automatically satisfied by $\mc{C}(\ms{T}_S)$. If $(S, \ms{T}_S)$ is Hausdorff, then property 3 also holds, in which case $\mc{C}(\ms{T}_S)$ is an $L^*$-convergence.

Conversely, given a sequential convergence $C$ on $S$, we can define a topology $T(C)$, which is defined as the collection of sets $U \sse S$ such that for any $s \in U$ and any sequence $\{s_n\}_{n \geq 1}$ such that $s_n \ra_C s$, we have that $s_n \in U$ eventually. Note that given a topology $\ms{T}_S$, we always have $\ms{T}_S \sse T(\mc{C}(\ms{T}_S))$. 

Now, a natural question is whether the notion of convergence induced by the topology $T(C)$ is the same as $C$. That is, is it true that $s_n \ra_{\mc{C}(T(C))} s$ if and only if $s_n \ra_C s$? For one, by the definition of $T(C)$, we have that $s_n \ra_C s$ implies $s_n \ra_{\mc{C}(T(C))} s$. Therefore the only remaining question is whether the converse direction holds. The following result from \cite[Section 2]{D1964} (see the paragraph just after \cite[Theorem 2.1]{D1964}) answers this question.

\begin{lemma}\label{lemma:convergence-of-topology-of-convergence}
Let $C$ be an $L$-convergence on a set $S$. Then $s_n \ra_{\mc{C}(T(C))} s$ if and only if for every subsequence $\{s_{n_k}\}_{k \geq 1}$, there exists a further subsequence $\{s_{n_{k_j}}\}_{j \geq 1}$ such that $s_{n_{k_j}} \ra_C s$. In particular, if $C$ is an $L^*$-convergence, then $s_n \ra_{\mc{C}(T(C))} s$ if and only if $s_n \ra_C s$.
\end{lemma}

We now record some nice properties of the topological space $(S, T(C))$ induced by a sequential convergence $C$ on a set $S$. This result is not new (for instance, it is stated in \cite[Section 2]{D1964}), but we give the proof in Appendix \ref{section:3d-miscellaneous-proofs} to illustrate how to work with sequential convergences and their induced topologies. 

\begin{lemma}\label{lemma:sequential-convergence-topology-properties}
Let $C$ be a sequential convergence on a set $S$, and let $(S, T(C))$ be the resulting topological space. A set $V \sse S$ is closed in $(S, T(C))$ if and only if it is sequentially closed with respect to $\ra_C$; that is, if and only if for any $s \in S$ such that there is a sequence $\{s_n\}_{n \geq 1} \sse V$ with $s_n \ra_C s$, we have that $s \in V$. Let $C'$ be another sequential convergence on another set $S'$. If a function $f : S \ra S'$ is such that $s_n \ra_C s$ implies $f(s_n) \ra_{C'} f(s)$, then $f : (S, T(C)) \ra (S', T(C'))$ is continuous.
\end{lemma}

The following corollary looks at continuity of $f : (S, T(C)) \ra (X, d)$, where $(X, d)$ is a metric space. The proof is in Appendix \ref{section:3d-miscellaneous-proofs}.

\begin{cor}\label{cor:sequential-convergence-metric-space}
Let $C$ be a sequential convergence on a set $S$. Let $(X, d)$ be a metric space. Let $f : S \ra X$. If $f$ is sequentially continuous with respect to $C$ and $d$, (i.e., $s_n \ra_C s$ implies $d(f(s_n), f(s)) \ra 0$), then $f : (S, T(C)) \ra (X, d)$ is continuous.
\end{cor}

This concludes our short primer on the various relevant properties of sequential convergences and the topologies that they induce. We next use these results to obtain nice topologies on $\honeconnspace$ and $\honeorbitspace$.

First, let $\ms{T}_w$ be the weak topology on $\honeconnspace$, and let $\mc{C}(\ms{T}_w)$ be the sequential convergence generated by $\ms{T}_w$. By definition, we have that $A_n \ra_{\mc{C}(\ms{T}_w)} A$ if and only if $A_n \weakarrow A$. Thus we will write the former rather than the latter. 

\begin{definition}[Sequential weak topology]\label{def:sequential-weak-toplogy}
Let $\ms{T}_{sw} := T(\mc{C}(\ms{T}_w))$. That is, $U \in \ms{T}_{sw}$ if and only if for any sequence $\{A_n\}_{n \leq \infty} \sse \honeconnspace$ such that $A_n \weakarrow A_\infty \in U$, we have that $A_n \in U$ eventually. We will refer to $\ms{T}_{sw}$ as the sequential weak topology. Note that $\ms{T}_w \sse \ms{T}_{sw}$, since $\ms{T}_0 \sse T(\mc{C}(\ms{T}_0))$ for any topology $\ms{T}_0$. Also, since $(\honeconnspace, \ms{T}_w)$ is Hausdorff, we have that $\mc{C}(\ms{T}_w)$ is an $L^*$-convergence, and thus by Lemma \ref{lemma:convergence-of-topology-of-convergence}, we have that $A_n \ra A$ in $(\honeconnspace, \ms{T}_{sw})$ if and only if $A_n \weakarrow A$. 
\end{definition}

The sequential weak topology is convenient because convergence of sequences in this topology is the same as  weak convergence, and to verify continuity of functions on $(\honeconnspace, \ms{T}_{sw})$, it suffices to work with sequences rather than nets (by Lemma~\ref{lemma:convergence-of-topology-of-convergence} and Corollary~\ref{cor:sequential-convergence-metric-space}). This would not be true for $(\honeconnspace, \ms{T}_w)$, since $(\honeconnspace, \ms{T}_w)$ is not a sequential space --- there exists a set $V \sse \honeconnspace$ which is closed in $(\honeconnspace, \ms{T}_{sw})$ (i.e., sequentially weakly closed) but not closed in $(\honeconnspace, \ms{T}_w)$ (i.e., not weakly closed). For an example, see~\cite[Example 3.31]{BC2011}. Consequently, the identity map $\iota : (\honeconnspace, \ms{T}_w) \ra (\honeconnspace, \ms{T}_{sw})$ is sequentially continuous but not continuous. We would like to work with weakly converging sequences $A_n \weakarrow A_\infty$, rather than weakly converging nets $A_\alpha \weakarrow A_\infty$, because weakly converging sequences are $H^1$ norm-bounded (by Lemma \ref{lemma:weak-convergence-and-h1-norm}), while weakly converging nets may not be.

We now define a topology on $\honeorbitspace$. Define the quotient map $\projection^{1,2} : \honeconnspace \ra \honeorbitspace$ by $A \mapsto \gorbithone{A}$. On $\honeorbitspace$, we can place multiple topologies. One option is the quotient topology $\toporbitspace[1] := \{U \sse \honeorbitspace : (\projection^{1,2})^{-1}(U) \in \ms{T}_{sw}\}$. Another option is the topology generated by the following sequential convergence (which is inspired by Weak Uhlenbeck Compactness --- recall Theorem \ref{thm:weak-uhlenbeck-compactness}): $\gorbithone{A_n} \weakarrow_q \gorbithone{A_\infty}$ if and only if there exists $\tilde{A}_n \in \gorbithone{A_n}$ for all $n \leq \infty$ such that $\tilde{A}_n \weakarrow \tilde{A}_\infty$. Note that properties 1 and 2 in Definition \ref{def:sequential-convergence} are satisfied for $\weakarrow_q$. (In a slight clash with our previous notation, although we will refer to $\weakarrow_q$ as the sequential convergence, converging sequences will be denoted $\gorbithone{A_n} \weakarrow_q \gorbithone{A_\infty}$, rather than $\gorbithone{A_n} \ra_{\weakarrow_q} \gorbithone{A_\infty}$.) The topology is then $\toporbitspace[2] := T(\weakarrow_q)$, that is, the collection of sets $U \sse \honeorbitspace$ such that for any sequence $\{\gorbithone{A_n}\}_{n \leq \infty}$ with $\gorbithone{A_n} \weakarrow_q \gorbithone{A_\infty} \in U$, we have that $\gorbithone{A_n} \in U$ eventually.

\begin{lemma}\label{lemma:topologies-on-quotient-coincide}
We have $\toporbitspace[1] = \toporbitspace[2]$.
\end{lemma}
\begin{proof}
If $A_n \weakarrow A_\infty$, then $\gorbithone{A_n} \weakarrow_q \gorbithone{A_\infty}$. Thus, $\projection^{1,2} :(\honeconnspace, \ms{T}_{sw}) \ra (\honeorbitspace, \toporbitspace[2])$ is continuous (by Lemma \ref{lemma:sequential-convergence-topology-properties}). It follows that $\toporbitspace[2] \sse \toporbitspace[1]$. For the reverse inclusion, let $U \in \toporbitspace[1]$, and consider a sequence $\gorbithone{A_n} \weakarrow_q \gorbithone{A_\infty} \in U$. Then there exists $\tilde{A}_n \in \gorbithone{A_n}$ for all $n \leq \infty$ such that $\tilde{A}_n \weakarrow \tilde{A}_\infty$. As $\tilde{A}_\infty \in (\projection^{1,2})^{-1}(U) \in \ms{T}_{sw}$, it follows that $\tilde{A}_n \in (\projection^{1,2})^{-1}(U)$ eventually, which implies that for all sufficiently large $n$,
\[
\gorbithone{A_n} = \gorbithone{\tilde{A}_n} \in \projection^{1,2}((\projection^{1,2})^{-1}(U)) \sse U
\]
Thus $U \in \toporbitspace[2]$, and the desired result now follows.
\end{proof}

In light of Lemma \ref{lemma:topologies-on-quotient-coincide}, we make the following definition.

\begin{definition}
Define the topology $\toporbitspace := \toporbitspace[1] = \toporbitspace[2]$.
\end{definition}

We next state the following lemma regarding the uniqueness of limits with respect to the sequential convergence $\weakarrow_q$.

\begin{lemma}\label{lemma:quotient-convergence-uniqueness-of-limits}
If $\{\gorbithone{A_n}\}_{n \geq 1} \sse \honeorbitspace$, $\gorbithone{A_\infty}, \gorbithone{\tilde{A}_\infty} \in \honeorbitspace$ are such that $\gorbithone{A_n} \weakarrow_q \gorbithone{A_\infty}$ and $\gorbithone{A_n} \weakarrow_q \gorbithone{\tilde{A}_\infty}$, then $\gorbithone{A_\infty} = \gorbithone{\tilde{A}_\infty}$. In particular, $\weakarrow_q$ satisfies property 3 of Definition \ref{def:sequential-convergence}, and thus $\weakarrow_q$ is an $L$-convergence.
\end{lemma}
\begin{proof}
This is a direct consequence of Lemma \ref{lemma:limiting-gauge-transformation}.
\end{proof}

\begin{remark}
We note here that in general, even if $C$ is an $L$-convergence on a set $S$ (so that in particular sequential limits are unique), it may not be the case that $(S, T(C))$ is Hausdorff. For an example, see \cite[Section 2]{D1964}. In our case, we will in fact have that $(\honeorbitspace, \toporbitspace)$ is Hausdorff, because there is a continuous injection from $(\honeorbitspace, \toporbitspace)$ to a Hausdorff space (see Theorem \ref{thm:countable-index-homeomorphism}). The proof of this uses some additional properties of $(\honeorbitspace, \toporbitspace)$ beyond the fact that it arises from an $L$-convergence.
\end{remark}

As was the case for the space $(\honeconnspace, \ms{T}_{sw})$, we may ask whether convergence in the topological space $(\honeorbitspace, \toporbitspace)$ is the same as $\weakarrow_q$. Unfortunately, at the moment we do not know whether this is true. Therefore, given a sequence $\{\gorbithone{A_n}\}_{n \leq \infty} \sse \honeorbitspace$, we will write $\gorbithone{A_n} \ra \gorbithone{A_\infty}$ to denote convergence in the space $(\honeorbitspace, \toporbitspace)$ --- that is, for any open set $U \in \toporbitspace$ with $U \ni \gorbithone{A_\infty}$, we have that $\gorbithone{A_n} \in U$ eventually. We certainly have that $\gorbithone{A_n} \weakarrow_q \gorbithone{A_\infty}$ implies $\gorbithone{A_n} \ra \gorbithone{A_\infty}$, and so $\weakarrow_q$ is the stronger notion of convergence. By Lemma \ref{lemma:convergence-of-topology-of-convergence}, the question of whether $\gorbithone{A_n} \ra \gorbithone{A_\infty}$ implies $\gorbithone{A_n} \weakarrow_q \gorbithone{A_\infty}$ boils down to whether $\weakarrow_q$ is an $L^*$-convergence, that is, whether it satisfies property 4 of Definition \ref{def:sequential-convergence}. At this moment, we do not know whether this is true. On the other hand, since $\weakarrow_q$ is an $L$-convergence by Lemma \ref{lemma:quotient-convergence-uniqueness-of-limits}, we have the following immediate consequence of Lemma \ref{lemma:convergence-of-topology-of-convergence} (the proof is omitted).

\begin{lemma}\label{lemma:orbit-space-two-notions-of-convergence}
Let $\{\gorbithone{A_n}\}_{n \leq \infty} \sse \honeorbitspace$. We have that $\gorbithone{A_n} \ra \gorbithone{A_\infty}$ if and only if for any subsequence $\{\gorbithone{A_{n_k}}\}_{k \geq 1}$, there exists a further subsequence $\{\gorbithone{A_{n_{k_j}}}\}_{j \geq 1}$ such that $\gorbithone{A_{n_{k_j}}}\weakarrow_q \gorbithone{A_\infty}$.
\end{lemma}

\begin{remark}
Although it is a bit unwieldy to have two (possibly) different notions of convergence, we fortunately can mostly work with the stronger notion $\gorbithone{A_n} \weakarrow_q \gorbithone{A_\infty}$ throughout this section. The main reason why we can do so is because by Lemma \ref{lemma:sequential-convergence-topology-properties}, to check that a set $V$ is closed in $(\honeorbitspace, \toporbitspace)$, we can just check that it is sequentially closed with respect to $\weakarrow_q$, and to check that a function $f$ on $(\honeorbitspace, \toporbitspace)$ is continuous, we can just check that it is sequentially continuous with respect to $\weakarrow_q$. The next lemma is an example of this.
\end{remark}

\begin{lemma}\label{lemma:regwlorbit-continuous}
For any piecewise $C^1$ loop $\wloop$, character $\character$ of $\liegroup$, and $t > 0$, the regularized Wilson loop observable $[W]_{\wloop, \character, t}^{1,2} : (\honeorbitspace, \toporbitspace) \ra \C$ is continuous.
\end{lemma}
\begin{proof}
This follows directly by Corollary \ref{cor:sequential-convergence-metric-space}, the definition of the sequential convergence $\weakarrow_q$, and Lemma \ref{lemma:regularized-wilson-loop-weak-sequential-continuity}.
\end{proof}

We next begin to build towards the key result about the topology $\toporbitspace$ (Theorem~\ref{thm:countable-index-homeomorphism}). This is where Weak Uhlenbeck Compactness (Theorem \ref{thm:weak-uhlenbeck-compactness}) enters into the picture.

\begin{definition}
For $\cutoff \geq 0$, define the set 
\[ \cutoffspace{\cutoff} := \{\gorbithone{A} \in \honeorbitspace: \sym^{1,2}(\gorbithone{A}) \leq \cutoff\} \]
to be the set of gauge orbits with Yang--Mills action at most $\cutoff$. Let $\toporbitspace[\cutoff]$ denote the subspace topology on $\cutoffspace{\cutoff}$ induced by $(\honeorbitspace, \toporbitspace)$. 
\end{definition}

The next lemma shows that $\cutoffspace{\cutoff}$ is closed and sequentially compact.

\begin{lemma}\label{lemma:sequentially-compact-cutoff-space}
For any $\cutoff \geq 0$, the set $\cutoffspace{\cutoff}$ is closed in $(\honeorbitspace, \toporbitspace)$. Additionally, for any sequence $\{\gorbithone{A_n}\}_{n \geq 1} \sse \cutoffspace{\cutoff}$, there exists $\gorbithone{A} \in \cutoffspace{\cutoff}$ and a subsequence $\{\gorbithone{A_{n_k}}\}_{k \geq 1}$ such that $\gorbithone{A_{n_k}} \weakarrow_q \gorbithone{A}$. Consequently, $(\cutoffspace{\cutoff}, \toporbitspace[\cutoff])$ is sequentially compact.
\end{lemma}
\begin{proof}
First, note that by Weak Uhlenbeck Compactness (Theorem \ref{thm:weak-uhlenbeck-compactness}), for any sequence $\{\gorbithone{A_n}\}_{n \geq 1} \sse \cutoffspace{\cutoff}$, there exists $\gorbithone{A} \in \honeorbitspace$ and a subsequence $\{\gorbithone{A_{n_k}}\}_{k \geq 1}$ such that $\gorbithone{A_{n_k}} \weakarrow_q \gorbithone{A}$. If $\cutoffspace{\cutoff}$ is indeed closed, then $\gorbithone{A} \in \cutoffspace{\cutoff}$.

Thus it remains to show that $\cutoffspace{\cutoff}$ is closed. Let $\{\gorbithone{\tilde{A}_n}\}_{n \geq 1} \sse \cutoffspace{\cutoff}$, $\gorbithone{\tilde{A}_n} \weakarrow_q \gorbithone{\tilde{A}} \in \honeorbitspace$. Then there exists $A_n \in \gorbithone{\tilde{A}_n}$ for all $n$, and $A \in \gorbithone{\tilde{A}}$, such that $A_n \weakarrow A$. We know that $\sym^{1,2}(A_n) \leq \cutoff$ for all $n$, and if we can show that also $\sym^{1,2}(A) \leq \cutoff$, then that would mean that $\cutoffspace{\cutoff}$ is sequentially closed with respect to $\weakarrow_q$, and thus closed (by Lemma \ref{lemma:sequential-convergence-topology-properties}). We claim that
\[ \curv{A_n} = dA_n + \frac{1}{2} [A_n \wedge A_n] \ra dA + \frac{1}{2} [A \wedge A] = \curv{A} ~~\text{ weakly in $L^2(\threetorus, \lalg^3)$}. \]
Given this claim, the desired result follows, since the norm of a Hilbert space is sequentially lower semicontinuous with respect to the weak topology (see, e.g., \cite[Corollary 21.9]{Jost2005}), which gives
\[ \sym^{1,2}(A) = \|\curv{A}\|_2^2 \leq \liminf_n \|\curv{A_n}\|_2^2 = \liminf_n \sym^{1,2}(A_n) \leq \cutoff.\]
To show the claim, first note that by the Sobolev embedding $H^1 \hookrightarrow L^4$ (Theorem~\ref{sobolevthm} and Remark \ref{soboloverem}), we have that $A_n \ra A$ in $L^4$, which implies $[A_n \wedge A_n] \ra [A \wedge A]$ in $L^2$. Next, note that since $A_n \weakarrow A$, we have that $\ptl_i A_n \ra \ptl_i A$ weakly in $L^2$ for each $1 \leq i \leq 3$, which implies $dA_n \ra dA$ weakly in $L^2$. Combining the two results yields the claim.
\end{proof}

\begin{definition}\label{def:countable-index-homeomorphism}
Let $\indexset_0$ be as in Lemma \ref{lemma:countable-index-set-suffices}. Define the map $\Psi_0 : \honeorbitspace \ra \C^{\indexset_0}$ by $\gorbithone{A} \mapsto ([W]_{\wloop, \character, t}^{1, 2}(\gorbithone{A}), (\wloop, \character, t) \in \indexset_0)$.
\end{definition}

The next theorem is the key result about the topology $\toporbitspace$. 

\begin{theorem}\label{thm:countable-index-homeomorphism}
The map $\Psi_0 : (\honeorbitspace, \toporbitspace) \ra \C^{\indexset_0}$ is continuous and one-to-one. Consequently, $(\honeorbitspace, \toporbitspace)$ is Hausdorff. Moreover, for any closed subset $F$ of $(\honeorbitspace, \toporbitspace)$ and any $\cutoff \geq 0$, $\Psi_0(F \cap \cutoffspace{\cutoff})$ is a compact subset of $\C^{\indexset_0}$. Finally, $(\cutoffspace{\cutoff}, \toporbitspace[\cutoff])$ is a compact metrizable space for each $\cutoff \geq 0$.
\end{theorem}
\begin{proof}
The fact that $\Psi_0$ is continuous follows because each individual coordinate function $\gorbithone{A} \mapsto [W]_{\wloop, \character, t}^{1, 2}(\gorbithone{A})$ is continuous (by Lemma \ref{lemma:regwlorbit-continuous}). Lemma \ref{lemma:countable-index-set-suffices} implies that $\Psi_0$ is one-to-one. Thus, since $\Psi_0$ is a continuous one-to-one map from $(\honeorbitspace, \toporbitspace)$ into a Hausdorff space, it follows that $(\honeorbitspace, \toporbitspace)$ is Hausdorff.

Now fix a closed subset $F$ of $(\honeorbitspace, \toporbitspace)$ and $\cutoff \geq 0$. Since $\cutoffspace{\cutoff}$ is sequentially compact (by Lemma \ref{lemma:sequentially-compact-cutoff-space}) and $F$ is closed, we have that $F \cap \cutoffspace{\cutoff}$ is also sequentially compact. Let $\{x_n\}_{n \geq 1} \sse \Psi_0(F \cap \cutoffspace{\cutoff})$. Since $\C^{\indexset_0}$ is a metric space (here we use that $\indexset_0$ is countable), to show that $\Psi_0(F \cap \cutoffspace{\cutoff})$ is compact, it suffices to show that there is a subsequence $\{x_{n_k}\}_{k \geq 1}$ and a point $x \in \Psi_0(F \cap \cutoffspace{\cutoff})$ such that $x_{n_k} \ra x$. Towards this end, note that for each $n \geq 1$, there is some $\gorbithone{A_n} \in F \cap \cutoffspace{\cutoff}$ such that $x_n = \Psi_0(\gorbithone{A_n})$. By the sequential compactness of $F \cap \cutoffspace{\cutoff}$, there is a subsequence $\{\gorbithone{A_{n_k}}\}_{k \geq 1}$ and $\gorbithone{A} \in F \cap \cutoffspace{\cutoff}$ such that $\gorbithone{A_{n_k}} \ra \gorbithone{A}$. Since $\Psi_0$ is continuous, we obtain
\[\Psi_0(\gorbithone{A}) = \lim_{k \toinf} \Psi_0(\gorbithone{A_{n_k}}) = \lim_{k \toinf} x_{n_k}. \]
We can thus set $x = \Psi_0(\gorbithone{A})$.

It remains to show the final claim. Let $\cutoff \geq 0$. Note that $\Psi_0$ is a continuous bijection between $(\cutoffspace{\cutoff}, \toporbitspace[\cutoff])$ and the metric space $\Psi_0(\cutoffspace{\cutoff}) \sse \C^{\indexset_0}$. Thus it suffices to show that $(\Psi_0)^{-1} : \Psi_0(\cutoffspace{\cutoff}) \ra (\cutoffspace{\cutoff}, \toporbitspace[\cutoff])$ is continuous. (The fact that $(\cutoffspace{\cutoff}, \toporbitspace[\cutoff])$ is compact will then automatically follow, because we would then have that it is homeomorphic to a metric space, and by Lemma \ref{lemma:sequentially-compact-cutoff-space} it is sequentially compact.) This follows from the fact that $\Psi_0$ maps closed subsets of $(\cutoffspace{\cutoff}, \toporbitspace[\cutoff])$ to closed subsets of $\Psi_0(\cutoffspace{\cutoff})$, which itself follows from the fact that for any closed subset $F$ of $(\honeorbitspace, \toporbitspace)$, the image $\Psi_0(F \cap \cutoffspace{\cutoff})$ is compact.
\end{proof}

Before we state the following corollary, recall that a Lusin space is a topological space that is homeomorphic to a Borel subset of a compact metric space --- see \cite[(82.1) Definition]{RW1994}).

\begin{cor}\label{cor:orbitspace-lusin}
The space $(\honeorbitspace, \toporbitspace)$ is a Lusin space.
\end{cor}
\begin{proof}
By Theorem \ref{thm:countable-index-homeomorphism}, $(\honeorbitspace, \toporbitspace)$ is a Hausdorff space, and also, it is a countable union of compact metrizable subspaces. The desired result now follows directly by \cite[Chapter II, Corollary 2 of Theorem 5]{Schwartz1973}, which says that the countable union of Lusin subspaces of a Hausdorff space is a Lusin space.
\end{proof}

The following definition and corollary will be convenient for proofs later on.

\begin{definition}\label{def:regwl-topology-honeorbitspace}
Let $\toporbitspace^w$ be the smallest topology on $\honeorbitspace$ which makes each of the regularized Wilson loop observables $[W]_{\wloop, \character, t}^{1,2}$ continuous. For $\cutoff \geq 0$, let $\toporbitspace[\cutoff]^w$ be the subspace topology on $\cutoffspace{\cutoff}$ coming from $(\honeorbitspace, \toporbitspace^w)$. 
\end{definition}

\begin{cor}\label{cor:regwl-honeorbitspace-topology}
We have that $\toporbitspace^w \sse \toporbitspace$. Also, for all $\cutoff \geq 0$, we have that $\toporbitspace[\cutoff] = \toporbitspace[\cutoff]^w$. Finally, for all $\cutoff \geq 0$, we have that $\cutoffspace{\cutoff}$ is a closed subset of $(\honeorbitspace, \toporbitspace^w)$.
\end{cor}
\begin{proof}
The fact that $\toporbitspace^w \sse \toporbitspace$ follows by Lemma \ref{lemma:regwlorbit-continuous} and the definition of $\toporbitspace^w$. Next, let $\Psi_0$ be as in Definition \ref{def:countable-index-homeomorphism}. By the definition of $\toporbitspace^w$, the map $\Psi_0 : (\honeorbitspace, \toporbitspace) \ra \C^{\indexset_0}$ is continuous. Also, $\Psi_0$ is one-to-one by Lemma \ref{lemma:countable-index-set-suffices}. Thus $\Psi_0$ is a continuous injection from $(\honeorbitspace, \toporbitspace)$ into a Hausdorff space, and thus $(\honeorbitspace, \toporbitspace)$ is Hausdorff.

Next, we claim that for all $\cutoff \geq 0$, $\Psi_0 : (\cutoffspace{\cutoff}, \toporbitspace[\cutoff]^w) \ra \Psi_0(\cutoffspace{\cutoff})$ is a homeomorphism. Let us assume this claim for the moment. From the proof of Theorem \ref{thm:countable-index-homeomorphism}, we also have that $\Psi_0 : (\cutoffspace{\cutoff}, \toporbitspace[\cutoff]) \ra \Psi_0(\cutoffspace{\cutoff})$ is a homeomorphism. We thus obtain that the identity map from $(\cutoffspace{\cutoff}, \toporbitspace[\cutoff]^w)$ to $(\cutoffspace{\cutoff}, \toporbitspace[\cutoff])$ is a homeomorphism, and thus $\toporbitspace[\cutoff] = \toporbitspace[\cutoff]^w$. Finally, the fact that $\cutoffspace{\cutoff}$ is a closed subset of $(\honeorbitspace, \toporbitspace^w)$ follows because $\cutoffspace{\cutoff}$ is a compact subset of the Hausdorff space $(\honeorbitspace, \toporbitspace^w)$.

It remains to show the claim. We argue as in the proof of Theorem \ref{thm:countable-index-homeomorphism}. We have already noted that $\Psi_0 : (\honeorbitspace, \toporbitspace^w) \ra \Psi_0(\honeorbitspace)$ is a continuous injection. It remains to show that the inverse $\Psi_0^{-1} : \Psi_0(\honeorbitspace) \ra (\honeorbitspace, \toporbitspace^w)$ is continuous. To show this, it suffices to show that $\Psi_0$ maps closed subsets of $(\cutoffspace{\cutoff}, \toporbitspace^w)$ to closed subsets of $\Psi_0(\cutoffspace{\cutoff})$. To show this, it suffices to show that for any closed subset $F$ of $(\honeorbitspace, \toporbitspace^w)$, the image $\Psi_0(F \cap \cutoffspace{\cutoff})$ is compact. Fix a closed subset $F$ of $(\honeorbitspace, \toporbitspace)$ and $\cutoff \geq 0$. Since $(\cutoffspace{\cutoff}, \toporbitspace^w)$ is sequentially compact (by Lemma \ref{lemma:sequentially-compact-cutoff-space} and the fact that $\toporbitspace^w \sse \toporbitspace$) and $F$ is closed, we have that $F \cap \cutoffspace{\cutoff}$ is also sequentially compact. Let $\{x_n\}_{n \geq 1} \sse \Psi_0(F \cap \cutoffspace{\cutoff})$. Since $\C^{\indexset_0}$ is a metric space (here we use that $\indexset_0$ is countable), to show that $\Psi_0(F \cap \cutoffspace{\cutoff})$ is compact, it suffices to show that there is a subsequence $\{x_{n_k}\}_{k \geq 1}$ and a point $x \in \Psi_0(F \cap \cutoffspace{\cutoff})$ such that $x_{n_k} \ra x$. Towards this end, note that for each $n \geq 1$, there is some $\gorbithone{A_n} \in F \cap \cutoffspace{\cutoff}$ such that $x_n = \Psi_0(\gorbithone{A_n})$. By the sequential compactness of $F \cap \cutoffspace{\cutoff}$, there is a subsequence $\{\gorbithone{A_{n_k}}\}_{k \geq 1}$ and $\gorbithone{A} \in F \cap \cutoffspace{\cutoff}$ such that $\gorbithone{A_{n_k}} \ra \gorbithone{A}$ in $(\cutoffspace{\cutoff}, \toporbitspace^w)$. Since $\Psi_0 : (\cutoffspace{\cutoff}, \toporbitspace^w) \ra \Psi_0(\cutoffspace{\cutoff})$ is continuous, we obtain
\[\Psi_0(\gorbithone{A}) = \lim_{k \toinf} \Psi_0(\gorbithone{A_{n_k}}) = \lim_{k \toinf} x_{n_k}. \]
We can thus set $x = \Psi_0(\gorbithone{A})$. The desired result now follows.
\end{proof}

\subsection{Measure theory on the orbit space}\label{section:orbit-space-measure-theory}

Since $(\honeorbitspace, \toporbitspace)$ is a Lusin space by Corollary \ref{cor:orbitspace-lusin}, we are in the standard setting of probability theory, in that results such as Prokhorov's theorem or the portmanteau lemma hold for probability measures on $\honeorbitspace$ (see, e.g., \cite[Chapter II.6]{RW1994}). We collect in this section some measure theory results which will be needed later in Section \ref{section:nonlinear-dist-space}. First, we review some basic measure theory concepts. 

\begin{definition}\label{def:topological-space-measure-theory}
Let $(S, \ms{T}_S)$ be a topological space. Let $\mc{B}(S) := \sigma(\ms{T}_S)$ be the Borel $\sigma$-algebra of $(S, \ms{T}_S)$. Let $\mu$ be a probability measure on $(S, \mc{B}(S))$. Let $f : (S, \mc{B}(S)) \ra \R$ be a measurable function which is integrable with respect to $\mu$. We write $\mu(f) := \int_S f d\mu$. This notation will most commonly be used for $f : (S, \ms{T}_S) \ra \R$ a bounded continuous function, in which case $f$ is automatically measurable, and moreover, integrable. Given a probability measure $\mu$ on $(S, \mc{B}(S))$ and a sequence of probability measures $\{\mu_n\}_{n \geq 1}$ on $(S, \mc{B}(S))$, we say that $\mu_n$ converges weakly to $\mu$ if for any bounded continuous function $f : (S, \ms{T}_S) \ra \R$, we have that $\mu_n(f) \ra \mu(f)$. 
\end{definition}

\begin{definition}[Pushforwards]
Let $(\Omega, \mc{F})$, $(\Omega', \mc{F}')$ be two measurable spaces. Let $\mu$ be a probability measure on $(\Omega, \mc{F}$). Let $f : (\Omega, \mc{F}) \ra (\Omega', \mc{F}')$ be a measurable function. The pushforward $f_* \mu$ is the probability measure on $(\Omega', \mc{F}')$ defined by $(f_* \mu)(F) := \mu(f^{-1}(F))$ for all $F \in \mc{F}'$. We say that $f_* \mu$ is the pushforward of $\mu$ by $f$.
\end{definition}

\begin{definition}[Random variables]\label{def:random-variables}
Let $(S, \mc{G})$ be a measurable space. An $(S, \mc{G})$-valued random variable $X$ is a measurable function from some underlying measurable space $(\Omega, \mc{F})$ to $(S, \mc{G})$. We will often also write ``$S$-valued random variable" when the $\sigma$-algebra $\mc{G}$ is understood from context. We always assume that the underlying measurable space $(\Omega, \mc{F})$ is equipped with a probability measure $\p$. The law $\mu_X$ of $X$ is the pushforward of $\p$ by $X$, that is, $\mu_X := X_* \p$. In other words, $\mu_X$ is a probability measure on $(S, \mc{G})$ with $\mu_X(B) = \p(X \in B)$ for all $B \in \mc{G}$. 
If two $(S, \mc{G})$-valued random variables $X, Y$ have the same law, then we will write $X \stackrel{d}{=} Y$. 
If $\{X_n\}_{n \leq \infty}$ is a sequence of random variables with corresponding laws $\{\mu_n\}_{n \leq \infty}$, such that $\mu_n$ converges weakly to $\mu_\infty$, then we will say that $X_n$ converges in distribution to $X_\infty$, and we will write $X_n \stackrel{d}{\ra} X_\infty$.
\end{definition}

\begin{definition}[Tight family]
Let $\{\mu_\alpha\}_{\alpha \in \Gamma}$ be a family of probability measures on $(\R, \mc{B}(\R))$, where $\Gamma$ is some arbitrary index set. We say that $\{\mu_\alpha\}_{\alpha \in \Gamma}$ is a tight family if
\[ \lim_{K \toinf} \sup_{\alpha \in \Gamma} \mu_\alpha([-K, K]^c) = 0.\]
\end{definition}

Our measurable spaces will typically be of the form $(S, \mc{B}(S))$, where $(S, \ms{T}_S)$ is a topological space, and $\mc{B}(S) = \sigma(\ms{T}_S)$ is the Borel $\sigma$-algebra. Therefore we will often write ``$S$-valued random variable" instead of ``$(S, \mc{B}(S))$-valued random variable". 

\begin{definition}[Inner regularity]\label{def:inner-regular}
Let $(S, \ms{T}_S)$ be a topological space. Let $\mu$ be a probability measure on $(S, \mc{B}(S))$. We say that $\mu$ is inner regular if for all $B \in \mc{B}(S)$, we have
\[ \mu(B) = \sup\{\mu(K) : K \sse B, \text{ $K$ compact}\}. \]
(By the phrase ``$K$ compact", we mean that $K$ is a compact subset of $(S, \ms{T}_S)$.)
\end{definition}

\begin{definition}
Let $\borelorbitspace := \mc{B}(\honeorbitspace)$ denote the Borel $\sigma$-algebra of the topological space $(\honeorbitspace, \toporbitspace)$. 
\end{definition}

The following result follows directly from Prokhorov's theorem for Lusin spaces (see \cite[(83.10) Theorem]{RW1994}), combined with the fact that the subspaces $\cutoffspace{\cutoff}$ are compact for any $\cutoff \geq 0$ (by Theorem \ref{thm:countable-index-homeomorphism}).
The proof is omitted.

\begin{prop}[Prokhorov's theorem for the orbit space]\label{prop:prokhorov-orbit-space}
Let $\{\mu_n\}_{n \geq 1}$ be a sequence of probability measures on $(\honeorbitspace, \borelorbitspace)$. Suppose that the sequence of pushforwards $\{(\sym^{1,2})_* \mu_n\}_{n \geq 1}$ is a tight family on $(\R, \mc{B}(\R))$. Note this is equivalent to 
\beq\label{eq:tightness-condition-orbit-space} \lim_{\cutoff \uparrow \infty} \sup_{n \geq 1} \mu_n(\cutoffspace{\cutoff}^c) = 0. \eeq
Then there exists a subsequence $\{\mu_{n_k}\}_{k \geq 1}$ and a probability measure $\mu$ on the space $(\honeorbitspace, \borelorbitspace)$ such that $\mu_{n_k}$ converges weakly to $\mu$.
\end{prop}
\begin{remark}\label{remark:sym-measurable}
Note that $\sym^{1,2} : (\honeorbitspace, \borelorbitspace) \ra \R$ is measurable, since for all $\cutoff \geq 0$, we have that $(\sym^{1,2})^{-1}((-\infty, \cutoff]) = \cutoffspace{\cutoff}$, which is a closed subset of $(\honeorbitspace, \toporbitspace)$ (by Lemma \ref{lemma:sequentially-compact-cutoff-space}), and thus, is an element of $\borelorbitspace$.
\end{remark}





Later on, we will need some results about probability measures on spaces which are products of $\honeorbitspace$. We introduce those results now.

\begin{definition}\label{def:orbit-space-products}
Let $I$ be a countable index set. Let $\toporbitspace^I$ be the product topology on $(\honeorbitspace)^I$ coming from the topology $\toporbitspace$ on $\honeorbitspace$. Let $\mc{B}((\honeorbitspace)^I)$ be the Borel $\sigma$-algebra of $((\honeorbitspace)^I, \toporbitspace^I)$. Let $\borelorbitspace^I$ be the product $\sigma$-algebra on $(\honeorbitspace)^I$ coming from the $\sigma$-algebra $\borelorbitspace$ on $\honeorbitspace$.
\end{definition}

The following result comes from the fact that for countable products of Lusin spaces, the product and Borel $\sigma$-algebras coincide. This follows directly from \cite[Remark, Page 105]{Schwartz1973}, and thus the proof is omitted.

\begin{lemma}\label{lemma:borel-equals-product}
Let $I$ be a countable index set. Then $\mc{B}((\honeorbitspace)^I) = \borelorbitspace^I$.
\end{lemma}

Recall Definition \ref{def:inner-regular}, the definition of inner regular probability measures.

\begin{lemma}\label{lemma:finite-product-orbit-space-inner-regular}
Let $I$ be a countable index set. Every probability measure $\mu$ on the space $((\honeorbitspace)^I, \borelorbitspace^I)$ is inner regular.
\end{lemma}
\begin{proof}
By Lemma \ref{lemma:borel-equals-product}, we have that $\borelorbitspace^I = \mc{B}((\honeorbitspace)^I)$. Since the space $(\honeorbitspace, \toporbitspace)$ is a Lusin space (by Corollary \ref{cor:orbitspace-lusin}), and since countable products of Lusin spaces are Lusin spaces (by \cite[Chapter II, Lemma 4]{Schwartz1973}), it follows that $((\honeorbitspace)^I, \toporbitspace^I)$ is a Lusin space. The desired result now follows because probability measures on Lusin spaces (equipped with the Borel $\sigma$-algebra) are always inner regular (see \cite[(82.4) Lemma]{RW1994}).
\end{proof}

Our next result states that the finite dimensional distributions of the regularized Wilson loop observables determine the law of an $\honeorbitspace$-valued random variable.

\begin{lemma}\label{lemma:orbit-space-regwl-determine-law}
Let $\rgclass_1, \rgclass_2$ be $\honeorbitspace$-valued random variables. Suppose that for any finite collection $\wloop_i, \character_i, t_i$, $1 \leq i \leq k$, where $\wloop_i : [0, 1] \ra \threetorus$ is a piecewise $C^1$ loop, $\character_i$ is a character of $\liegroup$, and $t_i > 0$, we have that
\[ ([W]_{\wloop_i, \character_i, t_i}^{1,2}(\rgclass_1), 1 \leq i \leq k) \stackrel{d}{=} ([W]_{\wloop_i, \character_i, t_i}^{1,2}(\rgclass_2), 1 \leq i \leq k).\]
(The above identity is interpreted as equality in distribution of two $\C^k$-valued random variables.) Then $\rgclass_1 \stackrel{d}{=} \rgclass_2$.
\end{lemma}
\begin{proof}
Let $\indexset_0$ be as in Lemma \ref{lemma:countable-index-set-suffices}, and let $\Psi_0$ be as in Definition \ref{def:countable-index-homeomorphism}. We have by the assumption in the lemma statement that
\[ \Psi_0(\rgclass_1) \stackrel{d}{=} \Psi_0(\rgclass_2). \]
Here both sides of the identity are $\C^{\indexset_0}$-valued random variables; the identity follows because the Borel and product $\sigma$-algebras on $\C^{\indexset_0}$ agree (see, e.g., \cite[Lemma 1.2]{K2002}), so that equality of all finite dimensional distributions implies equality in law. We claim that for all $B \in \borelorbitspace$, we have that $\Psi_0(B)$ is a Borel measurable subset of $\C^{\indexset_0}$. Given this claim, we have for all $B \in \borelorbitspace$ (using also that $\Psi_0$ is injective by Theorem \ref{thm:countable-index-homeomorphism})
\[ \p(\rgclass_1 \in B) = \p(\Psi_0(\rgclass_1) \in \Psi_0(B)) = \p(\Psi_0(\rgclass_2) \in \Psi_0(B)) = \p(\rgclass_2 \in B), \]
and thus $\rgclass_1 \stackrel{d}{=} \rgclass_2$. It remains to show the claim. Since $\Psi_0$ is injective, it suffices to show the claim for $F$ a closed subset of $(\honeorbitspace, \toporbitspace)$. Note that $\Psi_0(F) = \bigcup_{m=1}^\infty \Psi_0(F \cap \cutoffspace{m})$. Thus, it suffices to show that for all $m \geq 1$, $\Psi_0(F \cap \cutoffspace{m})$ is a Borel-measurable subset of $\C^{\indexset_0}$. But this follows by Theorem~\ref{thm:countable-index-homeomorphism}, which gives that $\Psi_0(F \cap \cutoffspace{m})$ is in fact a compact, and thus closed, subset of $\C^{\indexset_0}$.
\end{proof}

Recall the topologies $\toporbitspace^w$ and $\toporbitspace[\cutoff]^w$ from Definition \ref{def:regwl-topology-honeorbitspace}.

\begin{lemma}\label{lemma:orbit-space-borel-sigma-algebras-coincide}
We have that $\borelorbitspace = \sigma(\toporbitspace^w)$. That is, the Borel $\sigma$-algebras of the topological spaces $(\honeorbitspace, \toporbitspace)$ and $(\honeorbitspace, \toporbitspace^w)$ are the same.
\end{lemma}
\begin{proof}
Since $\toporbitspace^w \sse \toporbitspace$ (by Corollary \ref{cor:regwl-honeorbitspace-topology}), it suffices to show that $\borelorbitspace \sse \sigma(\toporbitspace^w)$. Thus let $B \in \borelorbitspace$. Since $B = \bigcup_{m = 1}^\infty (B \cap \cutoffspace{m})$, it suffices to show that for all $m \geq 1$, we have that $B \cap \cutoffspace{m} \in \sigma(\toporbitspace^w)$. Fix $m \geq 1$. By a standard measure theory exercise, we have that $\sigma(\toporbitspace[m]) = \{F \cap \cutoffspace{m} : F \in \borelorbitspace\}$, and similarly $\sigma(\toporbitspace[m]^w) = \{F \cap \cutoffspace{m} : F \in \sigma(\toporbitspace^w)\}$. By Corollary \ref{cor:regwl-honeorbitspace-topology}, we have that $\sigma(\toporbitspace[m]) = \sigma(\toporbitspace[m]^w)$. We thus have that $B \cap \cutoffspace{m} \in \sigma(\toporbitspace[m]^w)$. Now observe that every element of $\sigma(\toporbitspace[m]^w)$ is of the form $F \cap \cutoffspace{m}$ for some $F \in \sigma(\toporbitspace^w)$. Thus, since $\cutoffspace{m} \in \sigma(\toporbitspace^w)$ (recall $\cutoffspace{m}$ is closed by Corollary \ref{cor:regwl-honeorbitspace-topology}), we have that every element of $\sigma(\toporbitspace[m]^w)$ is an element of $\sigma(\toporbitspace^w)$. Thus, $B \cap \cutoffspace{m} \in \sigma(\toporbitspace^w)$, as desired.
\end{proof}

\subsection{The nonlinear distribution space}\label{section:nonlinear-dist-space}

In this section, we will define the nonlinear distribution space $\nonlineardistspace^{1,2}$, which is the analogue of $\nonlineardistspace$, where we use $\honeorbitspace$ in place of $\orbitspace$. We will then prove the analogue of Theorem \ref{mainthm} for $\nonlineardistspace^{1,2}$ (see Theorem \ref{thm:nonlinear-dist-space-tightness}). Using this analogue, we will then be able to prove Theorem \ref{mainthm} in Section \ref{section:completing-the-proof}. We begin with some preliminary results that are needed before defining $\nonlineardistspace^{1,2}$.~Recall the Yang--Mills heat semigroup $(\ymsemigroup_t^{1,2}, t \geq 0)$ from Definition \ref{def:ym-semigroup}.


\begin{lemma}\label{lemma:ym-semigroup-continuous}
For any $t \geq 0$, the map $\ymsemigroup_t^{1,2} : (\honeorbitspace, \toporbitspace) \ra (\honeorbitspace, \toporbitspace)$ is continuous.
\end{lemma}
\begin{proof}
By Lemma \ref{lemma:sequential-convergence-topology-properties} and the definition of $\toporbitspace$, to show that $\ymsemigroup_t^{1,2}$ is continuous, it suffices to show that if $\{A_{0, n}\}_{n \leq \infty} \sse \honeconnspace$ is such that $A_{0, n} \weakarrow A_{0, \infty}$, then $\ymsemigroup_t^{1,2}(\gorbithone{A_{0, n}}) \weakarrow_q \ymsemigroup_t^{1,2}(\gorbithone{A_{0, \infty}})$. Towards this end, for $n \leq \infty$, let $B_n$ be the solution to \eqref{eq:ZDDS} on $[0, \timefn(\|A_{0, n}\|_{H^1}))$ with initial data $B_n(0) = A_{0, n}$ (as given by Proposition \ref{prop:ZDDS-local-existence}, and where $\timefn$ is as in that proposition). By Lemma \ref{lemma:zdds-solutions-weak-continuity}, there is $T > 0$ such that for all $t \in (0, T]$, we have that $B_n(t) \ra B_\infty(t)$ in $H^1$. By Lemma \ref{lemma:ym-zdds-solution-gauge-transformation}, we have that for all $t \in (0, T]$, $\gorbithone{B_n(t)} = \ymsemigroup_t^{1,2}(\gorbithone{A_{0, n}})$. This shows that $\ymsemigroup_t^{1,2}(\gorbithone{A_{0, n}}) \weakarrow_q \ymsemigroup_t^{1,2}(\gorbithone{A_{0, \infty}})$ for all $t \in (0, T]$. To obtain the convergence for $t \geq T$, fix $t_0 \in (0, T]$. For $n \leq \infty$, let $A_n$ be the solution to \eqref{eq:YM} with initial data $A_n(0) = B_n(t_0)$. Note that $B_n(t_0)$ is smooth (by Proposition \ref{prop:ZDDS-local-existence}). Since $A_n(0) \ra A_\infty(0)$ in $H^1$, we have by Lemma \ref{lemma:ym-continuity-in-initial-data-smooth-case} that $A_n(t) \ra A_\infty(t)$ in $H^1$ for all $t \geq 0$. By definition, we have that $\gorbithone{A_n(t)} = \ymsemigroup_t^{1,2}(\gorbithone{B_n(t_0)}) = \ymsemigroup_t^{1,2}(\ymsemigroup_{t_0}^{1,2}(\gorbithone{A_{0, n}}) = \ymsemigroup_{t + t_0}^{1,2}(\gorbithone{A_{0, n}})$. Combining the previous few observations, we obtain that $\ymsemigroup_t^{1,2}(\gorbithone{A_{0, n}}) \weakarrow_q \ymsemigroup_t^{1,2}(\gorbithone{A_{0, \infty}})$ for all $t \geq t_0$, and thus the desired result follows.
\end{proof}

\begin{lemma}\label{lemma:ym-semigroup-continuous-in-time}
For any $\gorbithone{A_0} \in \honeorbitspace$, the map  $t \mapsto \ymsemigroup_t^{1,2}(\gorbithone{A_0})$ from $[0, \infty)$ into $(\honeorbitspace, \toporbitspace)$ is continuous.
\end{lemma}
\begin{proof}
Let $A_0 \in \honeconnspace$, and let $A$ be the solution to \eqref{eq:YM} on $[0, \infty)$ with initial data $A(0) = A_0$ (as given by Theorem \ref{thm:YM-global-existence}). By Definition \ref{def:YM-solution-def}, we have that $t \mapsto A(t)$ is a continuous function from $[0, \infty)$ into $\connspace^{1,2}$ (when $\honeconnspace$ is equipped with the $H^1$ norm topology). This implies that $t \mapsto \gorbithone{A(t)}$ is continuous from $[0, \infty)$ into $(\honeorbitspace, \toporbitspace)$. To finish, note by definition, we have that $\ymsemigroup_t^{1,2}(\gorbithone{A_0}) = \gorbithone{A(t)}$ for all $t \geq 0$.
\end{proof}

We now define the nonlinear distribution space $\nonlineardistspace^{1,2}$.

\begin{definition}[Nonlinear distribution space]\label{def:nonlinear-dist-space}
Let $\nonlineardistspace^{1,2}$ be the set of functions $X : (0, \infty) \ra \honeorbitspace$ such that for all $0 < s \leq t$, we have that $X(t) = \ymsemigroup_{t-s}^{1,2}(X(s))$. Typical elements of $\nonlineardistspace^{1,2}$ will be denoted $X$. By Lemma \ref{lemma:ym-semigroup-continuous-in-time}, elements of $\nonlineardistspace^{1,2}$ are continuous functions $(0, \infty) \ra (\honeorbitspace, \toporbitspace)$. For $t > 0$, define the evaluation map $\evalmap_t : \nonlineardistspace^{1,2} \ra \honeorbitspace$ by $X \mapsto X(t)$. Let $\topnonlineardistspace$ be the weakest topology on $\nonlineardistspace^{1,2}$ which makes each of the $\evalmap_t$, $t > 0$ continuous (where $\honeorbitspace$ is equipped with the topology $\toporbitspace$). Let $\borelnonlineardistspace := \mc{B}(\nonlineardistspace^{1,2})$ be the Borel $\sigma$-algebra of $(\nonlineardistspace^{1,2}, \topnonlineardistspace)$.
\end{definition}

The space $\honeorbitspace$ is naturally embedded in $\nonlineardistspace^{1,2}$, as follows.

\begin{definition}\label{def:embedding-orbit-space-nonlinear-dist-space}
Define the map $\iota_{\ymsemigroup^{1,2}} : (\honeorbitspace, \toporbitspace )\ra (\nonlineardistspace^{1,2}, \topnonlineardistspace)$ by letting $\iota_{\ymsemigroup^{1,2}}(\gorbithone{A})(t) := \ymsemigroup_t^{1,2}(\gorbithone{A})$ for each $t > 0$. Note that $\iota_{\ymsemigroup^{1,2}}$ is injective (by Lemma \ref{lemma:ym-semigroup-continuous-in-time}) and continuous (by Lemma \ref{lemma:ym-semigroup-continuous}). In a slight abuse of notation, given $\gorbithone{A} \in \honeorbitspace$, we will often write $\gorbithone{A}(\cdot)$ instead of $\iota_{\ymsemigroup^{1,2}}(\gorbithone{A})$ and $\gorbithone{A}(t)$ instead of $\iota_{\ymsemigroup^{1,2}}(\gorbithone{A})(t)$.
\end{definition}

It turns out that $(\nonlineardistspace^{1,2}, \topnonlineardistspace)$ is a Lusin space. The proof of this is deferred until later in this section.

\begin{lemma}\label{lemma:nldistspace-lusin}
The space $(\nonlineardistspace^{1,2}, \topnonlineardistspace)$ is a Lusin space.
\end{lemma}

We now state the following result, which is the analogue of Theorem \ref{mainthm} for the space $\nonlineardistspace^{1,2}$.

\begin{theorem}\label{thm:nonlinear-dist-space-tightness}
Let $\{\rvnldistspace_n\}_{n \geq 1}$ be a sequence of $\nonlineardistspace^{1,2}$-valued random variables. Suppose that for all $t > 0$, $\{\sym^{1,2}(\rvnldistspace_n(t))\}_{n \geq 1}$ is a tight sequence of $\R$-valued random variables. Then there exists a subsequence $\{\rvnldistspace_{n_k}\}_{k \geq 1}$, and a probability measure $\mu$ on $(\nonlineardistspace^{1,2}, \borelnonlineardistspace)$, such that the laws of $\rvnldistspace_{n_k}$ converge weakly to $\mu$. That is, for all bounded continuous functions $f : (\nonlineardistspace^{1,2}, \topnonlineardistspace) \ra \R$, we have that $\E f(\rvnldistspace_{n_k}) \ra  \int_{\nonlineardistspace^{1,2}} f d\mu$.
\end{theorem}

In the rest of this section, we work towards the proof of Theorem \ref{thm:nonlinear-dist-space-tightness}. 
We will prove Theorem \ref{thm:nonlinear-dist-space-tightness} by first proving the analogous statement on a slightly more convenient space than $\nonlineardistspace^{1,2}$, which we now define.

\begin{definition}\label{def:countable-nonlinear-dist-space}
Let $I_0 := \Q \cap (0, \infty)$. Let $\nonlineardistspace_0^{1,2}$ be the set of functions $X : I_0 \ra \honeorbitspace$ such that for each $s, t \in I_0$, $s \leq t$, we have that $X(t) = \ymsemigroup_{t-s}^{1,2}(X(s))$. Naturally, $\nonlineardistspace_0^{1,2}$ can be thought of as obtained by the restrictions of the elements of $\nonlineardistspace^{1,2}$ to the index set $I_0$. In a slight abuse of notation, for $t \in I_0$, let $\evalmap_t : \nonlineardistspace_0^{1,2} \ra \honeorbitspace$ be the evaluation map which maps $X \mapsto X(t)$. (We trust that the context should make clear whether the domain of $\evalmap_t$ is $\nonlineardistspace^{1,2}$ or $\nonlineardistspace_0^{1,2}$.) Let $\ms{T}_{\nonlineardistspace_0^{1,2}}$ be the weakest topology on $\nonlineardistspace_0^{1,2}$ which makes each of the evaluation maps $\evalmap_t, t \in I_0$ continuous (where $\honeorbitspace$ is equipped with the topology $\toporbitspace$). Let $\mc{B}_{\nonlineardistspace_0^{1,2}}$ be the Borel $\sigma$-algebra of $(\nonlineardistspace_0^{1,2}, \ms{T}_{\nonlineardistspace_0^{1,2}})$. Define the natural inclusion map $\inclnldistpace : \nonlineardistspace_0^{1,2} \hookrightarrow \nonlineardistspace^{1,2}$ as follows. For $X \in \nonlineardistspace_0^{1,2}$ and $t \in (0, \infty)$, take $t_0 \leq t$, $t_0 \in I_0$, and let $\inclnldistpace(X)(t) := \ymsemigroup_{t-t_0}^{1,2}(X(t_0))$. This is well-defined by the semigroup property of $(\ymsemigroup_t^{1,2}, t \geq 0)$.
\end{definition}

\begin{lemma}\label{lemma:homeomorphism-of-nonlinear-dist-space}
The map $\inclnldistpace : (\nonlineardistspace_0^{1,2}, \ms{T}_{\nonlineardistspace_0^{1,2}}) \ra (\nonlineardistspace^{1,2}, \topnonlineardistspace)$ is a homeomorphism.
\end{lemma}
\begin{proof}
Note that $\inclnldistpace$ is bijective, and note that the inverse $\inclnldistpace^{-1} : \nonlineardistspace^{1,2} \ra \nonlineardistspace_0^{1,2}$ is the obvious restriction map, i.e. $\inclnldistpace^{-1}(X)(t) = X(t)$ for all $X \in \nonlineardistspace^{1,2}$, $t \in I_0$. To see why $\inclnldistpace^{-1}$ is continuous, it suffices to show that if $\{X_\alpha\}$ is a net in $\nonlineardistspace^{1,2}$ such that $X_\alpha \ra X \in \nonlineardistspace^{1,2}$, then $\inclnldistpace^{-1}(X_\alpha) \ra \inclnldistpace^{-1}(X)$. By the definition of $\ms{T}_{\nonlineardistspace_0^{1,2}}$ as the smallest topology generated by $\{\evalmap_t\}_{t \in I_0}$, convergence of a net $\{Y_\alpha\} \sse \nonlineardistspace_0^{1,2}$ to some $Y \in \nonlineardistspace_0^{1,2}$ is equivalent to the convergence of $Y_\alpha(t) \ra Y(t)$ in $(\honeorbitspace, \toporbitspace)$ for all $t \in I_0$. Thus to show that $\inclnldistpace^{-1}(X_\alpha) \ra \inclnldistpace^{-1}(X)$, it suffices to show that for all $t \in I_0$, we have that $\inclnldistpace^{-1}(X_\alpha)(t) \ra \inclnldistpace(X)(t)$. But this follows since $\inclnldistpace^{-1}(X_\alpha)(t) = X_\alpha(t) \ra X(t) = \inclnldistpace^{-1}(X)(t)$, where the limit follows since $\evalmap_t$ is continuous, so that $X_\alpha(t) = \evalmap_t(X_\alpha) \ra \evalmap_t(X) = X(t)$.

It remains to show that $\inclnldistpace$ is continuous. Towards this end, let $\{X_\alpha\}$ be a net in $\nonlineardistspace_0^{1,2}$ such that $X_\alpha \ra X \in \nonlineardistspace_0^{1,2}$. As previously noted, this convergence is equivalent to the convergence of $X_\alpha(s) \ra X(s)$ in $(\honeorbitspace, \toporbitspace)$ for all $s \in I_0$. Note also that (similar to convergence in $\ms{T}_{\nonlineardistspace_0^{1,2}}$) by the definition of $\ms{T}_{\nonlineardistspace^{1,2}}$, convergence of a net $\{Y_\alpha\} \sse \nonlineardistspace^{1,2}$ to some $Y \in \nonlineardistspace^{1,2}$ is equivalent to the convergence of $Y_\alpha(t) \ra Y(t)$ in $(\honeorbitspace, \toporbitspace)$ for all $t > 0$. Thus to show $\inclnldistpace(X_\alpha) \ra \inclnldistpace(X)$, we need to show that for all $t > 0$, $\inclnldistpace(X_\alpha)(t) \ra \inclnldistpace(X)(t)$ in $(\honeorbitspace, \toporbitspace)$. Let $t > 0$, and take $s \in I_0$, $s \leq t$. Then we have that $\inclnldistpace(X_\alpha)(t) = \ymsemigroup_{t-s}^{1,2}(X_\alpha(s))$ and $\inclnldistpace(X)(t) = \ymsemigroup_{t-s}^{1,2}(X(s))$. Using that $X_\alpha(s) \ra X(s)$ in $(\honeorbitspace, \toporbitspace)$ and that $\ymsemigroup_{t-s}^{1,2}$ is continuous (by Lemma \ref{lemma:ym-semigroup-continuous}), we obtain $\inclnldistpace(X_\alpha)(t) \ra \inclnldistpace(X)(t)$, as desired.
\end{proof}

\begin{lemma}\label{lemma:open-sets-nonlinear-dist-space-representation}
Any open set $O \in \topnonlineardistspace$ can be written as a countable union
\[ O = \bigcup_{t \in I_0} (\evalmap_t)^{-1}(U_t), \]
where for each $t \in I_0$, $U_t \in \toporbitspace$ is an open set of $\honeorbitspace$. The same is true for open sets $O \in \ms{T}_{\nonlineardistspace_0^{1,2}}$.
\end{lemma}
\begin{proof}
We will prove this for $O \in \topnonlineardistspace$. The proof for $O \in \ms{T}_{\nonlineardistspace_0^{1,2}}$ is essentially the same. By definition of $\topnonlineardistspace$, $O$ may be written in the following form. There is an index set $\Gamma$, and for each $\alpha \in \Gamma$, there are finite sets $\{t_1^\alpha, \ldots, t_{k_\alpha}^\alpha\} \sse (0, \infty)$ and $\{O_1^\alpha, \ldots, O_{k_\alpha}^\alpha\} \sse \toporbitspace$, such that
\[ O = \bigcup_{\alpha \in \Gamma} \bigcap_{i=1}^{k_\alpha} (\evalmap_{t_i^\alpha})^{-1}(O_i^\alpha). \]
Observe that for each $0  < s \leq t$, we have that $\evalmap_t = \ymsemigroup_{t-s}^{1,2} \circ \evalmap_s$. Thus for each $\alpha \in \Gamma$, fix $t_0^\alpha \in I_0$, $t_0^\alpha \leq \min_{1 \leq i \leq k_\alpha} t_i^\alpha$. We then have for each $1 \leq i \leq k_\alpha$,
\[ (\evalmap_{t_i^\alpha})^{-1}(O_i^\alpha) = (\evalmap_{t_0^\alpha})^{-1} ((\ymsemigroup_{t_i^\alpha - t_0^\alpha}^{1,2})^{-1}(O_i^\alpha)).\]
Thus for each $\alpha \in \Gamma$, let
\[ O^\alpha := \bigcap_{i=1}^{k_\alpha}(\ymsemigroup_{t_i^\alpha - t_0^\alpha}^{1,2})^{-1}(O_i^\alpha). \]
Since $\ymsemigroup_t^{1,2}$ is continuous for all $t \geq 0$ (by Lemma \ref{lemma:ym-semigroup-continuous}), we have that $O^\alpha \in \toporbitspace$.
We may write
\[ O = \bigcup_{\alpha \in \Gamma} (\evalmap_{t_0^\alpha})^{-1}(O^\alpha).\]
Now, for each $t \in I_0$, let $\alpha(t) := \{\alpha \in \Gamma : t_0^\alpha = t\}$, and let $U_t := \bigcup_{\alpha \in \alpha(t)} O^\alpha \in \toporbitspace$. We then have
\[ O = \bigcup_{t \in I_0} \bigcup_{\alpha \in \alpha(t)} (\evalmap_t)^{-1}(O^\alpha) = \bigcup_{t \in I_0} (\evalmap_t)^{-1}(U_t), \]
as desired.
\end{proof}

In the following, recall the definitions of $(\honeorbitspace)^I$, $\toporbitspace^{I}$, $\mc{B}((\honeorbitspace)^I)$, and $\borelorbitspace^{I}$ for a countable index set $I$ (see Definition \ref{def:orbit-space-products}).
Note that we can think of $\nonlineardistspace_0^{1,2}$ as a subset of $(\honeorbitspace)^{I_0}$. This leads to the next two lemmas.

\begin{lemma}\label{lemma:countable-nonlinear-dist-space-measurable}
For every $s, t \in I_0$, $s < t$, the set
\[ \{X \in (\honeorbitspace)^{I_0} : X(t) = \ymsemigroup_{t-s}^{1,2}(X(s))\} \in \borelorbitspace^{I_0}. \]
Consequently, as $\nonlineardistspace_0^{1,2}$ is a countable intersection (over all $s, t \in I_0$, $s < t$) of the above sets, we have that $\nonlineardistspace_0^{1,2} \in \borelorbitspace^{I_0}$.
\end{lemma}
\begin{proof} 
First, we claim that for any $s, t \in I_0$, $s < t$, the set
\[ \{X \in (\honeorbitspace)^{\{s, t\}} : X(t) = \ymsemigroup_{t-s}^{1,2}(X(s))\}\]
is a closed subset of $((\honeorbitspace)^{\{s, t\}}, \toporbitspace^{\{s, t\}})$. To see this, let $\{X_\alpha\}$ be a net contained in the above set, such that $X_\alpha \ra X \in (\honeorbitspace)^{\{s, t\}}$. This is equivalent to $X_\alpha(s) \ra X(s)$ and $X_\alpha(t) \ra X(t)$. Note that $X_\alpha(t) = \ymsemigroup_{t-s}^{1,2}(X_\alpha(s))$. Thus since $\ymsemigroup_{t-s}^{1,2}$ is continuous (by Lemma \ref{lemma:ym-semigroup-continuous}), we have that
\[ X(t) = \lim_\alpha X_\alpha(t) = \ymsemigroup_{t-s}^{1,2}(\lim_\alpha X_\alpha(s)) = \ymsemigroup_{t-s}^{1,2}(X(s)), \]
and the desired closedness follows. Now using that $\mc{B}((\honeorbitspace)^{\{s, t\}}) = \borelorbitspace^{\{s, t\}} $ (by Lemma \ref{lemma:borel-equals-product}), and the fact that $\borelorbitspace^{I_0}$ contains all sets of the form
\[ \{X \in (\honeorbitspace)^{I_0} : (X(s), X(t)) \in B\}, ~~ B \in \borelorbitspace^{\{s, t\}}, \]
the desired result now follows.
\end{proof}

\begin{lemma}\label{lemma:borel-equals-infinite-product}
Let $\borelorbitspace^{I_0} |_{\nonlineardistspace_0^{1,2}} := \{B \cap \nonlineardistspace_0^{1,2} : B \in \borelorbitspace^{I_0}\}$ be the restriction of the product $\sigma$-algebra $\borelorbitspace^{I_0}$ to $\nonlineardistspace_0^{1,2}$. Then $\borelorbitspace^{I_0} |_{\nonlineardistspace_0^{1,2}} = \mc{B}_{\nonlineardistspace_0^{1,2}}$.
\end{lemma}
\begin{proof}
First, by definition, we have that $\borelorbitspace^{I_0}$ is generated by sets of the form $\{X \in (\honeorbitspace)^{I_0} : X(t) \in B\}$, for $t \in I_0$ and $B \in \borelorbitspace$. Since $\borelorbitspace = \sigma(\toporbitspace)$, by a standard measure theory exercise, we can then obtain that $\borelorbitspace^{I_0}$ is generated by sets of the form $\{X \in (\honeorbitspace)^{I_0} : X(t) \in U\}$ for $t \in I_0$ and $U \in \toporbitspace$. The intersection of this set with $\nonlineardistspace_0^{1,2}$ is exactly $(\evalmap_t)^{-1}(U)$, and thus we can obtain that $\borelorbitspace^{I_0} |_{\nonlineardistspace_0^{1,2}}$ is generated by $(\evalmap_t)^{-1}(U)$ for $t \in I_0, U \in \toporbitspace$. By definition, these sets are open, and thus they are elements of $\mc{B}_{\nonlineardistspace_0^{1,2}}$. It follows that $\borelorbitspace^{I_0} |_{\nonlineardistspace_0^{1,2}} \sse \mc{B}_{\nonlineardistspace_0^{1,2}}$. For the reverse inclusion, it follows by Lemma \ref{lemma:open-sets-nonlinear-dist-space-representation} that any open set $O \in \ms{T}_{\nonlineardistspace_0^{1,2}}$ may be written as a countable union of elements of $\borelorbitspace^{I_0} |_{\nonlineardistspace_0^{1,2}}$. From this, the reverse inclusion follows.
\end{proof}

Using Lemma \ref{lemma:countable-nonlinear-dist-space-measurable}, we can prove that $(\nonlineardistspace^{1,2}, \topnonlineardistspace)$ is a Lusin space (Lemma \ref{lemma:nldistspace-lusin}).

\begin{proof}[Proof of Lemma \ref{lemma:nldistspace-lusin}]
First, note that as in the proof of Lemma \ref{lemma:finite-product-orbit-space-inner-regular}, we have that $((\honeorbitspace)^{I_0}, \toporbitspace^{I_0})$ is a Lusin space. Now by Lemma \ref{lemma:countable-nonlinear-dist-space-measurable}, and the fact that $\borelorbitspace^{I_0} = \mc{B}((\honeorbitspace)^{I_0})$ (by Lemma \ref{lemma:borel-equals-product}), we have that $\nonlineardistspace_0^{1,2}$ is a Borel subset of $(\honeorbitspace)^{I_0}$. Since Borel subsets of Lusin spaces are Lusin spaces (by \cite[Chapter II, Theorem 2]{Schwartz1973}), it follows that $(\nonlineardistspace_0^{1,2}, \ms{T}_{\nonlineardistspace_0^{1,2}})$ is a Lusin space (note the subspace topology on $\nonlineardistspace_0^{1,2}$ induced by $\toporbitspace^{I_0}$ is exactly $\ms{T}_{\nonlineardistspace_0^{1,2}}$). Since this space is homeomorphic to $(\nonlineardistspace^{1,2}, \topnonlineardistspace)$ (by Lemma \ref{lemma:homeomorphism-of-nonlinear-dist-space}), the desired result now follows.
\end{proof}

Using Lemmas \ref{lemma:nldistspace-lusin} and \ref{lemma:open-sets-nonlinear-dist-space-representation}, we prove the following lemma, which gives a sufficient condition for weak convergence of probability measures on $(\nonlineardistspace^{1,2}, \borelnonlineardistspace)$.

\begin{lemma}\label{lemma:weak-convergence-nonlinear-dist-space}
Let $\{\mu_n\}_{n \leq \infty}$ be probability measures on $(\nonlineardistspace^{1,2}, \borelnonlineardistspace)$. Suppose that for all $t > 0$, $(\evalmap_t)_* \mu_n$ converges weakly to $(\evalmap_t)_* \mu_\infty$ (as probability measures on $(\honeorbitspace, \borelorbitspace)$). Then $\mu_n$ converges weakly to $\mu_\infty$.
\end{lemma}
\begin{proof}
We will show that for any open set $O \in \topnonlineardistspace$,
\[ \liminf_n \mu_n(O) \geq \mu_\infty(O).\]
Since $(\nonlineardistspace^{1,2}, \topnonlineardistspace)$ is a Lusin space (by Lemma \ref{lemma:nldistspace-lusin}), the weak convergence of $\mu_n$ to $\mu$ will then follow by the portmanteau lemma (see, e.g., \cite[(83.4) Theorem]{RW1994}). Thus, let $O \in \topnonlineardistspace$. By Lemma \ref{lemma:open-sets-nonlinear-dist-space-representation}, there is an indexed family $\{U_t, t \in I_0\} \sse \toporbitspace$ of open subsets of $\honeorbitspace$ such that
\[ O = \bigcup_{t \in I_0} (\evalmap_t)^{-1}(U_t).\]
Let $I_{0, 1} \sse I_{0, 2} \sse \cdots $ be a nested family of finite sets which increases to $I_0$. For $m \geq 1$, let $O_m := \bigcup_{t \in I_{0, m}} (\evalmap_t)^{-1}(U_t)$. Then $O_m \uparrow O$. Thus it suffices to show that for each $m \geq 1$,
\[ \liminf_n \mu_n(O) \geq \mu_\infty(O_m). \]
Since $\mu_n(O) \geq \mu_n(O_m)$, it suffices to show that for each $m \geq 1$,
\[ \liminf_n \mu_n(O_m) \geq \mu_\infty(O_m). \]
Thus fix $m \geq 1$. Since $I_{0, m}$ is finite, we can in fact express $O_m = (\evalmap_{t_m})^{-1}(U_m)$ for some $t_m > 0$ and $U_m \in \toporbitspace$ (take $t_m = \min_{t \in I_{0, m}} t$, as we did in the proof of Lemma \ref{lemma:open-sets-nonlinear-dist-space-representation}). Since $O_m = (\evalmap_{t_m})^{-1}(U_m)$, we have $\mu_n(O_m) = ((\evalmap_{t_m})_* \mu_n)(U_m)$ for all $n \leq \infty$. By assumption, $(\evalmap_{t_m})_* \mu_n$ converges weakly to $(\evalmap_{t_m})_* \mu_\infty$. Since $(\honeorbitspace, \toporbitspace)$ is a Lusin space (by Corollary \ref{cor:orbitspace-lusin}), the desired inequality now follows by the portmanteau lemma (see, e.g., \cite[(83.4) Theorem]{RW1994}).
\end{proof}

We will also need the following generalization of Kolmogorov's extension theorem (also called the Kolmogorov--Bochner theorem) due to Rao \cite[Theorem 2.2]{Rao1971}.

\begin{theorem}[Theorem 2.2 of \cite{Rao1971}]\label{thm:kolmogorov-bochner}
Let $T$ be any index set and $D$ the class of all its finite subsets, directed by inclusion. Let $(\Omega_t, \Sigma_t)_{t \in T}$ be a collection of measurable spaces where $\Omega_t$ is a topological space and $\Sigma_t$ is a $\sigma$-algebra containing all the compact subsets of $\Omega_t$. For $I \in D$, let $\Omega_I := \prod_{t \in I} \Omega_t$, $\Sigma_I := \times_{t \in I} \Sigma_t$. Also let $\Omega_T := \prod_{t \in T} \Omega_t$, $\Sigma_T := \times_{t \in T} \Sigma_t$. For $I \in D$, let $\pi_I : \Omega_T \ra \Omega_I$ be the coordinate projection, and similarly for $I_1, I_2 \in D$, $I_1 \sse I_2$, let $\pi_{I_1 I_2} : \Omega_{I_2} \ra \Omega_{I_1}$ be the coordinate projection. Suppose that for all $I \in D$, $\mu_I$ is an inner regular probability measure on $(\Omega_I, \Sigma_I)$ (here $\Omega_I$ is endowed with its product topology). Suppose also that $(\mu_I, I \in D)$ is a compatible collection of probability measures, in that for $I_1, I_2 \in D$, $I_1 \sse I_2$, we have $\mu_{I_1} = (\pi_{I_1 I_2})_* \mu_{I_2}$. Then there exists a unique probability measure $\mu_T$ on $(\Omega_T, \Sigma_T)$ such that $\mu_I = (\pi_I)_* \mu_T$ for all $I \in D$.
\end{theorem}

With all the setup complete, we now prove the following lemma, which is the key intermediate step in proving Theorem \ref{thm:nonlinear-dist-space-tightness}.

\begin{lemma}\label{lemma:kolmogorov-extension-application}
Let $\{\nu_n\}_{n \geq 1}$ be a sequence of probability measures on $(\nonlineardistspace_0^{1,2}, \mc{B}_{\nonlineardistspace_0^{1,2}})$ such that for each $t \in I_0$, the pushforwards $(\evalmap_t)_* \nu_n$ converge weakly as probability measures on $(\honeorbitspace, \borelorbitspace)$. Then there exists a probability measure $\nu$ on $(\nonlineardistspace_0^{1,2}, \mc{B}_{\nonlineardistspace_0^{1,2}})$ such that for each $t \in I_0$, $(\evalmap_t)_* \nu_n$ converges weakly to $(\evalmap_t)_* \nu$.
\end{lemma}
\begin{proof}
Essentially, this is an example of applying Theorem \ref{thm:kolmogorov-bochner}. We will first define a compatible collection of probability measures --- one probability measure on $((\honeorbitspace)^I, \borelorbitspace^I)$ for each finite subset $I \sse I_0$ --- and then apply Theorem~\ref{thm:kolmogorov-bochner} to obtain a probability measure on $((\honeorbitspace)^{I_0}, \borelorbitspace^{I_0})$. Then, we will show that this probability measure is supported on $\nonlineardistspace_0^{1,2}$, and so the desired probability measure $\nu$ is obtained upon restricting to $\nonlineardistspace_0^{1,2}$.

Let $I \sse I_0$ be a finite subset. Take any $t_0 \in I_0$, $t_0 \leq \min_{t \in I} t$. By assumption, $\{(\evalmap_{t_0})_* \nu_n\}_{n \geq 1}$ converges weakly to some probability measure on $(\honeorbitspace, \borelorbitspace)$; call it $\nu_{\{t_0\}}$. Let $\rgclass_{t_0}$ be a random variable with law $\nu_{\{t_0\}}$. Define $\nu_I$, a probability measure on $((\honeorbitspace)^I, \borelorbitspace^I)$, as the law of the random variable $(\ymsemigroup_{t - t_0}^{1,2}(\rgclass_{t_0}), t \in I)$. To see why this is well defined, note that for $s_0, t_0 \in I_0$, $s_0 \leq t_0$, we have $\evalmap_{t_0} = \evalmap_{s_0} \circ \ymsemigroup_{t_0 - s_0}^{1,2}$, and consequently $(\evalmap_{t_0})_* \nu_n$ = $(\ymsemigroup_{t_0 - s_0}^{1,2})_* ((\evalmap_{s_0})_* \nu_n)$ for all $n \geq 1$. By the continuity of $\ymsemigroup_{t_0 - s_0}^{1,2}$, we obtain that $\nu_{\{t_0\}} = (\ymsemigroup_{t_0 - s_0}^{1,2})_* \nu_{\{s_0\}}$. Combining this with the semigroup property of $(\ymsemigroup_t^{1,2}, t \geq 0)$, we obtain that $\nu_I$ is well-defined. Moreover, by construction, $(\nu_I, I \sse I_0, \text{$I$ finite})$ is a compatible collection of probability measures. By Lemma \ref{lemma:finite-product-orbit-space-inner-regular}, we have that every measure in the collection $(\nu_I, I \sse I_0, \text{$I$ finite})$ is inner regular. Finally, note that $\borelorbitspace$ contains all compact subsets of $(\honeorbitspace, \toporbitspace)$, because compact subsets are closed, since $(\honeorbitspace, \toporbitspace)$ is Hausdorff (by Theorem \ref{thm:countable-index-homeomorphism}). We have thus shown that the conditions of Theorem \ref{thm:kolmogorov-bochner} are satisfied, and thus applying this theorem, we obtain a unique probability measure $\tilde{\nu}$ on $((\honeorbitspace)^{I_0}, \borelorbitspace^{I_0})$ which extends each of the measures $\nu_I$.

In particular, taking $s, t \in I_0$, $s < t$, we have that (recall Lemma \ref{lemma:countable-nonlinear-dist-space-measurable})
\[\begin{split}
\tilde{\nu}(\{X &\in (\honeorbitspace)^{I_0} : X(t) = \ymsemigroup_{t-s}^{1,2}(X(s))\}) = \\
&\nu_{\{s, t\}}(\{X \in (\honeorbitspace)^{\{s, t\}} : X(t) = \ymsemigroup_{t-s}^{1,2}(X(s))\})= 1,
\end{split}\]
and thus taking intersection over all such $s, t$, we obtain $\tilde{\nu}(\nonlineardistspace_0^{1,2}) = 1$. We may thus restrict $\tilde{\nu}$ to obtain a probability measure $\nu$ on $(\nonlineardistspace_0^{1,2}, \borelorbitspace^{I_0} |_{\nonlineardistspace_0^{1,2}})$, which by Lemma~\ref{lemma:borel-equals-infinite-product}, is the same space as $(\nonlineardistspace_0^{1,2}, \mc{B}_{\nonlineardistspace_0^{1,2}})$. The fact that $(\evalmap_t)_* \nu_n$ converges weakly to $(\evalmap_t)_* \nu$ follows because $(\evalmap_t)_* \nu = \nu_{\{t\}}$ by construction.
\end{proof}

We now use Lemma \ref{lemma:kolmogorov-extension-application} to prove Theorem \ref{thm:nonlinear-dist-space-tightness}.

\begin{proof}[Proof of Theorem \ref{thm:nonlinear-dist-space-tightness}]
By restriction, we can regard $\{\rvnldistspace_n\}_{n \geq 1}$ as a sequence of $\nonlineardistspace_0^{1,2}$-valued random variables (so we are abusing notation here; it would be more accurate to write $\{\inclnldistpace^{-1}(\rvnldistspace_n)\}_{n \geq 1}$, where $\iota_{\nonlineardistspace_0^{1,2}} : \nonlineardistspace_0^{1,2} \ra \nonlineardistspace^{1,2}$ is as in Definition~\ref{def:countable-nonlinear-dist-space}). Now by Prokhorov's theorem for the orbit space $\honeorbitspace$ (i.e., Proposition~\ref{prop:prokhorov-orbit-space}), and a diagonal argument, we can extract a subsequence $\{\rvnldistspace_{n_k}\}_{k \geq 1}$ such that for each $t \in I_0$ (which is countable), the laws of $\rvnldistspace_{n_k}(t)$ converge weakly. Applying Lemma \ref{lemma:kolmogorov-extension-application}, we may thus obtain a probability measure $\nu$ on $(\nonlineardistspace_0^{1,2}, \mc{B}_{\nonlineardistspace_0^{1,2}})$, such that letting $\tilde{\rvnldistspace}$ be a random variable with law $\nu$, we have that $\rvnldistspace_{n_k}(t) \stackrel{d}{\ra} \tilde{\rvnldistspace}(t)$ for all $t \in I_0$. Now, recalling the map $\inclnldistpace$ from Definition~\ref{def:countable-nonlinear-dist-space} and Lemma \ref{lemma:homeomorphism-of-nonlinear-dist-space}, we set $\rvnldistspace := \inclnldistpace (\tilde{\rvnldistspace})$, and we let $\mu$ be the law of $\rvnldistspace$. By construction, for all $t \in I_0$, $\rvnldistspace_{n_k}(t) \stackrel{d}{\ra} \rvnldistspace(t)$. Then using the fact that for any $t > 0$, we may take $s \leq t$, $s \in I_0$ such that $\rvnldistspace_{n_k}(t) = \ymsemigroup_{t-s}^{1,2}(\rvnldistspace_{n_k}(s))$ and $\rvnldistspace(t) = \ymsemigroup_{t-s}^{1,2}(\rvnldistspace(s))$, along with the continuity of $\ymsemigroup_{t-s}^{1,2}$ (Lemma \ref{lemma:ym-semigroup-continuous}), we obtain that for all $t > 0$, $\rvnldistspace_{n_k}(t) \stackrel{d}{\ra} \rvnldistspace(t)$. The desired result now follows by Lemma~\ref{lemma:weak-convergence-nonlinear-dist-space}.
\end{proof}

\subsection{Completing the proof}\label{section:completing-the-proof}

In this section, we will use Theorem \ref{thm:nonlinear-dist-space-tightness} to prove Theorem \ref{mainthm}. The arguments that we use will also allow us to prove Lemmas \ref{fdlemma} and \ref{symlmm}. First, we need some preliminary lemmas.

\begin{definition}[Regularized Wilson loop observables on $\nonlineardistspace^{1,2}$]\label{def:regwl-nonlinear-dist-space}
Let $\wloop : [0, 1] \ra \threetorus$ be a piecewise $C^1$ loop, $\character$ a character of $\liegroup$, and $t > 0$. Define the function $\obsnldistspace{\wloop}{t}{\character} : \nonlineardistspace^{1,2} \ra \C$ as follows. Take $0 < s < t$. Let $\obsnldistspace{\wloop}{t}{\character}(X) := [W]_{\wloop, \character, t - s}^{1,2}(X(s))$. This is well-defined by Definition \ref{def:regwl} and the semigroup property of $(\ymsemigroup_t^{1,2}, t \geq 0)$. By Lemma \ref{lemma:regwlorbit-continuous}, $\obsnldistspace{\wloop}{t}{\character} : (\nonlineardistspace^{1,2}, \topnonlineardistspace) \ra \C$ is a continuous function. Also, note that given $\gorbithone{A} \in \honeorbitspace$, we have that $[W]_{\wloop, \character, t}^{1,2}(\gorbithone{A}) = \obsnldistspace{\wloop}{t}{\character}(\iota_{\ymsemigroup^{1,2}}(\gorbithone{A}))$ (where recall $\iota_{\ymsemigroup^{1,2}}$ is the embedding $\honeorbitspace \ra \nonlineardistspace^{1,2}$ from Definition \ref{def:embedding-orbit-space-nonlinear-dist-space}). Thus we will also refer to the observables $\obsnldistspace{\wloop}{t}{\character}$ as regularized Wilson loop observables. 
\end{definition}

\begin{definition}\label{def:i12}
Define the map $i^{1,2} : \nonlineardistspace \ra \nonlineardistspace^{1,2}$ as follows. Given $X \in \nonlineardistspace$, define $Y : (0, \infty) \ra \honeorbitspace$ as follows. For $t > 0$, let $A \in X(t)$ (note $X(t) \in \orbitspace$, so that $A \in \connspace$), and set $Y(t) := \gorbithone{A}$. Define $i^{1,2}(X) := Y$.
\end{definition}

\begin{lemma}\label{lemma:i12-measurable}
The map $i^{1,2} : (\nonlineardistspace, \mc{B}(\nonlineardistspace)) \ra (\nonlineardistspace^{1,2}, \borelnonlineardistspace)$ is measurable. 
\end{lemma}
\begin{proof}
Recall the index set $I_0$, the space $\nonlineardistspace_0^{1,2}$ and the natural inclusion $\iota_{\nonlineardistspace_0^{1,2}}$ defined in Definition \ref{def:countable-nonlinear-dist-space}. Define the map $i_0^{1,2} : \nonlineardistspace \ra \nonlineardistspace_0^{1,2}$ as follows. Given $X \in \nonlineardistspace$, define $Y \in \nonlineardistspace_0^{1,2}$ as follows. For $t \in I_0$, let $A \in X(t)$, and set $Y(t) := \gorbithone{A(t)}$. Define $i_0^{1,2}(X) := Y$. Observe that $i^{1,2} = \iota_{\nonlineardistspace_0^{1,2}} \circ i_0^{1,2}$. Thus, recalling that $\iota_{\nonlineardistspace_0^{1,2}}$ is continuous (by Lemma \ref{lemma:homeomorphism-of-nonlinear-dist-space}), it suffices to show that $i_0^{1,2} : (\nonlineardistspace, \mc{B}(\nonlineardistspace)) \ra (\nonlineardistspace_0^{1,2}, \mc{B}_{\nonlineardistspace_0^{1,2}})$ is measurable. Recalling that we may think of $\nonlineardistspace_0^{1,2}$ as a measurable subset of $(\honeorbitspace)^{I_0}$ (Lemma \ref{lemma:countable-nonlinear-dist-space-measurable}), along with Lemma \ref{lemma:borel-equals-infinite-product}, it suffices to show that $i_0^{1,2} : (\nonlineardistspace, \mc{B}(\nonlineardistspace)) \ra ((\honeorbitspace)^{I_0}, \borelorbitspace^{I_0})$ is measurable. To show this, it suffices (since $\borelorbitspace^{I_0}$ is the product $\sigma$-algebra) to show that for all $t \in I_0$, the map $(\nonlineardistspace, \mc{B}(\nonlineardistspace)) \ra (\honeorbitspace, \borelorbitspace)$ defined by $X \mapsto (i_0^{1,2}(X))(t)$ is measurable. Let $\toporbitspace^w$ be the topology on $\honeorbitspace$ defined in Definition \ref{def:regwl-topology-honeorbitspace}. By Lemma \ref{lemma:orbit-space-borel-sigma-algebras-coincide}, we have that $\borelorbitspace = \sigma(\toporbitspace^w)$, and thus it suffices to show that for all $t \in I_0$, the map $(\nonlineardistspace, \newtop) \ra (\honeorbitspace, \toporbitspace^w)$ defined by $X \mapsto (i_0^{1,2}(X))(t)$ is continuous. 

Towards this end, fix $t_0 \in I_0$. Let $\{X_\alpha\}$ be a net of elements in $\nonlineardistspace$ converging to some $X \in \nonlineardistspace$. Let $Y_\alpha = i_0^{1,2}(X_\alpha)$, and $Y = i_0^{1,2}(X)$. By the definition of $\newtop$, we have that for all $\wloop, \character, t$, $W_{\wloop, \character}(X_\alpha(t)) \ra W_{\wloop, \character}(X(t))$. This implies that for all $\wloop, \character, t$, we have that $[W]_{\wloop, \character, t}^{1,2}(Y_\alpha(t_0)) \ra [W]_{\wloop, \character, t}^{1,2}(Y(t_0))$. By the definition of $\toporbitspace^w$, this implies that $Y_\alpha(t_0) \ra Y_\alpha(t_0)$, and thus the desired continuity follows.
\end{proof}

\begin{lemma}\label{lemma:gauge-orbit-smooth-choice}
Let $X \in \nonlineardistspace^{1, 2}$. For all $t > 0$, there exists a smooth {\oneform} $A \in X(t)$. 
\end{lemma}
\begin{proof}
This follows by Lemma \ref{lemma:orbitspace-gauge-orbit-smooth-choice} along with the fact that $X(t) = \ymsemigroup_{t-s}^{1,2}(X(s))$ for all $0 < s < t$.
\end{proof}

\begin{definition}\label{def:j12}
Define the map $j^{1,2} : \nonlineardistspace^{1,2} \ra \nonlineardistspace$ as follows. Given $X \in \nonlineardistspace^{1,2}$, define $Y : (0, \infty) \ra \orbitspace$ as follows. For $t > 0$, let $A \in X(t)$ be such that $A \in \connspace$ (this exists by Lemma \ref{lemma:gauge-orbit-smooth-choice}), and set $Y(t) := [A]$ (this is well-defined by Lemma \ref{lemma:wilson-loop-gauge-invariant-not-smooth-gauge-transform}). Define $j^{1,2}(X) := Y$.
\end{definition}

The following lemma follows directly from the definitions of $i^{1,2}, j^{1,2}$, and thus its proof is omitted.

\begin{lemma}\label{lemma:j12-i12-identity}
We have that $j^{1,2} \circ i^{1,2} : \nonlineardistspace \ra \nonlineardistspace$ is the identity map. 
\end{lemma}

\begin{lemma}\label{lemma:regwl-equality}
Let $X \in \nonlineardistspace^{1,2}$ and let $Y = j^{1,2}(X) \in \nonlineardistspace$. Then for any piecewise $C^1$ loop $\wloop : [0, 1] \ra \threetorus$, character $\character$ of $\liegroup$, and $t > 0$, we have that $\obsnldistspace{\wloop}{t}{\character}(X) = W_{\wloop, \character, t}(Y)$. 
\end{lemma}
\begin{proof}
Let $t > 0$, and take $t_0 \in (0, t)$. By Lemma \ref{lemma:gauge-orbit-smooth-choice}, there exists $A_0 \in X(t_0)$ which is smooth. Let $A$ be the solution to \eqref{eq:YM} on $[0, \infty)$ with initial data $A(0) = A_0$. Note that by Theorem \ref{thm:YM-global-existence}, $A \in C^\infty([0, \infty) \times \threetorus, \lalg^3)$. We also have that $\gorbithone{A(s)} = X(t_0 + s)$ for all $s \geq 0$, which implies that $Y(t) = [A(t - t_0)]$. It follows that $\obsnldistspace{\wloop}{t}{\character}(X) = [W]_{\wloop, \character, t - t_0}^{1,2}(X(t_0)) = W_{\wloop, \character}(A(t - t_0)) = W_{\wloop, \character, t}(Y)$, as desired.
\end{proof}

\begin{lemma}\label{lemma:j12-cont}
The map $j^{1,2} : (\nonlineardistspace^{1,2}, \topnonlineardistspace) \ra (\nonlineardistspace, \newtop)$ is continuous, and thus $j^{1,2} : (\nonlineardistspace^{1,2}, \borelnonlineardistspace) \ra (\nonlineardistspace, \mc{B}(\nonlineardistspace))$ is measurable.
\end{lemma}
\begin{proof}
Let $\{X_\alpha\} \sse \nonlineardistspace^{1,2}$ be a net which converges to some  $X \in \nonlineardistspace^{1, 2}$. This implies that for all piecewise $C^1$ loops $\wloop : [0, 1] \ra \threetorus$, characters $\character$ of $\liegroup$, and $t > 0$, we have that $\obsnldistspace{\wloop}{t}{\character}(X_\alpha) \ra \obsnldistspace{\wloop}{t}{\character}(X)$. Let $Y_\alpha := j^{1,2}(X_\alpha)$, $Y := j^{1,2}(X)$. Then by Lemma \ref{lemma:regwl-equality}, we have that for all $\wloop, \character, t$, $W_{\wloop, \character, t}(Y_\alpha) \ra W_{\wloop, \character, t}(Y)$. By the definition of $\newtop$, this is equivalent to $Y_\alpha \ra Y$, and thus the continuity of $j^{1,2}$ follows.
\end{proof}

We now have enough to prove Theorem \ref{mainthm}.

\begin{proof}[Proof of Theorem \ref{mainthm}]
Define the sequence $\{\tilde{\rvnldistspace}_n\}_{n \geq 1}$ of $\nonlineardistspace^{1,2}$-valued random variables by $\tilde{\rvnldistspace}_n := i^{1,2}(\rvnldistspace_n)$ for each $n \geq 1$. Note that for all $n \geq 1$ and $t > 0$, we have that $\sym^{1,2}(\tilde{\rvnldistspace}_n(t)) = \sym(\rvnldistspace_n(t))$. Thus by Theorem \ref{thm:nonlinear-dist-space-tightness}, there exists a subsequence $\{\tilde{\rvnldistspace}_{n_k}\}_{k \geq 1}$ and a probability measure $\tilde{\mu}$ on $(\nonlineardistspace^{1,2}, \borelnonlineardistspace)$ such that the laws of $\tilde{\rvnldistspace}_{n_k}$ converge weakly to $\tilde{\mu}$. Define the probability measure $\mu$ on $(\nonlineardistspace, \mc{B}(\nonlineardistspace))$ by taking the pushforward $\mu := (j^{1,2})_* \tilde{\mu}$. To finish, we claim that the laws of $\rvnldistspace_{n_k}$ converge weakly to $\mu$. To see this, let $f : (\nonlineardistspace, \newtop) \ra \R$ be a bounded continuous function. By Lemma \ref{lemma:j12-cont}, we have that $f \circ j^{1, 2} : (\nonlineardistspace^{1,2}, \topnonlineardistspace) \ra \R$ is continuous. We thus have that $\E [ f (j^{1,2}(\tilde{\rvnldistspace}_{n_k}))] \ra \int_{\nonlineardistspace^{1,2}} f \circ j^{1,2} d\tilde{\mu}$. By Lemma \ref{lemma:j12-i12-identity}, we have that $j^{1,2}(\tilde{\rvnldistspace}_{n_k}) = j^{1, 2}(i^{1,2}(\rvnldistspace_{n_k})) = \rvnldistspace_{n_k}$ for all $k \geq 1$. Also, since $\mu = (j^{1,2})_* \tilde{\mu}$, we have that $\int_{\nonlineardistspace^{1,2}} f \circ j^{1,2} d\tilde{\mu} = \int_{\nonlineardistspace} f d\mu$. We thus obtain that $\E[f(\rvnldistspace_{n_k})] \ra \int_{\nonlineardistspace} f d\mu$. Since $f$ was an arbitrary bounded continuous function, the desired weak convergence follows.
\end{proof}

\begin{proof}[Proof of Lemma \ref{symlmm}]
Let $t > 0$. Observe that the map $t \mapsto \sym(X(t))$ may be written as the composition $\sym^{1,2} \circ \evalmap_t \circ i^{1,2}$, where $i^{1,2} : \nonlineardistspace \ra \nonlineardistspace^{1,2}$ is as in Definition \ref{def:i12}, $\evalmap_t : \nonlineardistspace^{1,2} \ra \honeorbitspace$ is as in Definition \ref{def:nonlinear-dist-space}, and $\sym^{1,2} : \honeorbitspace \ra \R$ is as defined just before Lemma \ref{lemma:curvature-bounded-by-h1-norm}. The desired result now follows because each of these maps is measurable (recall Lemma \ref{lemma:i12-measurable} and Remark \ref{remark:sym-measurable}). 
\end{proof}

We next prove Lemma \ref{fdlemma}. First, we need the following result.

\begin{lemma}\label{lemma:nl-dist-space-regwl-determine-law}
Let $\rvnldistspace_1, \rvnldistspace_2$ be $\nonlineardistspace^{1,2}$-valued random variables. Suppose that for any finite collection $\wloop_i, \character_i, t_i$, $1 \leq i \leq k$, where $\wloop_i : [0, 1] \ra \threetorus$ is a piecewise $C^1$ loop, $\character_i$ is a character of $\liegroup$, and $t_i > 0$, we have that
\[ (\obsnldistspace{\wloop_i}{t_i}{\character_i}(\rvnldistspace_1), 1 \leq i \leq k) \stackrel{d}{=} (\obsnldistspace{\wloop_i}{t_i}{\character_i}(\rvnldistspace_2), 1 \leq i \leq k).\]
(The above identity is interpreted as equality in distribution of two $\C^k$-valued random variables.) Then $\rvnldistspace_1 \stackrel{d}{=} \rvnldistspace_2$.
\end{lemma}
\begin{proof}
By definition, for $X \in \nonlineardistspace^{1,2}$, we have $\obsnldistspace{\wloop}{t}{\character}(X) = [W]_{\wloop, \character, t-s}^{1, 2}(X(s))$ for any $\wloop, \character, t$ and any $0 < s < t$. It follows by Lemma \ref{lemma:orbit-space-regwl-determine-law} that for any $t > 0$, we have $\rvnldistspace_1(t) \stackrel{d}{=} \rvnldistspace_2(t)$. Let $\mu_1, \mu_2$ denote the laws of $\rvnldistspace_1, \rvnldistspace_2$, respectively. Then we have that for all $t > 0$, $(\evalmap_t)_* \mu_1 = (\evalmap_t)_* \mu_2$. Define the sequence $\{\nu_m\}_{m \geq 1}$ by $\nu_m = \mu_1$ for all $m \geq 1$. Then by Lemma \ref{lemma:weak-convergence-nonlinear-dist-space}, we have that $\nu_m$ converges weakly to $\mu_2$, from which we obtain $\mu_1 = \mu_2$, as desired.
\end{proof}

\begin{proof}[Proof of Lemma \ref{fdlemma}]
Let $\tilde{\mbf{X}} = i^{1,2}(\mbf{X})$, $\tilde{\mbf{Y}} = i^{1,2}(\mbf{Y})$. By Lemma \ref{lemma:j12-i12-identity}, we have that $j^{1,2}(\tilde{\mbf{X}}) = \mbf{X}$, $j^{1,2}(\tilde{\mbf{Y}}) = \mbf{Y}$. Thus it suffices to show that $\tilde{\mbf{X}} \stackrel{d}{=} \tilde{\mbf{Y}}$. By Lemma \ref{lemma:regwl-equality}, for any $\wloop, \character, t$, we have that $\obsnldistspace{\wloop}{t}{\character}(\tilde{\mbf{X}}) = W_{\wloop, \character, t}(\mbf{X})$, and similarly for $\tilde{\mbf{Y}}, \mbf{Y}$. Combining this with Lemma \ref{lemma:nl-dist-space-regwl-determine-law} and the assumption in the lemma statement, we obtain that $\tilde{\mbf{X}} \stackrel{d}{=} \tilde{\mbf{Y}}$, as desired.
\end{proof}


\section{Proofs of the second and third main results}\label{section:proofs-second-third}

First, we give the following lemma, which will be needed in the proof of Theorem \ref{thm:free-field-behavior-implies-tightness}.

\begin{lemma}\label{lemma:sym-decreasing}
For all $X \in \nonlineardistspace$, the function $(0, \infty) \ra [0, \infty)$ defined by $t \mapsto \sym(X(t))$ is non-increasing.
\end{lemma}
\begin{proof}
Let $X \in \nonlineardistspace$. Let $t_0 > 0$, and let $A_0 \in X(t_0)$. Let $A$ be the solution to \eqref{eq:YM} on $[0, \infty)$ with initial data $A(0) = A_0$. By \cite[Theorem 7.1]{CG2013}, we have that $t \mapsto \sym(A(t))$ is non-increasing on $[0, \infty)$. Since $\sym(A(t)) = \sym(X(t + t_0))$, we obtain that $t \mapsto \sym(X(t))$ is non-increasing on $[t_0, \infty)$. Since $t_0$ was arbitrary, the desired result follows.
\end{proof}

We next cite the following direct consequence of \cite[Proposition 3.30]{CaoCh2021} (applied with $d = 3$ and $\varep = 1/2$ (say)). Note that the assumptions that we made on $\{\rconn^n\}_{n \geq 1}$ in Section \ref{section:free-field-behavior-implies-tightness} are exactly the Assumptions A-E which are assumed in \cite[Proposition 3.30]{CaoCh2021} (the fact that Assumption C is satisfied follows from \cite[Lemma 3.7]{CaoCh2021}).

\begin{prop}[See Proposition 3.30 in \cite{CaoCh2021}]\label{prop:free-field-like-behavior-implies-tightness}
Let $\{\rconn^n\}_{n \geq 1}$ be as in Theorem \ref{thm:free-field-behavior-implies-tightness}. Then for each $n \geq 1$, there is a $(0,1]$-valued random variable $T_n$ and a $\lalg^3$-valued stochastic process $\rconn_n = (\rconn_n(t, x), t \in [0, 1), x \in \threetorus)$ such that the following hold. The function $(t, x) \mapsto \rconn_n(t, x)$ is in $C^\infty([0, T_n) \times \threetorus, \lalg^3)$, and moreover it is the solution to \eqref{eq:ZDDS} on $[0, T)$ with initial data $\rconn_n(0) = \rconn^n$. We also have that $\sup_{n \geq 1} \E[T_n^{-p}] < \infty$ for all $p \geq 1$. Finally, for all $p \geq 1$, we have that
\[ \sup_{n \geq 1} \E\bigg[\sup_{t \in (0, T_n)} t^{p(1 + (1/4)(3-\alpha))} \|\rconn_n(t)\|_{C^1}^p \bigg] < \infty.\]
\end{prop}

We now use Proposition \ref{prop:free-field-like-behavior-implies-tightness} to prove Theorem \ref{thm:free-field-behavior-implies-tightness}.

\begin{proof}[Proof of Theorem \ref{thm:free-field-behavior-implies-tightness}]
By Proposition \ref{prop:free-field-like-behavior-implies-tightness}, we obtain the following. Let $T_n, \rconn_n$ be as given by the proposition. Then $\sup_{n \geq 1} \E[T_n^{-p}] < \infty$. Also, for all $n \geq 1$, $t \in (0, T_n)$, we have that $\rconn_n(t) \in \rvnldistspace_n(t)$ (this follows by Lemmas \ref{lemma:ym-zdds-solution-gauge-transformation} and \ref{lemma:wilson-loop-gauge-invariant-not-smooth-gauge-transform}). Finally, defining
\[ R_n := \sup_{t \in (0, T_n)} t^{1 + (1/4)(3-\alpha)} \|\rconn_n(t)\|_{C^1}, \]
we have that $\sup_{n \geq 1} \E[R_n^p] < \infty$ for all $p \geq 1$. 

Now, fix $t > 0$. By Lemma \ref{lemma:sym-decreasing}, we have that
\[ \sym(\rvnldistspace_n(t)) \leq \ind(t < T_n) \sym(\rconn_n(t)) + \ind(t \geq T_n) \sym(\rconn_n(T_n/2)).\]
By Lemma \ref{lemma:curvature-bounded-by-h1-norm} we have that for all $s \in (0, T)$,
\[\begin{split}
\sym(\rconn_n(s)) &\leq \const(\|\rconn_n(s)\|_{H^1} + \|\rconn_n(s)\|_{H^1}^2) \\
&\leq \const (\|\rconn_n(s)\|_{C^1} + \|\rconn_n(s)\|_{C^1}^2).
\end{split}\]
Let $\gamma := 1 + (1/4)(3-\alpha)$. Then recalling the definition of $R_n$, we obtain
\begin{align*}
\sym(\rvnldistspace_n(t)) &\leq \const \ind(t < T_n) (t^{-2\gamma} R_n^2 + t^{-\gamma} R_n) \\
&\qquad + \const \ind(t \geq T_n) (T_n^{-2\gamma} R_n^2 + T_n^{-\gamma} R_n), 
\end{align*}
which implies
\[ \sym(\rvnldistspace_n(t)) \leq \const (\min\{t, T_n\}^{-2\gamma} R_n^2 + \min\{t, T_n\}^{-\gamma} R_n).\]
Using H\"{o}lder's inequality and the fact that $\{T_n^{-1}\}_{n \geq 1}$ and $\{R_n\}_{n \geq 1}$ are both uniformly $L^p$-bounded for all $p \geq 1$, we obtain $\sup_{n \geq 1} \E[ \sym(\rvnldistspace_n(t))] < \infty$.
\end{proof}



We next move on to proving Theorem \ref{thm:gauge-invariant-free-field}. Let $\rconn^0, \{\rconn^0_N\}_{N \geq 1}$ be as constructed in Section \ref{section:gauge-invariant-free-field}. We quote \cite[Theorem 1.5]{CaoCh2021}, which is the main technical result that enables us to construct the gauge invariant free field.

\begin{theorem}[Theorem 1.5 in \cite{CaoCh2021}]\label{thm:ym-heat-flow-gff}
There exists a $\lalg^3$-valued stochastic process $\rconn = (\rconn(t, x), t \in (0, 1), x \in \threetorus)$ and a random variable $T \in (0, 1]$ such that the following hold. The function $(t, x) \mapsto \rconn(t, x)$ is in $C^\infty((0, T) \times \threetorus, \lalg^3)$, and moreover it is a solution to \eqref{eq:ZDDS} on $(0, T)$. Also, $\E[T^{-p}] <  \infty$ for all $p \geq 1$. The process $\rconn$ relates to $\rconn^0$ in the following way. There exists a sequence $\{T_N\}_{N \geq 0}$ of $(0, 1]$-valued random variables such that the following hold. First, for any $p \geq 1$, we have that $\sup_{N \geq 0} \E[T_N^{-p}] < \infty$, and that $\E[|T_N^{-1} - T^{-1}|^p] \ra 0$ (which implies that $T_n$ converges to $T$ in probability). Also, there is a sequence $\{\rconn_N\}_{N \geq 0}$ of $\lalg^3$-valued stochastic processes $\rconn_N = (\rconn_N(t, x), t \in [0, 1), x \in \torus^3)$ such that for each $N \geq 0$, the function $(t, x) \mapsto \rconn_N(t, x)$ is in $C^\infty([0, T_N) \times \threetorus, \lalg^3)$, and moreover it is the solution to \eqref{eq:ZDDS} on $[0, T_N)$ with initial data $\rconn_N(0) = \rconn^0_N$. Finally, for any $p \geq 1$, $\delta, \varep > 0$, we have that
\[\begin{split} \lim_{N \toinf} \E\bigg[ \sup_{\substack{t \in (0, (1-\delta)T)}} t^{p((3/4) + \varep)}\|\rconn_N(t) - \rconn(t)\|_{C^1}^p \bigg] &= 0.
\end{split}\] 
\end{theorem}

Let $\rconn, T, \rconn_N, T_N$ be as given by Theorem \ref{thm:ym-heat-flow-gff}. Define a random variable $\rvnldistspace \in \nonlineardistspace$ as follows. For $t \in (0, T)$, define $\rvnldistspace(t) := [\rconn(t)]$. For $t \geq T$, let $s \in (0, T)$, and define $\rvnldistspace(t) := \Phi_{t-s}(\rvnldistspace(s))$. (Since $\rconn$ is a solution to \eqref{eq:ZDDS} on $(0, T)$, the specific choice of $s$ does not matter, by Lemmas \ref{lemma:ym-zdds-solution-gauge-transformation} and \ref{lemma:wilson-loop-gauge-invariant-not-smooth-gauge-transform}.) The following lemma shows that $\rvnldistspace$ is an $\nonlineardistspace$-valued random variable. The proof is in Appendix \ref{section:measurability}.

\begin{lemma}\label{lemma:gauge-invariant-free-field-measurable}
Let $\rvnldistspace$ be as just defined. Then $\rvnldistspace$ is an $\nonlineardistspace$-valued random variable (i.e., it is measurable).
\end{lemma}

We can now prove Theorem \ref{thm:gauge-invariant-free-field}.

\begin{proof}[Proof of Theorem \ref{thm:gauge-invariant-free-field}]
Let $f$ be a bounded continuous function.~We will show that $\E f(\rvnldistspace_N) \ra \E f(\rvnldistspace)$. It suffices to show that for all subsequences $\{N_k\}_{k \geq 1}$, there exists a further subsequence $\{N_{k_j}\}_{j \geq 1}$ such that $\E f(\rvnldistspace_{N_{k_j}}) \ra \E f(\rvnldistspace)$. By bounded convergence, it suffices to show that $f(\rvnldistspace_{N_{k_j}}) \stackrel{a.s.}{\ra} f(\rvnldistspace)$. Thus let $\{N_k\}_{k \geq 1}$ be a subsequence. By Theorem \ref{thm:ym-heat-flow-gff}, we have that $T_N \stackrel{p}{\ra} T$, and also by the last claim of the theorem (applied with $p=1$, $\delta = 1/2$, $\varep  = 1/4$, say), we have
\[ \sup_{t \in (0, T/2)} t \|\rconn_N(t) - \rconn(t)\|_{C^1} \stackrel{p}{\ra} 0. \]
Thus there exists a further subsequence $\{N_{k_j}\}_{j \geq 1}$ such that a.s.,
\[ T_{N_{k_j}} \ra T ~~~~\text{   and   }~~~~ \sup_{t \in (0, T/2)} t \|\rconn_{N_{k_j}}(t) - \rconn(t)\|_{C^1} \ra  0.\]
Let $E$ be the event that the above happens. Then $\p(E) = 1$, and thus it suffices to show that on $E$, we have that $f(\rvnldistspace_{N_{k_j}}) \ra f(\rvnldistspace)$. Since $f$ is continuous, it suffices to show that on $E$, we have that $\rvnldistspace_{N_{k_j}} \ra \rvnldistspace$ (in the topology $\newtop$). By the definition of $\newtop$, we need to show that on $E$, for all piecewise $C^1$ loops $\wloop$, characters $\character$ of $\liegroup$, and $t > 0$, we have that $W_{\wloop, \character}(\rvnldistspace_{N_{k_j}}(t)) \ra W_{\wloop, \character}(\rvnldistspace(t))$. Thus fix $\wloop, \character, t$. Take $s \in (0, T/2)$, $s < t$. Recall the regularized Wilson loop observable $W_{\wloop, \character, t}^{1, 2}$ from Definition \ref{def:regwl}. By the construction of $\rvnldistspace$, we have that $\rconn(u) \in \rvnldistspace(u)$ for all $u \in (0, T)$, and by Lemmas \ref{lemma:ym-zdds-solution-gauge-transformation} and \ref{lemma:wilson-loop-gauge-invariant-not-smooth-gauge-transform}, we have that $\rconn_N(u) \in \rvnldistspace_N(u)$ for all $N \geq 0$ and $u \in (0, T_N)$. Thus
\[ W_{\wloop, \character}(\rvnldistspace(t)) = W_{\wloop, \character, t-s}^{1, 2}(\rconn(s)), \]
and for $j$ large enough so that $T_{N_{k_j}} > T/2$,
\[ W_{\wloop, \character}(\rvnldistspace_{N_{k_j}}(t)) = W_{\wloop, \character, t-s}^{1, 2}(\rconn_{N_{k_j}}(s)). \]
Since $s \in (0, T/2)$, we have that (on $E$) $\rconn_{N_{k_j}}(s) \ra \rconn(s)$ in $C^1$, which implies that $\rconn_{N_{k_j}}(s) \ra \rconn(s)$ in $\connspace^{1,2}$ (in $H^1$ norm), and thus $\rconn_{N_{k_j}}(s) \weakarrow \rconn(s)$. By Lemma \ref{lemma:regularized-wilson-loop-weak-sequential-continuity}, we obtain that
\[ \lim_{j \toinf} W_{\wloop, \character, t-s}^{1, 2}(\rconn_{N_{k_j}}(s)) = W_{\wloop, \character, t-s}^{1, 2}(\rconn(s)).\]
Combining the previous few results, we obtain the desired convergence.
\end{proof}

\appendix

\numberwithin{theorem}{section}

\section{Proofs of miscellaneous technical lemmas}\label{section:3d-miscellaneous-proofs}

To prove Lemma \ref{lemma:limiting-gauge-transformation}, we quote the following result, which is \cite[Lemma A.8 (ii-iii)]{W2004}, in the special case where the parameters in the lemma statement are set to $k = 2$, $p = 2$, $n = 3$, and the principal $G$-bundle $P = \threetorus \times \liegroup$ is trivial.

\begin{lemma}[Corollary of Lemma A.8 (ii-iii) in \cite{W2004}]\label{lemma:sequence-and-gauge-transform-bounded}
Let $\{A_n\}_{n \geq 1} \sse \honeconnspace$, $\{\gt_n\}_{n \geq 1} \sse \htwogaugetransf$ be such that
\[ \sup_{n \geq 1} \|A_n\|_{H^1}, ~~ \sup_{n \geq 1} \|\gaction{A_n}{\gt_n}\|_{H^1} < \infty.\]
Then there exists $\gt \in \htwogaugetransf$ and a subsequence $\{\gt_{n_k}\}_{k \geq 1}$ such that 
\[
\|\gt_{n_k} - \gt\|_\infty := \sup_{x \in \threetorus} \|\gt_{n_k}(x) - \gt(x)\| \ra 0
\]
(where the matrix norm is Frobenius norm), and moreover for all $1 \leq s < 6$, we have that $\gt_{n_k}^{-1} d\gt_{n_k} \ra \gt^{-1} d\gt$ in $L^s(\threetorus, \lalg^3)$. In particular, the convergence happens in $L^2(\threetorus, \lalg^3)$.
\end{lemma}

\begin{remark}
Technically, the conclusion of \cite[Lemma A.8 (ii)]{W2004} is that 
\[
\sup_{x \in \threetorus} d_\liegroup(\gt_{n_k}(x), \gt(x)) \ra 0,
\]
where $d_G$ is the metric on $\liegroup$ defined by (see \cite[Remark A.3 (iv)]{W2004})
\[ d_\liegroup(h_1, h_2) := \inf\{\|X\| : X \in T_{h_1}  G, h_2 = \exp_{h_1}(X)\}, ~~ h_1, h_2 \in \liegroup,\]
where $\exp_{h_1}(X) := h_1 \exp(h_1^{-1} X)$ (see \cite[Remark A.3 (iii)]{W2004}), $\exp$ is the matrix exponential, and $\|X\|$ is the Frobenius norm of $X$. To reconcile this with the conclusion of Lemma \ref{lemma:sequence-and-gauge-transform-bounded}, we note that for all $h_1, h_2 \in \liegroup$, we have that 
\[
\|h_1 - h_2\| \leq N^{1/2} d_\liegroup(h_1, h_2).
\]
(Recall that $\liegroup \sse \unitary(N)$.) To see this,  fix $h_1, h_2 \in \liegroup$, and suppose we have $X \in T_{h_1} \liegroup$ such that $h_2 = \exp_{h_1}(X) = h_1 \exp(h_1^{-1} X)$. Define the path $\gamma : [0, 1] \ra \liegroup$ by $\gamma(t) := h_1 \exp(t h_1^{-1} X)$. Then $\gamma(0) = h_1$, $\gamma(1) = h_2$, and $\gamma'(t) = X h_1^{-1} \gamma(t)$ for $t \in (0, 1)$. We thus have
\[ h_2 - h_1 = X \int_0^1 h_1^{-1}\gamma(t) dt,  \]
and thus
\[ \|h_1 - h_2\| \leq \|X\| \int_0^1 \|h_1^{-1} \gamma(t)\| dt = N^{1/2} \|X\|, \]
where we used that $\|h\| = N^{1/2}$ for all $h \in \liegroup$ (since $h^* h = \groupid$). Taking infimum over $X$ gives the claim.
\end{remark}

\begin{proof}[Proof of Lemma \ref{lemma:limiting-gauge-transformation}]
Note that by the weak convergence (and Lemma \ref{lemma:weak-convergence-and-h1-norm}), we have
\[ \sup_{n \geq 1} \|A_n\|_{H^1}, ~~\sup_{n \geq 1} \|\gaction{A_n}{\gt_n}\|_{H^1} < \infty. \]
Now by Lemma \ref{lemma:sequence-and-gauge-transform-bounded}, there exists $\gt_\infty \in \htwogaugetransf$ and a subsequence $\{\gt_{n_k}\}_{k \geq 1}$ such that $\|\gt_{n_k} - \gt_\infty\|_\infty \ra 0$ and $\gt_{n_k}^{-1} d \gt_{n_k} \ra \gt_\infty^{-1} d\gt_\infty$ in $L^2$ as $k \toinf$. Since $A_n \ra A_\infty$ in $L^2$ (by the Sobolev embedding $H^1 \hookrightarrow L^2$), we have that $\gaction{A_{n_k}}{\gt_{n_k}} \ra A_\infty^{\gt_\infty}$ in $L^2$. But we also have $\gaction{A_{n_k}}{\gt_{n_k}} = B_{n_k} \ra B_\infty$ in $L^2$ (again by the Sobolev embedding $H^1 \hookrightarrow L^2$). We thus obtain $B_\infty = \gaction{A_\infty}{\gt_\infty}$, as desired.
\end{proof}

\begin{proof}[Proof of Lemma \ref{lemma:wilson-loop-gauge-invariant-not-smooth-gauge-transform}]
By assumption, we have that 
\beq\label{eq:wilson-loop-not-smooth-proof-identity} \gt^{-1} d\gt = \tilde{A} - \gt^{-1} A \gt \text{ (where the equality holds a.e.)}. \eeq
Since $\gt \in \htwogaugetransf$ implies that $\gt$ is continuous, we have (by the above) that $\gt^{-1} d\gt$ is equal a.e.~to a continuous function. Thus the same is true for $d\gt$. We claim that this implies that $\gt$ is in fact in $C^1(\threetorus, \liegroup)$. This follows by the general statement that if $f \in C(\threetorus, \R^n)$ for some $n$ (for us $n = 2 (N \times N)$, where $N$ is such that $\liegroup \sse \unitary(N)$), such that for each $1 \leq i \leq 3$, the weak derivative $\ptl_i f$ exists and is continuous, then in fact $f \in C^1(\threetorus, \R^n)$. To see this, let $\{\eta_\varep\}_{\varep > 0}$ be a family of mollifiers (for instance, take $\eta_\varep$ to be the density of standard Brownian motion on $\threetorus$ at time $\varep$). Since $f$, $\ptl_i f$, $1 \leq i \leq 3$ are all continuous, we have that $f * \eta_\varep \ra f$ and $(\ptl_i f) * \eta_\varep \ra \ptl_i f$, $1 \leq i \leq 3$ as $\varep \downarrow 0$, where the limits are in $C(\threetorus, \R^n)$ (i.e., in supremum norm). Also, for each $\varep > 0$, we have that $f * \eta_\varep$ is smooth, and moreover, for $1 \leq i \leq 3$, we have that $\ptl_i (f * \eta_\varep) = (\ptl_i f) * \eta_\varep$ (this latter assertion may be verified by first showing that the identity holds as weak derivatives, and then using that the classical derivative, when it exists, coincides with the weak derivative). Taking a sequence $\varep_n \downarrow 0$, and combining the previous few observations, we obtain that $\{f * \eta_{\varep_n}\}_{n \geq 1}$ is a Cauchy sequence in $C^1(\threetorus, \R^n)$. Thus there exists some $\tilde{f} \in C^1(\threetorus, 
\R^n)$ such that $f * \eta_{\varep_n} \ra \tilde{f}$ in $C^1(\threetorus, \R^n)$. But since we have that $f * \eta_{\varep_n} \ra f$ in $C(\threetorus, \R^n)$, we obtain that $f = \tilde{f}$, and thus $f \in C^1(\threetorus, \R^n)$.

Having shown that $\sigma$ is in $C^1(\threetorus, \liegroup)$, using \eqref{eq:wilson-loop-not-smooth-proof-identity} we can then obtain that $d\gt$ is equal to a $C^1$ function, and thus we obtain that $\gt$ is in $C^2(\threetorus, \liegroup)$. Proceeding similarly, we can obtain that $\gt$ is in $C^k(\threetorus, \liegroup)$ for any $k \geq 1$, and the desired result follows.
\end{proof}

\begin{proof}[Proof of Lemma \ref{lemma:sequential-convergence-topology-properties}]
Suppose that $V$ is closed in $(S, T(C))$. Then for any $s \in S$ such that there is a sequence $\{s_n\}_{n \geq 1} \sse V$ with $s_n \ra s$ in $(S, T(C))$ (i.e., $s_n \ra_{\mc{C}(T(C))} s$), we must have $s \in V$. It follows that $V$ is sequentially closed with respect to $\ra_C$, since $s_n \ra_C s$ implies that $s_n \ra s$ in $(S, T(C))$. For the other direction, suppose that $V$ is sequentially closed with respect to $\ra_C$. We will show that $V^c$ is open. Towards this end, suppose that $s_n \ra_C s \in V^c$. We need to show that $s_n \in V^c$ eventually. Suppose not; then there exists a subsequence $\{s_{n_k}\}_{k \geq 1}$ such that $s_{n_k} \in V$ for all $k \geq 1$. By property (2) in Definition \ref{def:sequential-convergence}, we have that $s_{n_k} \ra_C s$, which implies that $s \in V$ (since $V$ is assumed to be sequentially closed). But we also have $s \in V^c$, a contradiction.

For the second claim, it suffices to show that for all closed sets $V$ in $(S', T(C'))$, we have that $f^{-1}(V)$ is closed in $(S, T(C))$. By the first claim, it suffices to show that if we have $s \in S$ and a sequence $\{s_n\}_{n \geq 1} \sse f^{-1}(V)$ with $s_n \ra_C s$, then $s \in f^{-1}(V)$. Suppose we have such an $s$ and $\{s_n\}_{n \geq 1}$. Then by assumption, $f(s_n) \ra_{C'} f(s)$. Since $V$ is closed in $(S', T(C'))$ and $f(s_n) \in V$ for all $n \geq 1$, we obtain $f(s) \in V$, and thus $s \in f^{-1}(V)$, as desired.
\end{proof}

\begin{proof}[Proof of Corollary \ref{cor:sequential-convergence-metric-space}]
It suffices to show that for all closed sets $V$ in $(X, d)$, we have that $f^{-1}(V)$ is closed in $(S, T(C))$. The argument is then essentially exactly the same as the proof of the second claim in Lemma \ref{lemma:sequential-convergence-topology-properties}.
\end{proof}

We next begin to prove the lemmas which were needed in Section \ref{section:3d-u(1)}. Thus, in the rest of Appendix \ref{section:3d-miscellaneous-proofs}, we take $\liegroup = \unitary(1)$. In the following, we will use notation that was introduced in Section \ref{section:3dalt-notation}; for instance, we will use the definitions of the exterior derivative $d$ and coderivative $d^*$. We will also use some facts about Fourier transforms. For instance, if $A$ is a smooth {\oneform}, then for $1 \leq j \leq 3$, $n \in \Z^3$, we have that $\widehat{\ptl_j A}(n) = \icomplex 2\pi n_j \hat{A}(n)$. We also have that for any $k \geq 1$,
\beq\label{eq:fourier-coefficients-rapid-decay} \sup_{n \in \Z^3} (1 + |n|^2)^{k/2} |\hat{A}(n)| < \infty. \eeq
(See, e.g., \cite[Chapter 3, equation (1.7)]{T2011a}.)

\begin{proof}[Proof of Lemma \ref{lemma:da-l2-norm-formula}]
Note that when $\liegroup = \unitary(1)$, we have that $\sym(A) = \|dA\|_2^2$, since all commutator terms disappear. In the following, all series converge absolutely, because $A$ is smooth. We have that
\[ (dA)_{kj} = 2\pi \icomplex \sum_{n \in \Z^3} (n_k \hat{A}_j(n) - n_j \hat{A}_k(n)) e_n. \]
Thus
\begin{align*}
\|dA \|_2^2 &= \sum_{k, j = 1}^3 \|(dA)_{kj}\|_2^2\\
&= 4\pi^2 \sum_{n \in \Z^3} \sum_{k,j = 1}^3 |n_k \hat{A}_j(n) - n_j \hat{A}_k(n)|^2 . 
\end{align*}
(In the second equality, we used the Plancherel identity --- see, e.g., \cite[Chapter 3]{T2011a}.)
For any $n \in \Z^3$, we have
\[\begin{split}
\sum_{k,j=1}^3 |n_k \hat{A}_j(n) - n_j \hat{A}_k(n)|^2 &= \sum_{k, j = 1}^3 (|n_k \hat{A}_j(n)|^2 + |n_j \hat{A}_k(n)|^2)  \\
&\qquad - \sum_{k,j=1}^3 n_k n_j \big(\hat{A}_j(n) \ovl{\hat{A}_k(n)} + \hat{A}_k(n) \ovl{\hat{A}_j(n)}\big). 
\end{split}\]
This can be rewritten
\[ 2 |n|^2 |\hat{A}(n)|^2- 2 |n \cdot \hat{A}(n)|^2, \]
and the desired result follows.
\end{proof}

We next proceed to prove Lemma \ref{lemma:wilson-loop-heat-kernel-regularized}. First, we need a preliminary result.

\begin{lemma}\label{lemma:wilson-loop-formula-u(1)}
Let $A$ be a smooth $1$-form, $\wloop : [0, 1] \ra \threetorus$ a piecewise $C^1$ loop, and $\character$ a character of $\unitary(1)$. Then the Wilson loop observable is given by
\[ W_{\wloop, \character}(A) = \character\bigg(\exp(\int_0^1 A(\wloop(s)) \cdot \wloop'(s) ds )\bigg).\]
\end{lemma}
\begin{proof}
First, note that the solution $f : [0, 1] \ra \icomplex \R$ to the ODE
\[ f'(t) = A(\wloop(t)) \cdot \wloop'(t), ~~ f(0) = 0\]
is just given by integrating:
\[ f(t) = \int_0^t A(\wloop(s)) \cdot \wloop'(s) ds. \]
If we define $h : [0, 1] \ra \unitary(1)$ by $h(t) := \exp(f(t))$, then we see that
\[ h'(t) = h(t) A(\wloop(t)) \cdot \wloop'(t), ~~ h(0) = 1, \]
and thus $h(1) = \exp(f(1))$ is the holonomy of $A$ around $\wloop$. To finish, note that by definition, $W_{\wloop, \character}(A) = \character(h(1))$.
\end{proof}

\begin{proof}[Proof of Lemma \ref{lemma:wilson-loop-heat-kernel-regularized}]
This now follows directly by Lemma \ref{lemma:wilson-loop-formula-u(1)} and \eqref{eq:heat-kernel-def}.
\end{proof}

We now proceed to prove Lemma \ref{lemma:gauge-fixing-u(1)}. First, we need some preliminary results.

\begin{lemma}\label{lemma:da-zero-characterization}
Let $A$ be a smooth {\oneform}. Then $dA = 0$ if and only if for all $n \in \Z^3 \setminus \{0\}$, there exists a constant $\alpha_n \in \C$ such that $\hat{A}(n) = \alpha_n n$.
\end{lemma}
\begin{proof}
By Lemma \ref{lemma:da-l2-norm-formula} and the fact that $\sym(A) = \|dA\|_2^2$ when $\liegroup = \unitary(1)$, we have that $dA = 0$ if and only if for all $n \in \Z^3 \setminus \{0\}$, 
\[ |n|^2 |\hat{A}(n)|^2 = |n \cdot \hat{A}(n)|^2.\]
By the equality case of Cauchy--Schwarz, the above identity is true if and only if there exists $\alpha_n \in \C$ such that $\hat{A}(n) = \alpha_n n$.
\end{proof}

\begin{lemma}\label{lemma:d-star-a-equal-0}
Let $A$ be a smooth {\oneform}. Then $d^* A = 0$ if and only if $n \cdot \hat{A}(n) = 0$ for all $n \in \Z^3$.
\end{lemma}
\begin{proof}
We have that $d^* A = - \sum_{k=1}^3 \ptl_k A_k$. Note for $1 \leq k \leq 3$, we have that $\ptl_k A_k = 2\pi \icomplex \sum_{n \in \Z^3} n_k \hat{A}_k(n) e_n$ (where the series converges absolutely, because $A_k$ is smooth), and thus
\[ d^* A = - 2\pi \icomplex \sum_{n \in \Z^3} \big(n \cdot \hat{A}(n)\big) e_n. \]
The desired result now follows.
\end{proof}

\begin{lemma}\label{lemma:gauge-equiv-characterization}
Let $A_1, A_2$ be smooth {\oneforms}. Then $[A_1] = [A_2]$ if and only if $d(A_1 - A_2) = 0 $.
\end{lemma}
\begin{proof}
Suppose $[A_1] = [A_2]$. Then there exists $\gt \in \gaugetransf$ such that $A_2 = \gaction{A_1}{\gt} = A_1 + \gt^{-1} d\gt$. Thus it suffices to show that $d(\gt^{-1} d\gt) = 0$. For $1 \leq k, j \leq 3$, we have
\[ (d(\gt^{-1} d\gt))_{kj} = \ptl_k (\gt^{-1} \ptl_j \gt) - \ptl_j (\gt^{-1} \ptl_k \gt). \]
Note
\[ \ptl_k (\gt^{-1} \ptl_j \gt) = - \gt^{-1} (\ptl_k \gt) \gt^{-1} \ptl_j \gt + \gt^{-1} \ptl_{kj} \gt. \]
Swapping $k, j$, we obtain another identity, and then subtracting the two identities and observing the cancellation that occurs gives us $(d(\gt^{-1} d\gt))_{kj} = 0$, as desired. (The fact that $\liegroup = \unitary(1)$, so that $\lalg = \icomplex\R$, gives that $\gt^{-1} (\ptl_k \gt)$ commutes with $\gt^{-1} (\ptl_j \gt)$.)

Conversely, suppose $d(A_1 - A_2) = 0$. By Lemma \ref{lemma:da-l2-norm-formula}, for all $n \in \Z^3 \setminus \{0\}$, there exists $\alpha_n$ such that
\[ \hat{A}_1(n) = \hat{A}_2(n) + \alpha_n n.\]
Note then for any piecewise $C^1$ loop $\wloop$ and character $\character$ of $\liegroup$, we have that
\[\begin{split} \hat{A}_1(n) \cdot \int_0^1 e_n(\wloop(s)) \wloop'(s) ds = \hat{A}_2(n)& \cdot \int_0^1 e_n(\wloop(s)) \wloop'(s) ds ~+ \\
&\alpha_n \int_0^1 e_n(\wloop(s)) n \cdot \wloop'(s) ds. \end{split}\]
Noting that $\nabla e_n = 2\pi \icomplex e_n n$, we have that the last term on the right hand side is
\[ \frac{\alpha_n}{2\pi \icomplex} \int_0^1 \frac{d}{ds} (e_n(\wloop(s))) ds. \]
The integral is $e_n(\wloop(1)) - e_n(\wloop(0)) = 0$. Thus by Lemma \ref{lemma:wilson-loop-heat-kernel-regularized}, we obtain that $W_{\wloop, \character}(A_1) = W_{\wloop, \character}(A_2)$. Since this holds for all $\wloop, \character$, we can apply Lemma \ref{lmm:wilson-loops-determine-gauge-eq-class} to obtain $[A_1] = [A_2]$, as desired.
\end{proof}

\begin{proof}[Proof of Lemma \ref{lemma:gauge-fixing-u(1)}]
For existence, let $A$ be a smooth {\oneform}. Define the {\oneform} 
\[ B := \sum_{\substack{n \in \Z^3 \\ n \neq 0}} (\hat{A}(n) - (\hat{A}(n) \cdot n) n / |n|^2) e_n. \]
(Note that $B$ is smooth by the rapid decay of the Fourier coefficients of $A$ --- recall \eqref{eq:fourier-coefficients-rapid-decay}.) Note that $\hat{B}(0) = 0$, and for $n \neq 0$, we have that 
\[
\hat{B}(n) = \hat{A}(n) - (\hat{A}(n) \cdot n) n / |n|^2.
\]
Thus, by Lemma \ref{lemma:da-zero-characterization}, we have that $d(A - B) = 0$, and so, by Lemma \ref{lemma:gauge-equiv-characterization}, we have that $[A] = [B]$. Moreover, we have that $n \cdot \hat{B}(n) = 0$ for all $n \in \Z^3$.

For uniqueness, let $A_1, A_2 \in [A]$ be such that $n \cdot \hat{A}_1(n) = n \cdot \hat{A}_2(n) = 0$ for all $n \in \Z^3$. By Lemma \ref{lemma:da-zero-characterization} and the fact that $d(A_1 - A_2) = 0$ (by Lemma~\ref{lemma:gauge-equiv-characterization}), we have that for all $n \in \Z^3 - \{0\}$, there exists $\alpha_n \in \C$ such that
\[ \hat{A}_1(n) = \hat{A}_2(n) + \alpha_n n.\]
Taking inner product of both sides of the equation with $n$, we obtain that $\alpha_n |n|^2 = 0$, and thus $\hat{A}_1(n) = \hat{A}_2(n)$ for all $n \in \Z^3 \setminus \{0\}$. Thus if we also have $\hat{A}_1(0) = 0 = \hat{A}_2(0)$, then we must have $A_1 = A_2$.
\end{proof}

Before we prove Lemma \ref{lemma:in-nldist-space}, note that by using the definitions of $d$ and $d^*$ in Section \ref{section:3dalt-notation}, we have that for smooth {\oneforms} $A$, the Laplacian 
\[
\Delta A := \sum_{i=1}^3 \ptl_{ii} A = - (dd^* + d^* d) A.
\]
\begin{proof}[Proof of Lemma \ref{lemma:in-nldist-space}]
Note that when $\liegroup = \unitary(1)$, the Yang--Mills heat flow reduces to the PDE
\[ \ptl_t A(t) = - d^* dA(t), ~~t > 0. \]
(This follows because all commutator terms disappear.) Define the function $A : [0, \infty) \ra \connspace$ by $A(t) := e^{t \Delta} A_0$. Note that $A$ satisfies the heat equation
\[
\ptl_t A(t) = \Delta A(t) = -(d^* d + dd^*) A(t).
\] 
We claim that $d^* A(t) = 0$ for all $t \geq 0$, which implies that $\Delta A(t) = -d^* d A(t)$. To see this, note that for all $n \in \Z^3$, we have that $n \cdot \widehat{e^{t \Delta} A_0}(n) = e^{-4\pi^2 |n|^2 t} n \cdot \hat{A}_0(n) = 0$, where the last identity follows because $A_0 \in \connspace_0$. By Lemma \ref{lemma:d-star-a-equal-0}, we thus have that $d^* A(t) = 0$. We thus have that $A$ is the solution to the Yang-Mills heat flow with initial data $A(0) = A_0$. Since $X(t) = [A(t)]$ for all $t > 0$ by definition, it thus follows that $X \in \nonlineardistspace$.
\end{proof}

\section{Proofs of technical lemmas for YM and ZDDS flows}\label{section:ym-proof-h1}

In this section, we will extensively reference the results of Charalambous and Gross~\cite{CG2013}. Their results are all stated in a general setting where the base manifold is a compact 3-dimensional Riemannian manifold with smooth boundary, and thus many of the results are stated with boundary conditions. We will apply the results for the special case where the base manifold is $\threetorus$. Since $\threetorus$ has empty boundary, any stated boundary conditions will be superfluous, as they will be vacuously satisfied.

\begin{proof}[Proof of Theorem \ref{thm:YM-global-existence}]
The existence part of the theorem is \cite[Theorem 2.5]{CG2013} or \cite[Theorem 1]{R1992}. The uniqueness part of the theorem is \cite[Theorem 9.15]{CG2013}. One of the assertions of \cite[Theorem 1]{R1992} is that if $A_0$ is smooth, then the solution $A \in C([0, \infty), \honeconnspace)$ to \eqref{eq:YM} is smooth.
\end{proof}


We now begin to prove Proposition \ref{prop:ZDDS-local-existence}, Lemma \ref{lemma:ym-zdds-solution-gauge-transformation}, 
and Lemma \ref{lemma:zdds-solutions-weak-continuity}. The proofs are just small modifications and extensions of the arguments found in \cite[Sections 8 and 9]{CG2013}. Therefore we will be a bit brief and cite many of the results from \cite{CG2013} which we do not need to modify, while going into more detail for the arguments from \cite{CG2013} which we do need to modify. 

Let $\Delta := -(dd^* + d^* d)$ be the Laplacian on $\threetorus$. By using the definitions of $d$ and $d^*$ given in Section \ref{section:3dalt-notation}, we can see that as an operator on functions (i.e., {\zeroforms}), this is the usual Laplacian $\sum_{i=1}^3 \ptl_{ii}$, while as an operator on {\oneforms} $A$, this acts by
\[ \Delta A = (\Delta A_1, \Delta A_2, \Delta A_3), \]
where $\Delta A_i$ is the Laplacian of the function $A_i$. Consequently, the heat kernel $e^{t \Delta}$ on {\oneforms} acts by
\[ e^{t \Delta} A = (e^{t \Delta} A_1, e^{t\Delta} A_2, e^{t \Delta} A_3), \]
where $e^{t \Delta} A_i$ is the heat kernel acted on the function $A_i$. 

Note that \eqref{eq:ZDDS} may be written explicitly as
\beq\label{eq:zdds-explicit} \ptl_t A_i(t) = \Delta A_i +  \sum_{j=1}^3 [A_j, 2 \ptl_j A_i - \ptl_i A_j + [A_j, A_i]], ~~ 1 \leq i \leq 3.\eeq
Following \cite[Section 8.1]{CG2013}, we may rewrite the PDE \eqref{eq:ZDDS} as
\beq\label{eq:recast-ZDDS} A'(t) = \Delta A(t) + X(A(t)),\eeq
where $X(A(t))$ collects all the non-Laplacian terms. Note $X(A)$ has the symbolic form
\[  X(A) = A^3 + A \cdot \partial A,\]
where $A^3$ is a tuple of sums of terms of the form $A_i A_j A_k$ and $A \cdot \partial A$ is a tuple of sums of terms of the form $A_i (\ptl_k A_j)$ or $(\ptl_k A_j) A_i$, where $1 \leq i, j, k \leq 3$. As Charalambous and Gross do in \cite[Section 8.1]{CG2013}, we will extensively use this symbolic form in our arguments. As an example, if we have $A, \tilde{A}$, then upon adding and subtracting, we may write
\[ A \cdot \ptl A - \tilde{A} \cdot \ptl \tilde{A} = (A - \tilde{A}) \cdot \ptl A + \tilde{A} \cdot (\ptl  A - \ptl \tilde{A}).\]
The first term on the right hand side collects all terms like $(A_i - \tilde{A}_i) (\ptl_k A_j)$ or $(\ptl_k A_j)(A_i - \tilde{A}_i)$, and the second term collects all terms like $\tilde{A}_i (\ptl_k A_j - \ptl_k \tilde{A}_j)$ or $(\ptl_k A_j - \ptl_k \tilde{A}_j) \tilde{A}_i$.

Proposition \ref{prop:ZDDS-local-existence} is a small extension of \cite[Theorem 2.14, see also Theorem 8.3]{CG2013}. Charalambous and Gross's proof of \cite[Theorem 2.14]{CG2013} is by a contraction argument on the appropriate path space. In what follows, we will introduce this contraction argument, and then slightly extend/modify it to prove the results we need.

First, we may recast \eqref{eq:recast-ZDDS} into the following integral equation (here $A_0$ is the initial data)
\beq\label{eq:duhamel-ZDDS}A(t) = e^{t \Delta} A_0 + \int_0^t e^{(t-s) \Delta} X(A(s)) ds. \eeq
For $0 < T < \infty$, define the path space $\mc{P}_T$ to be the set of continuous functions $A : [0, T] \ra \honeconnspace$ such that for each $0 < t \leq T$, $\|A(t)\|_\infty$, $\|dA(t)\|_\infty$, $\|d^*A(t)\|_\infty < \infty$, and such that
\beq\label{eq:path-space-norm-def} \begin{split}
\|A\|_{\mc{P}_T} := \sup_{0 < t \leq T} \bigg\{ &\|A(t)\|_{H^1} + t^{1/4} \|A(t)\|_\infty + \\
&t^{3/4}\big(\|dA(t)\|_\infty + \|d^*A(t)\|_\infty \big)\bigg\} < \infty.
\end{split}\eeq
By slightly modifying the proof of the classic real analysis fact that the space $C([0, T], \R)$ is complete under the uniform norm, one can show that $(\mc{P}_T, \|\cdot\|_{\mc{P}_T})$ is indeed a Banach space.

The following lemma is proven in \cite[Section 8.1]{CG2013} (see in particular the proof of Theorem 2.14 and the end of the proof of Theorem 8.3 in that section).

\begin{lemma}\label{lemma:zdds-integral-equation-suffices}
Let $0 < T < \infty$. Suppose $A \in \mc{P}_T$ is such that \eqref{eq:duhamel-ZDDS} holds for all $0 \leq t \leq T$. Then $A$ is a solution to \eqref{eq:ZDDS} on $[0, T)$ with initial data $A(0) = A_0$. Moreover, $A \in C^\infty((0, T) \times \threetorus, \lalg^3)$.
\end{lemma}

Given $A_0 \in \honeconnspace$, $0 < T < \infty$, and $A \in \mc{P}_T$, define $W(A) : [0, T] \ra \honeconnspace$ by
\beq\label{eq:w-def} W(A)(t) := e^{t \Delta} A_0 + \int_0^t e^{(t - s) \Delta} X(A(s)) ds, ~~ t \in [0, T]. \eeq
In light of Lemma \ref{lemma:zdds-integral-equation-suffices}, the goal is now to show that $W$ is a contraction, when restricted to the right subspace of $\mc{P}_T$, for $T$ small enough depending only on $\|A_0\|_{H^1}$. By the contraction mapping theorem, this will then give a fixed point $W(A) = A$, i.e., a solution to the integral equation \eqref{eq:duhamel-ZDDS}. Moreover, the same argument that shows that $W$ is a contraction can be used to show the continuity of solutions to \eqref{eq:ZDDS} in the initial data.

We now state the following three lemmas from \cite[Section 8.1]{CG2013}, which establish the key estimates needed in the contraction argument.

\begin{lemma}[Remark 8.5 of \cite{CG2013}]
For any $1 \leq q \leq p \leq \infty$ and any $0 < t \leq 1$, we have
\begin{align} 
\|e^{t \Delta}\|_{q \ra p} &\leq \const t^{-(3/2)(1/q - 1/p)}, \label{eq:heat-semigroup-lq-lp-op-norm} \\
\|\ptl e^{t \Delta}\|_{q \ra p} &\leq \const t^{-1/2} t^{-(3/2)(1/q - 1/p)}, \text{ with $\ptl = d$ or $\ptl = d^*$}, \label{eq:heat-semigroup-derivative-lq-lp-op-norm} \\
\|e^{t\Delta}\|_{L^2 \ra H^1} &\leq \const t^{-1/2}. \label{eq:heat-semigroup-l2-h1-op-norm}
\end{align}
Here $\|\cdot\|_{q \ra p}$ denotes the $L^q \ra L^p$ operator norm, and $\|\cdot\|_{L^2 \ra H^1}$ denotes $L^2 \ra H^1$ operator norm.
\end{lemma}

\begin{lemma}[Lemma 8.6 of \cite{CG2013}] \label{lemma:linear-heat-equation-in-path-space}
There is a constant $\const_0$ such that the following holds. For any $0 < T \leq 1$ and $A_0 \in \honeconnspace$, we have that the path $A_1 : [0, T] \ra \honeconnspace$ defined by $A_1(t) := e^{t \Delta} A_0$ is such that $A_1 \in \mc{P}_T$, and moreover
\[ \|A_1\|_{\mc{P}_T} \leq \const_0 \|A_0\|_{H^1}.  \]
\end{lemma}

\begin{lemma}[Lemma 8.4 of \cite{CG2013}]\label{lemma:nonlinear-part-bound}
Let $0 < T < \infty$, $R \geq 0$, and $2 \leq q \leq \infty$. Let $A \in \mc{P}_T$. Suppose that $\|A\|_{\mc{P}_T} \leq R$. Then for all $0 < s < T$, we have
\[ \|X(A(s))\|_q \leq \const s^{-(3/2)(1/2 - 1/q)} (R^3 + s^{-1/4} R^2).\]
Moreover, for any $A_1, A_2 \in \mc{P}_T$ such that $\|A_1\|_{\mc{P}_T}, \|A_2\|_{\mc{P}_T} \leq R$, we have that for any $0 < s < T$,
\[ \|X(A_1(s)) - X(A_2(s))\|_q \leq \const s^{-(3/2)(1/2 - 1/q)} \|A_1 - A_2\|_{\mc{P}_T} (R^2 + s^{-1/4} R).\]
\end{lemma}

Given $A \in \mc{P}_T$, define $\rho(A) : [0, T] \ra \honeconnspace$ by
\beq\label{eq:rho-def-h1} \rho(A)(t) := \int_0^t e^{(t - s) \Delta} X(A(s)) ds, ~~ t \in [0, T].\eeq
We now use Lemma \ref{lemma:nonlinear-part-bound} to prove the following lemma, which is the key result needed to show that $W$ is a contraction, when restricted to the right space. Even though this lemma is essentially proven in \cite[Section 8.1]{CG2013}, we still give the proof here, because some of the results that we will prove later on can be taken care of by small modifications of the argument.

\begin{lemma}\label{lemma:duhamel-part-estimates}
Let $0 < T \leq 1$, $A \in \mc{P}_T$. Suppose for some $R$, we have $\|A\|_{\mc{P}_T} \leq R$. Write $\rho$ instead of $\rho(A)$ for brevity. Then we have
\[ \|\rho\|_{\mc{P}_T} \leq \const T^{1/4} (R^3 + R^2). \]
Additionally, suppose $A_1, A_2 \in \mc{P}_T$, and $\|A_i\|_{\mc{P}_T} \leq R$ for $i = 1,2$. Write $\rho_1, \rho_2$ instead of $\rho(A_1), \rho(A_2)$ for brevity. Then we have
\[ \|\rho_1 - \rho_2\|_{\mc{P}_T} \leq \const T^{1/4} (R^2 + R) \|A_1 - A_2\|_{\mc{P}_T} .\]
\end{lemma}
\begin{proof}
Throughout, we will repeatedly use the inequality 
\beq\label{eq:t-minus-sigma-sigma-integral} \int_0^t (t-s)^{-\beta} s^{-\gamma} ds = \const_{\beta, \gamma} t^{1 - \beta - \gamma}, ~~~ 0 \leq \beta, \gamma < 1, \eeq
where $\const_{\beta, \gamma}$ is some constant depending only on $\beta, \gamma$. This inequality can be seen by splitting $\int_0^t = \int_0^{t/2} + \int_{t/2}^t$.

For $t \in (0, T]$, we have
\[ \|\rho(t)\|_{H^1} \leq \int_0^t \|e^{(t - s) \Delta} X(A(s))\|_{H^1} d s. \]
By \eqref{eq:heat-semigroup-l2-h1-op-norm} and Lemma \ref{lemma:nonlinear-part-bound} with $q = 2$, we obtain the further bound
\[ \const \int_0^t (t-s)^{-1/2} (R^3 + s^{-1/4} R^2) ds \leq \const (t^{1/2} R^3 + t^{1/4} R^2). \]
Next, we bound $\|\rho(t)\|_\infty$. By \eqref{eq:heat-semigroup-lq-lp-op-norm} with $q = 6, p = \infty$ and Lemma \ref{lemma:nonlinear-part-bound} with $q = 6$, we have
\[ \|\rho(t)\|_\infty \leq \const \int_0^t (t - s)^{-1/4} s^{-1/2} (R^3 + s^{-1/4} R^2) ds. \]
Proceeding as before, we obtain
\[ t^{1/4} \|\rho(t)\|_\infty \leq \const (t^{1/2} R^3 + t^{1/4}R^2). \]
We finally bound $\| \ptl \rho(t)\|_\infty$ for $\ptl = d, d^*$. By \eqref{eq:heat-semigroup-derivative-lq-lp-op-norm} with $q = 6, p = \infty$, we have
\[ \|\ptl \rho(t)\|_\infty \leq \const \int_0^t (t-s)^{-3/4} \|X(A(s))\|_6 ds. \]
By Lemma \ref{lemma:nonlinear-part-bound} with $q = 6$, we obtain the further upper bound
\[ \const \int_0^t (t-s)^{-3/4} s^{-1/2}(R^3 + s^{-1/4} R^2), \]
and thus
\[ t^{3/4} \|\rho(t)\|_\infty \leq \const (t^{1/2} R^3 + t^{1/4} R^2).\]
Combining all the estimates and taking sup over $0 \leq t \leq T$ (and noting $T^{1/2} \leq T^{1/4}$ since $T \leq 1$) gives us the first claim. The second claim follows similarly.
\end{proof}

For $0 < T < \infty$ and $R \geq 0$, define the space
\[ \mc{P}_{T, R} := \{A \in \mc{P}_T : \|A\|_{\mc{P}_T} \leq R\}.\]

\begin{lemma}\label{lemma:w-is-a-contraction}
There is a continuous non-increasing function $\timefn : [0, \infty) \ra (0, 1]$ such that the following holds. Given $A_0 \in \honeconnspace$, let $T_0 := \timefn(\|A_0\|_{H^1})$, and let $R_0 := 2 \const_0 \|A_0\|_{H^1}$, where $\const_0$ is from Lemma \ref{lemma:linear-heat-equation-in-path-space}. Let $W$ be defined in terms of $A_0$ by \eqref{eq:w-def}. Then $W$ maps $\mc{P}_{T_0, R_0}$ into $\mc{P}_{T_0, R_0}$, and moreover, it is a $(1/2)$-contraction:
\[ \|W(A) - W(\tilde{A})\|_{\mc{P}_{T_0}} \leq \frac{1}{2} \|A - \tilde{A}\|_{\mc{P}_{T_0}} \text{ for all $A, \tilde{A} \in \mc{P}_{T_0, R_0}$.}\]
Additionally, let $\tilde{A}_0 \in \honeconnspace$, with $\tilde{W}$ defined in terms of $\tilde{A}_0$ by \eqref{eq:w-def}. Let $\tilde{T}_0 := \timefn(\|\tilde{A}_0\|_{H^1})$, $\tilde{R}_0 := 2 \const_0 \|\tilde{A}_0\|_{H^1}$. Let $T_1 := \min(T_0, \tilde{T}_0)$, $R_1 := \max(R_0, \tilde{R}_0)$. Let $A_1, \tilde{A}_1 : [0, T_1] \ra \honeconnspace$ be defined by $A_1(t) := e^{t \Delta}A_0$, $\tilde{A}_1(t) := e^{t \Delta} \tilde{A}_0$. Then for any $A, \tilde{A} \in \mc{P}_{T_1, R_1}$, we have
\[ \|W(A) - \tilde{W}(\tilde{A})\|_{\mc{P}_{T_1}} \leq \|A_1 - \tilde{A}_1\|_{\mc{P}_{T_1}} + \frac{1}{2} \|A - \tilde{A}\|_{\mc{P}_{T_1}}.\]
\end{lemma}
\begin{proof}
By Lemmas \ref{lemma:linear-heat-equation-in-path-space} and \ref{lemma:duhamel-part-estimates}, for any $0 < T \leq 1$, and any $A \in \mc{P}_{T, {R_0}}$, we have
\[ \|W(A)\|_{\mc{P}_T} \leq \const_0 \|A_0\|_{H^1} + \const T^{1/4} (R_0^3 + R_0^2).\]
Moreover, by Lemma \ref{lemma:duhamel-part-estimates}, for any $A, \tilde{A}\in \mc{P}_{T, R_1}$, we have
\[ \|W(A) - \tilde{W}(\tilde{A})\|_{\mc{P}_T} \leq  \|A_1 - \tilde{A}_1\|_{\mc{P}_T} + \const T^{1/4} (R_1^2 + R_1) \|A - \tilde{A}\|_{\mc{P}_T} .\]
We may thus define $\timefn(\|A_0\|_{H^1})$ as the largest $T \leq 1$ such that
\[ \const T^{1/4} (R_0^3 + R_0^2) \leq \const_0 \|A_0\|_{H^1} \text{ and } \const T^{1/4}(R_0^2 + R_0) \leq \frac{1}{2}.\]
Since $R_0 = 2 \const_0 \|A_0\|_{H^1}$, we see that $\timefn$ is in fact a function of $\|A_0\|_{H^1}$ (and only $\|A_0\|_{H^1}$). By taking $\tilde{A}_0 = A_0$, we see that $W$ is indeed as $(1/2)$-contraction.
\end{proof}

We can now prove Proposition \ref{prop:ZDDS-local-existence}.

\begin{proof}[Proof of Proposition \ref{prop:ZDDS-local-existence}]
Let $\timefn$ be as in Lemma \ref{lemma:w-is-a-contraction}. Given $A_0 \in \honeconnspace$, by Lemma \ref{lemma:w-is-a-contraction}, we have that $W$ is a contraction on $\mc{P}_{T_0, R_0}$, where $T_0 := \timefn(\|A_0\|_{H^1})$ and $R_0 := 2 \const_0 \|A_0\|_{H^1}$. Thus by the contraction mapping theorem (see, e.g., \cite[Theorem 2.1]{Tes2012}), there exists a unique fixed point $A \in \mc{P}_{T_0, R_0}$. By Lemma \ref{lemma:zdds-integral-equation-suffices}, we have that $A$ is a solution to \eqref{eq:ZDDS} on $[0, T_0)$ with initial data $A(0) = A_0$, and moreover $A \in C^\infty((0, T_0) \times \threetorus, \lalg^3)$. Uniqueness of solutions to \eqref{eq:ZDDS} is proven at the end of \cite[Section 8.1]{CG2013} --- see in particular the subsection titled ``Uniqueness for the parabolic equation".

We now show continuity in the initial data. Let $\{A_{0, n}\}_{n \leq \infty}$, $\{A_n\}_{n \leq \infty}$ be as in the proposition statement. For $n \leq \infty$, let $W_n$ be defined in terms of $A_n$ by \eqref{eq:w-def}, $T_n := \timefn(\|A_{0, n}\|_{H^1})$, $R_n := 2\const_0 \|A_{0, n}\|_{H^1}$. Additionally, for $n \leq \infty$, let $\tilde{T}_n := \min\{T_n, T_\infty\}$, and let $A_{1, n} : [0, T_n] \ra \honeconnspace$ be defined by $A_{1, n}(t) := e^{t\Delta} A_{0, n}$. Since $A_n \in \mc{P}_{T_n, R_n}$ for all $n \leq \infty$, we have upon applying Lemma \ref{lemma:w-is-a-contraction} (with $T_1 = \tilde{T}_n$, $R_1 = \max\{R_n, R_\infty\}$) that
\[ \|W_n(A_n) - W_\infty(A_\infty)\|_{\mc{P}_{\tilde{T}_n}} \leq \|A_{1, n} - A_{1, \infty}\|_{\mc{P}_{\tilde{T}_n}} + \frac{1}{2} \|A_n - A_\infty\|_{\mc{P}_{\tilde{T}_n}}.\]
By the fact that $W_n(A_n) = A_n$ for all $n \leq \infty$, Lemma \ref{lemma:linear-heat-equation-in-path-space}, and the fact that $\timefn$ is defined to be at most 1, we obtain
\[ \|A_n - A_\infty\|_{\mc{P}_{\tilde{T}_n}} \leq 2 \|A_{1, n} - A_{1, \infty}\|_{\mc{P}_{\tilde{T}_n}} \leq 2 \const_0 \|A_{0, n} - A_{0, \infty}\|_{H^1}. \]
By assumption, we have $\|A_{0, n} - A_{0, \infty}\|_{H^1} \ra 0$. By the continuity of $\timefn$, we have $\tilde{T}_n \ra T_\infty$. Thus for any $t_0 < T_\infty$, we have $\|A_n - A_\infty\|_{\mc{P}_{t_0}} \ra 0$. The desired result now follows.
\end{proof} 



We now proceed to prove Lemma \ref{lemma:zdds-solutions-weak-continuity}. The proof will be by a very similar argument as just given. For $0 < T < \infty$ and $A : [0, T] \ra \honeconnspace$, define the norm
\[ |A|_T := \sup_{0 \leq t \leq T} t^{1/2} \|A(t)\|_{H^1}.\]
Recall that given $A_0, \tilde{A}_0 \in \honeconnspace$, we have maps $W, \tilde{W}$ defined in terms of $A_0, \tilde{A}_0$ by \eqref{eq:w-def}, respectively. With $T_1 := \min\{\timefn(\|A_0\|_{H^1}), \timefn(\|\tilde{A}_0\|_{H^1})\}$ and $R_1 := 2\const_0 \max\{\|A_0\|_{H^1}, \|\tilde{A}_0\|_{H^1}\}$, we have that $W, \tilde{W}$ are $(1/2)$-contractions on the set $\mc{P}_{T_1, R_1}$. Thus we have fixed points $A, \tilde{A}$ of $W, \tilde{W}$, which are the solutions to \eqref{eq:ZDDS} on $[0, T_1)$ with initial data $A(0) = A_0, \tilde{A}(0) = \tilde{A}_0$. Moreover, we know that $A, \tilde{A} \in \mc{P}_{T_1, R_1}$, i.e., $\|A\|_{\mc{P}_{T_1}}, \|\tilde{A}\|_{\mc{P}_{T_1}} \leq R_1$.

\begin{lemma}\label{lemma:nonlinear-part-estimate-for-weak-convergence}
Let all notation be as in the previous paragraph. For any $T \leq T_1$, we have
\[ |A - \tilde{A}|_T \leq \const \|A_0 - \tilde{A}_0\|_2 + \const T^{1/4} (R_1^2 + R_1) |A - \tilde{A}|_T.\]
\end{lemma}
\begin{proof}
Since $T \leq T_1$, we have $\|A\|_{\mc{P}_T}, \|\tilde{A}\|_{\mc{P}_T} \leq R_1$, a fact which we will use several times. Let $A_1, \tilde{A}_1 : [0, T] \ra \honeconnspace$ be defined by $A_1(t) := e^{t \Delta} A_0$, $\tilde{A}_1(t) := e^{t \Delta} \tilde{A}_0$. Then $A = W(A) = A_1 + \rho(A)$, $\tilde{A} = \tilde{W}(\tilde{A}) = \tilde{A}_1 + \rho(\tilde{A})$, where $\rho$ is as defined in \eqref{eq:rho-def-h1}. Thus
\[ |A - \tilde{A}|_T \leq |A_1 - \tilde{A}_1|_T + |\rho(A) - \rho(\tilde{A})|_T.\]
The first term on the right hand side may be bounded by $\const \|A_0 - \tilde{A}_0\|_2$ by applying \eqref{eq:heat-semigroup-l2-h1-op-norm} (note $T_1 \leq 1$ by the definition of $\timefn$ in Lemma \ref{lemma:w-is-a-contraction}). It remains to bound the second term. For $0 < t \leq T$, we have
\[ \|\rho(A)(t) - \rho(\tilde{A})(t)\|_{H^1} \leq \int_0^t \|e^{(t - s) \Delta} (X(A(s)) - X(\tilde{A}(s)))\|_{H^1} ds. \]
To bound the integrand, first note that
\[ \begin{split}
\|e^{(t - s) \Delta} (X(A(s)) &- X(\tilde{A}(s)))\|_{H^1} \leq \|e^{(t- s) \Delta} (A(s)^3 - \tilde{A}(s)^3)\|_{H^1} + \\
&\|e^{(t - s) \Delta} (A(s) \cdot \ptl A(s) - \tilde{A}(s) \cdot \ptl \tilde{A}(s))\|_{H^1}
\end{split}\]
By \eqref{eq:heat-semigroup-l2-h1-op-norm}, we obtain the further upper bound on the first term of the right hand side
\[ \const (t-s)^{-1/2} \|A(s)^3 - \tilde{A}(s)^3\|_2. \]
To bound this, recall that the typical term is like
\[ \|A_i(s) A_j(s) A_k(s)- \tilde{A}_i(s) \tilde{A}_j(s) \tilde{A}_k(s)\|_2, \]
and so by introducing a telescoping sum, it suffices to bound a term like
\[ \|A_i(s) A_j(s) (A_k(s) - \tilde{A}_k(s))\|_2. \]
Thus by repeated applications of H\"{o}lder's inequality, we obtain
\[ \|A(s)^3 - \tilde{A}(s)^3\|_2 \leq \const \|A(s)\|_6^2 \|A(s) - \tilde{A}(s)\|_6. \]
Using the Sobolev embedding $H^1 \hookrightarrow L^6$ and the definition \eqref{eq:path-space-norm-def} of the norm $\|\cdot\|_{\mc{P}_T}$, we have $\|A(s)\|_6 \leq \const\|A\|_{\mc{P}_T} \leq R_1$.
We also have 
\[ \|A(s) - \tilde{A}(s)\|_6 \leq \const \|A(s) - \tilde{A}(s)\|_{H^1} \leq \const s^{-1/2} |A - \tilde{A}|_T. \]
We thus obtain (recalling \eqref{eq:t-minus-sigma-sigma-integral})
\[\begin{split}
t^{1/2}\int_0^t &\|e^{(t-s) \Delta} (A(s)^3 - \tilde{A}(s)^3)\|_{H^1} ds \leq \\
&\const R_1^2|A - \tilde{A}|_T t^{1/2}\int_0^t (t-s)^{-1/2} s^{-1/2} ds \leq \const t^{1/2} R_1^2 |A - \tilde{A}|_T.
\end{split} \]
We now bound the other term. First, we claim
\beq\begin{split}
\label{eq:heat-kernel-l3/2-h1-op-norm} \|e^{(t - s) \Delta} (A(s) \cdot &\ptl A(s) - \tilde{A}(s) \cdot \ptl \tilde{A}(s))\|_{H^1} \leq \\
& \const (t-s)^{-3/4} \|A(s) \cdot \ptl A(s) - \tilde{A}(s) \cdot \ptl \tilde{A}(s)\|_{3/2}.
\end{split}\eeq
Given this claim, we further have
\[\begin{split}
\|A(s) \cdot \ptl A(s) - \tilde{A}(s) \cdot \ptl \tilde{A}(s)\|_{3/2} \leq \|(A(s) &- \tilde{A}(s)) \cdot \ptl A(s)\|_{3/2} ~+ \\
&\|\tilde{A}(s) \cdot (\ptl A(s) - \ptl \tilde{A}(s))\|_{3/2}.
\end{split}\]
By H\"{o}lder's inequality and Theorem \ref{sobolevthm} (applied with $p = 6$), we have
\begin{align*}
\|(A(s) - \tilde{A}(s)) \cdot \ptl A(s)\|_{3/2} &\leq \const \|A(s) - \tilde{A}(s)\|_{6} \|A(s)\|_{H^1} \\
&\leq \const R_1 \|A(s) - \tilde{A}(s)\|_{H^1} \\
&\leq \const R_1  s^{-1/2} |A - \tilde{A}|_T.
\end{align*}
We may similarly obtain
\[\|\tilde{A}(s) \cdot (\ptl A(s) - \ptl \tilde{A}(s))\|_{3/2} \leq \const R_1 s^{-1/2} |A - \tilde{A}|_T. \]
We thus have (recalling \eqref{eq:t-minus-sigma-sigma-integral})
\[\begin{split}
t^{1/2} \int_0^t \|&e^{(t - s) \Delta} (A(s) \cdot \ptl A(s) - \tilde{A}(s) \cdot \ptl \tilde{A}(s))\|_{H^1} ds \leq \\
& \const R_1 |A - \tilde{A}|_T t^{1/2} \int_0^t (t-s)^{-3/4} s^{-1/2} ds \leq \const t^{1/4} R_1 |A - \tilde{A}|_T. 
\end{split}\]
Combining the previous estimates, and then taking sup over $0 \leq t \leq T$, we obtain
\[ |\rho(A) - \rho(\tilde{A})|_T \leq \const (T^{1/2} R_1^2 + T^{1/4} R_1) |A - \tilde{A}|_T . \]
Since by definition, $\timefn \leq 1$, it follows that $T \leq T_1 \leq 1$, and thus $T^{1/2} \leq T^{1/4}$. The desired result now follows, modulo the claim \eqref{eq:heat-kernel-l3/2-h1-op-norm}.

We now show the claim \eqref{eq:heat-kernel-l3/2-h1-op-norm}. For brevity, let 
\[ F := A(s) \cdot \ptl A(s) - \tilde{A}(s) \cdot \ptl \tilde{A}(s).\]
By \eqref{eq:heat-semigroup-l2-h1-op-norm} combined with \eqref{eq:heat-semigroup-lq-lp-op-norm} with $q = 3/2, p = 2$, we have
\begin{align*}
\|e^{(t-s)\Delta} F\|_{H^1} &= \|e^{((t-s)/2)\Delta} e^{((t-s)/2) \Delta} F\|_{H^1} \\
&\leq \|e^{((t-s)/2) \Delta}\|_{L^2 \ra H^1} \|e^{((t-s) / 2) \Delta}\|_{3/2 \ra 2} \|F\|_{3/2}\\
&\leq \const (t-s)^{-3/4} \|F\|_{3/2},
\end{align*}
as desired.
\end{proof}

We can now prove Lemma \ref{lemma:zdds-solutions-weak-continuity}.

\begin{proof}[Proof of Lemma \ref{lemma:zdds-solutions-weak-continuity}]
Let $M_0 > 0$ be such that $\sup_{n \leq \infty} \|A_{0, n}\|_{H^1} \leq M_0$. Let $T \leq \timefn(M_0)$ and $R := 2\const_0 M_0$ (where $\const_0$ is as in Lemma \ref{lemma:linear-heat-equation-in-path-space}). By the monotonicity of $\timefn$ and Lemma \ref{lemma:nonlinear-part-estimate-for-weak-convergence}, we have for any $n$,
\[ |A_n - A_\infty|_T \leq \const \|A_{0, n} - A_{0, \infty}\|_2 + \const T^{1/4} (R^2 + R) |A_n - A_\infty|_T.\]
Thus for $M_1 > 0$, let $\timefn_0(M_1)$ to be the largest $T \leq \timefn(M_1)/2$ such that
\[ \const T^{1/4} ((2 \const_0 M_1)^2 +  (2 \const_0 M_1)) \leq \frac{1}{2}.\]
Taking $M = \sup_{n \leq \infty} \|A_n\|_{H^1}$, we then obtain
\[ |A_n - A_\infty|_{\timefn_0(M)} \leq \const \|A_{0, n} - A_{0, \infty}\|_2. \]
(Note here we used the fact that $|A_n - A_\infty|_{\timefn_0(M)} < \infty$ for all $n$, since we have $\|A_n\|_{\mc{P}_{\timefn_0(M)}} < \infty$ for all $n \leq \infty$.) Since $A_{0, n} \weakarrow A_{0, \infty}$ implies $A_{0, n} \ra A_{0, \infty}$ in $L^2$ (by the Sobolev embedding $H^1 \hookrightarrow L^2$), the desired result now follows.
\end{proof}

Finally, we close this section off with the proof of Lemma \ref{lemma:ym-zdds-solution-gauge-transformation}. To do so, we need to describe how a solution to \eqref{eq:YM} is obtained from a solution to \eqref{eq:ZDDS}. The following is a subset of the content of \cite[Theorem 9.2 and Lemma 9.6]{CG2013}. Let $A_0 \in \honeconnspace$, and let $B$ be the solution to \eqref{eq:ZDDS} on some interval $[0, T)$ with initial data $B(0) = A_0$. For $0 < \varep < T$, let $\gt_\varep : [\varep, T) \times \threetorus \ra \liegroup$ be the solution to the following ODE on $[0, T)$:
\beq\label{eq:g-eps-ode} \ptl_t \gt_\varep(t) = d^* B(t) \gt_\varep(t), ~~ \varep \leq t < T, ~~\gt_\varep(\varep) \equiv \groupid. \eeq
(Technically, the above is a collection of ODEs, one for each $x \in \threetorus$.) For $\varep \leq t < T$, define $A_\varep(t) := \gaction{B(t)}{\gt_\varep(t)}$. Then $A_\varep \in C^\infty([\varep, T) \times \torus^3, \lalg^3)$ is a classical solution to \eqref{eq:YM} on $[\varep, T)$ with initial data $A_\varep(\varep) = B(\varep)$, and moreover, if $A$ is the solution to \eqref{eq:YM} on $[0, \infty)$ with initial data $A(0) = A_0$, then we have
\[ \lim_{\varep \downarrow 0} \sup_{\varep \leq t < T} \|A(t) - A_\varep(t)\|_{H^1} = 0. \]
As mentioned in \cite[Section 2.2]{CG2013}, we take a strictly positive time $\varep > 0$ and then send $\varep \downarrow 0$, to get around various difficulties caused by $d^* B(t)$ becoming more singular as $t\downarrow 0$ (in particular, this prevents us from taking $\varep = 0$ right away). 

\begin{proof}[Proof of Lemma \ref{lemma:ym-zdds-solution-gauge-transformation}]
For $0 < \varep < T$, let $\gt_\varep : [\varep, T) \times \threetorus \ra \liegroup$ be as defined by the ODE \eqref{eq:g-eps-ode}. Let $A_\varep(t) := \gaction{B(t)}{\gt_\varep(t)}$ for $\varep \leq t < T$. Now fix $0 < t < T$. Take a sequence $\varep_n \downarrow 0$. For notational brevity, relabel $A_n := A_{\varep_n}$ and $\gt_n := \gt_{\varep_n}$. As noted in the discussion just after \eqref{eq:g-eps-ode}, we have from \cite[Theorem 9.2]{CG2013} that $A_n(t) \ra A(t)$ (in $\honeconnspace$, i.e., in $H^1$ norm). In particular, this implies that
\[ \sup_{n \geq 1} \|\gaction{B(t)}{\gt_n(t)}\|_{H^1} < \infty. \]
Thus by Lemma \ref{lemma:sequence-and-gauge-transform-bounded}, there exists $\gt \in \htwogaugetransf$ and a subsequence $\{\gt_{n_k}(t)\}_{k \geq 1}$ such that $\|\gt_{n_k}(t) - \gt\|_\infty \ra 0$ and $\gt_{n_k}^{-1}(t) d\gt_{n_k}(t) \ra \gt^{-1} d\gt$ in $L^2$. This implies that $A_{n_k}(t) = \gaction{B(t)}{\gt_{n_k}(t)} \ra \gaction{B(t)}{\gt}$ in $L^2$. Since we also have $A_{n_k}(t) \ra A(t)$ in $L^2$, we obtain $A(t) = \gaction{B(t)}{\gt}$, as desired.
\end{proof}

\section{Proofs of technical lemmas for Wilson loops}\label{section:Wilson-proofs}

\begin{proof}[Proof of Lemma \ref{lemma:wilson-loop-continuous-dependence-on-conn-general-character}]
Take $\rho : \liegroup \ra \unitary(m)$ to be a unitary representation of $\liegroup$ such that $\character = \Tr ~\rho$ (and so $m = \character(\groupid)$). Let $h : [0, 1] \ra \liegroup$ be the solution to the ODE \eqref{eq:wilson-loop-ode}, and let $\tilde{h}$ be the solution to the same ODE but with $\tilde{A}$ in place of $A$. By definition, $W_{\wloop, \character}(A) = \character(h(1)) = \Tr(\rho(h(1)))$, and similarly for $W_{\wloop, \character}(\tilde{A})$. Define $u, \tilde{u} : [0, 1] \ra G$ by $u(t) := \rho(h(t))$, $\tilde{u}(t) := \rho(\tilde{h}(t))$. Observe that
\[ u'(t) = d\rho_{h(t)} (h'(t)) = d\rho_{h(t)} (h(t) A(\wloop(t)) \cdot \wloop'(t)), \]
where for $g \in \liegroup$, $d\rho_g : T_{g} \liegroup \ra T_{\rho(g)} \unitary(m)$ is the derivative of $\rho$ at $g$.
We have for any $g \in \liegroup$, $X \in \lalg$, 
\[ d\rho_g(gX) = \frac{d}{dt} \rho(ge^{tX}) \bigg|_{t = 0} = \rho(g) \frac{d}{dt} \rho(e^{tX}) \bigg|_{t = 0} = \rho(g) d\rho_{\groupid} (X).\]
For brevity, let $\Pi_\rho := d\rho_{\groupid}$. Note by definition, $\Pi_\rho : \lalg \ra T_{\rho(\groupid)} \unitary(m)$ is linear. Moreover, we will later use the fact that $\Pi_\rho$ maps into skew-Hermitian matrices, i.e., $(\Pi_\rho(X))^* = -\Pi_\rho(X)$ for all $X \in \lalg$. Let $A_{\Pi_\rho} := (\Pi_\rho \circ A_i, 1 \leq i \leq 3)$, and let $\tilde{A}_{\Pi_\rho}$ be defined analogously. We then obtain
\[ u'(t) = u(t) \Pi_\rho(A(\wloop(t)) \cdot \wloop'(t)) = u(t) A_{\Pi_\rho}(\wloop(t)) \cdot \wloop'(t), \]
where we have used the linearity of ${\Pi_\rho}$, and similarly
\[ \tilde{u}'(t) = \tilde{u}(t) \tilde{A}_{\Pi_\rho}(\wloop(t)) \cdot \wloop'(t). \]
We proceed to bound the Frobenius norm difference
\beq\label{eq:u-tildeu-frob-norm-diff} \|u(1) - \tilde{u}(1)\|^2 = 2 \Tr(\rho(\groupid)) - \Tr(u(1)^* \tilde{u}(1) + \tilde{u}(1)^* u(1)). \eeq
Let $A_{{\Pi_\rho}, \wloop}(t) := A_{\Pi_\rho}(\wloop(t)) \cdot \wloop'(t)$ for $0 \leq t \leq 1$, and let $\tilde{A}_{{\Pi_\rho}, \wloop}(t)$ be defined analogously. Let $v : [0, 1] \ra \liegroup$ be defined by $v(t) := u(t)^* \tilde{u}(t)$. Observe then
\[ v'(t) = A_{{\Pi_\rho}, \wloop}(t)^* v(t) + v(t) \tilde{A}_{{\Pi_\rho}, \wloop}(t), ~~ v(0) = \rho(\groupid).\]
Upon integrating and taking traces, we obtain
\begin{align*}
\Tr(v(1)) - \Tr(\rho(\groupid)) &= \int_0^1 \Tr\big(A_{{\Pi_\rho}, \wloop}(t)^* v(t) + v(t) \tilde{A}_{{\Pi_\rho}, \wloop}(t)\big) dt \\
&= \int_0^1 \Tr\big(v(t)(\tilde{A}_{{\Pi_\rho}, \wloop}(t) - A_{{\Pi_\rho}, \wloop}(t)\big) dt.
\end{align*}
In the second identity, we used the fact $(A_{{\Pi_\rho}, \wloop})^* = -A_{{\Pi_\rho}, \wloop}$, which follows from our previous observation that ${\Pi_\rho}$ maps into skew-Hermitian matrices. We may similarly obtain
\[ \Tr(v(1)^*) - \Tr(\rho(\groupid)) = \int_0^1 \Tr\big(v(t)^*(A_{{\Pi_\rho}, \wloop}(t) - \tilde{A}_{{\Pi_\rho}, \wloop}(t)\big) dt. \]
Thus upon adding the previous two displays, we obtain
\begin{align*}
&\Tr(v(1) + v(1)^*) - 2 \Tr(\rho(\groupid)) \\
&= \int_0^1 \Tr\big( (v(t) - v(t)^*) (\tilde{A}_{{\Pi_\rho}, \wloop}(t) - A_{{\Pi_\rho}, \wloop}(t))\big) dt.
\end{align*}
Recalling \eqref{eq:u-tildeu-frob-norm-diff}, we may thus obtain
\[ \|u(1) - \tilde{u}(1)\|^2 \leq 2 m^{1/2} \int_0^1 \|\tilde{A}_{{\Pi_\rho}, \wloop}(t) - A_{{\Pi_\rho}, \wloop}(t)\| dt. \]
Here we have used the Cauchy--Schwarz inequality $|\Tr(M_1^* M_2)| \leq \|M_1\| \|M_2\|$ for matrices $M_1, M_2$, as well as the fact that $v(t) \in \unitary(m)$ so that $\|v(t) - v(t)^*\| \leq \|v(t)\| + \|v(t)^*\| = 2 m^{1/2}$ for all $0 \leq t \leq 1$. 
Observe again by Cauchy--Schwarz that
\[ |W_{\wloop, \character}(A) - W_{\wloop, \character}(\tilde{A})|^2 = |\Tr(u(1) - \tilde{u}(1))|^2 \leq m \|u(1) - \tilde{u}(1)\|^2, \]
and thus
\[ |W_{\wloop, \character}(A) - W_{\wloop, \character}(\tilde{A})|^2 \leq 2m^{3/2} \int_0^1 \|\tilde{A}_{{\Pi_\rho}, \wloop}(t) - A_{{\Pi_\rho}, \wloop}(t)\| dt. \]
To finish, note we have
\begin{align*}
\|\tilde{A}_{{\Pi_\rho}, \wloop}(t) - A_{{\Pi_\rho}, \wloop}(t)\| &= \|(\tilde{A}_{\Pi_\rho}(\wloop(t)) - A_{\Pi_\rho}(\wloop(t))) \cdot \wloop'(t)\| \\
&\leq \|\tilde{A}_{\Pi_\rho}(\wloop(t)) - A_{\Pi_\rho}(\wloop(t))\| \cdot |\wloop'(t)|.
\end{align*}
As ${\Pi_\rho}$ is a linear map between finite-dimensional normed vector spaces, we have that $\|{\Pi_\rho}\|_{op} := \sup_{X \in \lalg, X \neq 0} \|{\Pi_\rho}(X)\| / \|X\| < \infty$, where as usual both matrix norms are the Frobenius norm. We thus obtain
\begin{align*}
\|\tilde{A}_{\Pi_\rho}(\wloop(t)) - A_{\Pi_\rho}(\wloop(t))\|
&\leq \|{\Pi_\rho}\|_{op} \|\tilde{A}(\wloop(t)) - A(\wloop(t))\|.
\end{align*}
The desired result now follows upon combining the previous few displays.
\end{proof}

\begin{proof}[Proof of Lemma \ref{lemma:wilson-loop-continuous-dependence-on-loop-general-character}]
By essentially repeating the same steps as in the proof of Lemma \ref{lemma:wilson-loop-continuous-dependence-on-conn-general-character}, we can obtain
\[ |W_{\wloop_1, \character}(A) - W_{\wloop_2, \character}(A)|^2\leq 2 m^{3/2} \int_0^1 \|A_{{\Pi_\rho}, \wloop_1}(t) - A_{{\Pi_\rho}, \wloop_2}(t)\| dt,\]
where $m, {\Pi_\rho}$ are as in the proof, $A_{{\Pi_\rho}, \wloop_j}(t) = A_{{\Pi_\rho}}(\wloop_j(t)) \cdot \wloop_j'(t)$ for $j = 1,2$, and $A_{\Pi_\rho} = ({\Pi_\rho} \circ A_i, 1 \leq i \leq 3)$. We can then bound for $t \in [0, 1]$
\[\begin{split}
\|A_{{\Pi_\rho}, \wloop_1}(t) - A_{{\Pi_\rho}, \wloop_2}(t)\|  \leq \|(A_{\Pi_\rho}(\wloop_1(t)) &- A_{\Pi_\rho}(\wloop_2(t))) \cdot \wloop_1'(t)\| ~+  \\
&\|A_{{\Pi_\rho}}(\wloop_2(t)) \cdot (\wloop_1'(t) - \wloop_2'(t))\|.
\end{split}\]
Then by arguing as at the end of the proof of Lemma \ref{lemma:wilson-loop-continuous-dependence-on-conn-general-character}, we may obtain the further upper bound on the right hand side
\[ \|{\Pi_\rho}\|_{op} \big(\|A(\wloop_1(t)) - A(\wloop_2(t))\| \cdot |\wloop_1'(t)| + \|A(\wloop_2(t))\| \cdot |\wloop_1'(t) - \wloop_2'(t)|\big). \]
The desired result now follows by combining the previous few displays.
\end{proof}

\begin{proof}[Proof of Lemma \ref{lemma:indexset-separable}]
Let $\Pi : \R^3 \ra \threetorus$ be the canonical projection map. Given a piecewise $C^1$ loop $\wloop : [0, 1] \ra \threetorus$, and a point $x_0 \in \R^3$ such that $\Pi(x_0) = \wloop(0)$, we can define the path $\tilde{\wloop} : [0, 1] \ra \R^3$ by
\[ \tilde{\wloop}(t) := x_0 + \int_0^t \wloop'(s) ds, ~~ 0 \leq t \leq 1.\]
Note that $\tilde{\wloop}$ is a lift of $\wloop$, in the sense that $\wloop = \Pi \circ \tilde{\wloop}$. Moreover, given a piecewise $C^1$ path $\gamma : [0, 1] \ra \R^3$ (the definition of piecewise $C^1$ paths in $\R^3$ is exactly the same as for paths in $\threetorus$), the path $\Pi \circ \gamma : [0, 1] \ra \threetorus$ is a piecewise $C^1$ path such that $(\Pi \circ \gamma)'(t) = \gamma'(t)$ whenever the derivatives exist. Thus it suffices to approximate paths in $\R^3$ rather than loops in $\threetorus$. 

Let $\tilde{\loopset}_0$ be the collection of all piecewise $C^1$ paths $\gamma : [0, 1] \ra \R^3$ with the following properties: (1) $\gamma(0) \in \Q^3$, (2) $(\Pi \circ \gamma)(0) = (\Pi \circ \gamma)(1)$, and (3) there exists some $n \geq 1$ such that for all $1 \leq i \leq n$, $\gamma'$ exists, is constant, and is in $\Q^3$ on the interval $((i-1)/n, i/n)$. (So $\tilde{\loopset}_0$ is essentially made of piecewise linear paths with some additional properties.) Note $\tilde{\loopset}_0$ is countable. Now for any piecewise $C^1$ path $\gamma : [0, 1] \ra \R^3$ such that $(\Pi \circ \gamma)(0) = (\Pi \circ \gamma)(1)$, there exists a sequence $\{\gamma_n\}_{n \geq 1} \sse \tilde{\loopset}_0$ such that $\gamma_n(0) \ra \gamma(0)$ and $\int_0^1 |\gamma_n'(t) - \gamma'(t)| dt \ra 0$ (we give a construction at the end of the proof). Now letting $\wloop := \Pi \circ \gamma$ and $\wloop_n := \Pi \circ \gamma_n$ for $n \geq 1$, we have
\[ \int_0^1 |\wloop_n'(t) - \wloop'(t)| dt = \int_0^1 |\gamma_n'(t) - \gamma'(t)| dt \ra 0, \]
and also
\[ \sup_{0 \leq t \leq 1} \metricthreetorus(\wloop_n(t), \wloop(t)) \leq |\gamma_n(0) - \gamma(0)| + \int_0^1|\gamma_n'(t) - \gamma'(t)| dt \ra 0. \]
It follows that $d_\loopset(\wloop_n, \wloop) \ra 0$. We thus see that $\loopset_0 := \{\Pi \circ \gamma : \gamma \in \tilde{\loopset}_0\}$ is a countable dense subset of $(\loopset, d_\loopset)$, as desired.

As promised, we now  construct the sequence $\{\gamma_n\}_{n \geq 1}$. Start by taking $\gamma_n(0) \ra \gamma(0)$, with $\gamma_n(0) \in \Q^3$ for all $n \geq 1$. Since $\gamma$ is piecewise $C^1$, there is some finite set of times $0 = t_0 < t_1 < \cdots < t_k = 1$, such that $\gamma'$ is uniformly continuous on $[t_{i-1}, t_i]$ for all $1 \leq i \leq k$. Let $n \geq 1$. Define $\gamma_n' := \sum_{i=1}^n c_{n, i} \ind_{((i-1)/n, i/n)}$, where $\{c_{n, i}\}_{1 \leq i \leq n}$ is defined as follows. For $1 \leq i \leq n -1$, let $1 \leq j \leq k$ be such that $((i-1)/n, i/n) \cap (t_{j-1}, t_j) \neq \varnothing$. Take arbitrary $s \in (t_{j-1}, t_j)$, and let $c_{n, i} \in \Q^3$  be such that $|\gamma'(s) - c_{n, i}| \leq 1 / n$ (say). Having defined $c_{n, i}$, $1 \leq i \leq n-1$, we now define $c_{n, n}$ so that $\int_0^1 \gamma_n'(t) dt \in \Z^3$ (this ensures that $(\Pi \circ \gamma_n)(1) = (\Pi \circ \gamma_n)(0))$. Towards this end, let $m \in \Z^3$ be such that $\gamma(1) = \gamma(0) + m$. Then, define $c_{n, n}$ by (note that the following is in $\Q^3$)
\[ c_{n, n} := nm - \sum_{i=1}^{n-1} c_{n, i}, \]
so that
\[ \int_0^1 \gamma_n'(t) dt = \sum_{i=1}^n \frac{1}{n} c_{n, i} = m \in \Z^3. \]
To finish, we show that $\int_0^1 |\gamma_n'(t) - \gamma'(t)| dt \ra 0$. By the definition of $\gamma_n'$, and the fact that $\gamma'$ is uniformly continuous on $(t_{j-1}, t_j)$ for all $1 \leq j \leq k$, we have that $\int_0^{1-1/n} |\gamma_n'(t) - \gamma'(t)| dt \ra 0$. It remains to show that $\int_{1-1/n}^1 |\gamma_n'(t) - \gamma'(t)| dt \ra 0$. Since $\gamma'$ is uniformly continuous on  $(t_{k-1}, 1)$, we have that $\gamma'$ is bounded on $(t_{k-1}, 1)$, and thus $\int_{1-1/n}^1 |\gamma'(t)| dt \ra 0$. It remains to show that $\int_{1-1/n}^1 |\gamma_n'(t)| dt \ra 0$ as well. Note that
\[ \int_{1-1/n}^1  |\gamma_n'(t)| dt = \frac{1}{n} |c_{n, n}| = \bigg|m  - \frac{1}{n} \sum_{i=1}^{n-1} c_{n, i}\bigg| = \bigg|m - \int_0^{1-1/n} \gamma_n'(t) dt \bigg|. \]
Using that $\int_0^{1-1/n} |\gamma_n'(t) - \gamma'(t)|  dt \ra 0$, we see that it suffices to show that
\[\bigg| m - \int_0^{1-1/n} \gamma'(t) dt \bigg| \ra 0. \]
This follows since $\int_0^1 \gamma'(t) dt =  m$ by assumption, and $\int_{1-1/n}^1 |\gamma'(t)| dt \ra 0$ (as previously noted).
\end{proof} 

\section{Measurability results}\label{section:measurability}

\begin{lemma}\label{lemma:limits-preserve-measurability}
Let $(\Omega, \mc{F})$ be a measure space, and let $(S, \ms{T}_S)$ be a Lusin space. Let $\{X_n\}_{n \geq 1}$ be a sequence of measurable functions $X_n : (\Omega, \mc{F}) \ra (S, \mc{B}(S))$, and suppose that there is some function $X : \Omega \ra S$ such that $X_n$ converges pointwise to $X$. Then $X : (\Omega, \mc{F}) \ra (S, \mc{B}(S))$ is measurable.
\end{lemma}
\begin{proof}
By the definition of Lusin space (see, e.g., \cite[(82.1) Definition]{RW1994}), there exists a compact metric space $(J, \rho)$, a Borel subset $B \sse J$, and a homeomorphism $f : (S, \ms{T}_S) \ra (B, \ms{T}_B)$, where $\ms{T}_B$ is the subset topology on $B$, or equivalently, the metric topology generated by the restriction of $\rho$ to $B$. Thus, it suffices to show that $f(X) : (\Omega, \mc{F}) \ra (B, \mc{B}(B))$ is measurable. Since $(J, \rho)$ is a compact metric space, it follows that the Borel $\sigma$-algebra $\mc{B}(B)$ of $B$ is generated by open balls $U(x, r) := \{y \in B : \rho(x, y) < r\}$ for $x \in B, r > 0$. Thus fix $x \in B, r > 0$. It suffices to show that $\{f(X) \in U(x, r)\} \in \mc{F}$. This follows since $f(X_n)$ is measurable for all $n \geq 1$, combined with the fact that $f(X_n) \ra f(X)$ pointwise, which implies
\[ \{f(X) \in U(x, r)\} = \bigcup_{k=1}^\infty \bigcup_{m=1}^\infty \bigcap_{n = m}^\infty \{f(X_n) \in U(x, r - 1/k)\}. \qedhere\]
\end{proof}

In the following, recall the measurable maps $i^{1, 2} : (\nonlineardistspace, \mc{B}_{\nonlineardistspace}) \ra (\nonlineardistspace^{1, 2}, \borelnonlineardistspace)$ and $j^{1,2} : (\nonlineardistspace^{1, 2}, \borelnonlineardistspace) \ra (\nonlineardistspace, \mc{B}_{\nonlineardistspace})$ from Definitions \ref{def:i12} and \ref{def:j12}. Recall also $\ymsemigroup_t^{1, 2}$ from Definition \ref{def:ym-semigroup} and $\iota_{\ymsemigroup^{1,2}}$ from Definition \ref{def:embedding-orbit-space-nonlinear-dist-space}.

\begin{proof}[Proof of Lemma \ref{lemma:measurability-from-stochastic-process}]
Define $\tilde{\rvnldistspace} := i^{1,2}(\rvnldistspace)$. Then since $\rvnldistspace = j^{1,2}(\tilde{\rvnldistspace})$ (recall Lemma \ref{lemma:j12-i12-identity}), it suffices to show that $\tilde{\rvnldistspace}$ is measurable. Observe that $\tilde{\rvnldistspace}(t) = \Phi_t^{1,2}([\rconn]^{1,2})$ for $t > 0$, and thus $\tilde{\rvnldistspace} = \iota_{\ymsemigroup^{1,2}}([\rconn]^{1,2})$. Thus, it suffices to show that $[\rconn]^{1,2}$ is an $(\honeorbitspace, \borelorbitspace)$-valued random variable. Towards this end, for $N \geq 0$, define $\rconn_N \in \connspace^{1,2}$ by
\[ \rconn_N(x) := \sum_{\substack{n \in \Z^3 \\ |n|_\infty \leq N}} \hat{\rconn}(n) e_n(x). \]
Since $\rconn$ is smooth, we have that $\rconn_N \ra \rconn$ in $C^1$, and thus also in $H^1$ norm. This implies that $[\rconn_N]^{1,2} \ra [\rconn]^{1,2}$ (in the topology $\toporbitspace$). Since $(\honeorbitspace, \toporbitspace)$ is a Lusin space (Corollary \ref{cor:orbitspace-lusin}) and limits preserve measurability (Lemma \ref{lemma:limits-preserve-measurability}), it suffices to show that for all $N \geq 0$, $[\rconn_N]^{1,2}$ is an $\honeorbitspace$-valued random variable. To show this, it suffices to show that $\rconn_N$ is an $(\connspace^{1,2}, \mc{B}(\ms{T}_{sw}))$-valued random variable (recall the definition of the sequential weak topology $\ms{T}_{sw}$ from Definition \ref{def:sequential-weak-toplogy}).

Towards this end, let $N \geq 0$. Observe that since $\rconn$ is smooth, the Fourier coefficients $(\hat{\rconn}(n), |n|_\infty \leq N)$ form a stochastic process. Also, note that $\hat{\rconn}(-n) = \ovl{\hat{\rconn}(n)}$, and so $\hat{\rconn}(0) \in \R$. Let $I_N \sse \{n \in \Z^3 \setminus \{0\} : |n|_\infty \leq N\}$ be such that for all $n \in \Z^3 \setminus \{0\}$ with $|n|_\infty \leq N$, exactly one of $n, -n$ is in $I_N$. Then by our preceding observations, we have that there is some map $f_N : \R \times \C^{I_N} \ra \connspace^{1,2}$, such that $\rconn_N = f_N((\hat{\rconn}(n), n \in \{0\} \cup I_N))$. Moreover, this map $f_N$ is continuous (even when $\connspace^{1,2}$ is equipped with the stronger $H^1$ norm topology), and thus measurable. The measurability of $\rconn_N$ now follows.
\end{proof}

\begin{proof}[Proof of Lemma \ref{lemma:gauge-invariant-free-field-measurable}]
Define $\tilde{\rvnldistspace} := i^{1,2}(\rvnldistspace)$. Then since $\rvnldistspace = j^{1,2}(\tilde{\rvnldistspace})$ (recall Lemma \ref{lemma:j12-i12-identity}), it suffices to show that $\tilde{\rvnldistspace}$ is measurable. For $N \geq 2$, define $\tilde{\rvnldistspace}_N \in \nonlineardistspace^{1,2}$ by $\tilde{\rvnldistspace}_N(t) := \Phi_t^{1,2}([\rconn(T/N)]^{1,2})$ for $t > 0$. Observe then that $\tilde{\rvnldistspace}_N(t) = \tilde{\rvnldistspace}(t + T/N) = \Phi_{T/N}^{1,2}(\tilde{\rvnldistspace}(t))$ for all $t > 0$, $N \geq 2$. Then by Lemma \ref{lemma:ym-semigroup-continuous-in-time}, we obtain that for any $t > 0$, $\tilde{\rvnldistspace}_N(t) \ra \tilde{\rvnldistspace}(t)$, and thus $\tilde{\rvnldistspace}_N \ra \tilde{\rvnldistspace}$. Since $\nonlineardistspace^{1,2}$ is a Lusin space (Lemma \ref{lemma:nldistspace-lusin}) and limits preserve measurability (Lemma \ref{lemma:limits-preserve-measurability}), it thus suffices to show that for all $N \geq 2$, we have that $\tilde{\rvnldistspace}_N$ is measurable.

Towards this end, fix $N \geq 2$. Define $\rvnldistspace_N \in \nonlineardistspace$ by $\rvnldistspace_N(t) := \ymsemigroup_t([\rconn(T/N)])$ for $t > 0$. Note that $\tilde{\rvnldistspace}_N = i^{1,2}(\rvnldistspace_N)$. Thus by Lemma \ref{lemma:measurability-from-stochastic-process}, it suffices to show that $(\rconn(T/N, x), x \in \threetorus)$ is a stochastic process (note it has smooth sample paths by Theorem \ref{thm:ym-heat-flow-gff}), i.e., it suffices to show that for all $x \in \threetorus$, $\rconn(T/N, x)$ is a $\lalg^3$-valued random variable. This follows since $T \in (0, 1]$ is a random variable and $\rconn \in C^\infty((0, T) \times \threetorus, \lalg^3)$, so that we may write $\rconn(T/N, x)$ as a limit of discrete approximations
\[ \rconn(T/N, x) = \lim_{k \toinf} \sum_{i=1}^k \ind(T \in ((i-1)/k, i/k]) \rconn(i/(kN), x). \qedhere\]
\end{proof}

\bibliographystyle{plainnat}

\end{document}